\documentclass[10pt,a4paper]{article}

\usepackage{amsfonts}           
\usepackage{amsmath}            
\usepackage{amsthm}             
\usepackage{fancyvrb}           
\usepackage{float}              
\usepackage[
    a4paper
]{geometry}                     
\usepackage{harpoon}            
\usepackage[
    colorlinks,
    citecolor=blue!70!black,
    linkcolor=blue!70!black,
    urlcolor=blue!70!black
]{hyperref}                     
\usepackage{mathtools}          
\usepackage{scalerel}           
\usepackage{tikz}               
    \usetikzlibrary{patterns}       
\usepackage{xcolor}             

\author{R\'emy Sigrist}
\date{\today}
\title{Nonperiodic tilings
       related to Stern's diatomic series
       and based on tiles decorated with elements of $\mathbb{F}_p$}

\hypersetup {
    pdfauthor={Remy Sigrist},
    pdfsubject={Nonperiodic tilings
                related to Stern's diatomic series},
    pdftitle={Nonperiodic tilings
              related to Stern's diatomic series
              and based on tiles decorated with elements of Fp},
    pdfkeywords={Stern,nonperiodic,tilings}
}

\newtheorem{theo}{Theorem}[section]
\newtheorem{lemma}[theo]{Lemma}
\newtheorem{corrollary}[theo]{Corollary}
\newtheorem{definition}[theo]{Definition}
\newtheorem{remark}[theo]{Remark}

\newcommand{\quotes}[1]{``#1''}
\newcommand{\twodots}{\mathinner {\ldotp \ldotp}}
\newcommand{\seqnum}[1]{\href{https://oeis.org/#1}{#1}}

\newcommand{\hexgrid}[3]{
    \draw[dotted,gray] (#1-#3,#2+#3) -- (#1+#3,#2+#3);

    \draw[dotted,gray] (#1,#2) -- (#1,#2+2*#3);

    \draw[dotted,gray] (#1+#3,#2) -- (#1-#3,#2+2*#3);

    \foreach \i in {1,...,#3} {
        \draw[dotted,gray] (#1-#3,#2+#3+\i) -- (#1+#3-\i,#2+#3+\i);
        \draw[dotted,gray] (#1-#3+\i,#2+#3-\i) -- (#1+#3,#2+#3-\i);

        \draw[dotted,gray] (#1+\i,#2) -- (#1+\i,#2+2*#3-\i);
        \draw[dotted,gray] (#1-\i,#2+\i) -- (#1-\i,#2+2*#3);

        \draw[dotted,gray] (#1+#3-\i,#2) -- (#1-#3,#2+2*#3-\i);
        \draw[dotted,gray] (#1+#3,#2+\i) -- (#1-#3+\i,#2+2*#3);
    }
}

\newcommand{\hexOfZeros}[3]{
    \foreach \i in {1,...,#3} {
        \node at (#1+\i,#2) {$0$};
        \node at (#1+#3,#2+\i) {$0$};
        \node at (#1+#3-\i,#2+#3+\i) {$0$};
        \node at (#1-\i,#2+2*#3) {$0$};
        \node at (#1-#3,#2+2*#3-\i) {$0$};
        \node at (#1-#3+\i,#2+#3-\i) {$0$};
    }
}

\begin{document}

\maketitle

\begin{abstract}

This paper presents the construction and various properties
of substitution tilings
related to Stern's diatomic series
and based on tiles
decorated with elements of $\mathbb{F}_p$ for some odd prime number $p$.
These substitution tilings are additive in a sense that will be clarified later
and lead to new nonperiodic tilings
in one and two dimensions.

\end{abstract}

\begin{figure}[b]
    \begin{center}
    \vstretch{1.732}{
        \includegraphics[height=3cm, keepaspectratio]{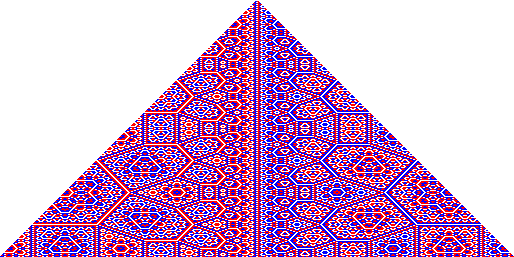}
    }
    \end{center}
    \label{fig:front}
\end{figure}

\clearpage

\newpage

\setcounter{tocdepth}{2}
\tableofcontents

\newpage

\section{Introduction}

Substitution tilings form a fascinating area of mathematics
at the confluence of several lines of research.
For an introduction to the theory of substitution tilings,
see Priebe~\cite{priebe}. Also, the Tilings Encyclopedia~\cite{tilings}
provides a wealth of examples of nonperiodic substitution tilings.

\vspace{\baselineskip}

In this paper, we will present
a two-dimensional substitution $\sigma$
and a one-dimensional substitution $\tau$
with construction rules similar to that of Stern's diatomic series,
based on tiles decorated with elements of $\mathbb{F}_p$
for some odd prime number $p$.
We will describe
the construction, various properties, including additivity and primitivity,
and tilings for both substitutions.
These tilings will be shown to be automatic, self-similar and nonperiodic.

\vspace{\baselineskip}

The usage of finite fields
in the framework of aperiodic tilings
was previously explored by Prunescu~\cite{prunescu}.

\begin{definition}
Throughout this document, and unless defined explicitly,
$p$ will denote an arbitrary odd prime number,
$\mathbb{F}_p$ will correspond to the finite field of order $p$,
and $\mathbb{F}_p^* = \mathbb{F}_p \setminus \{0\}$.
\end{definition}

\section{Stern's diatomic series}
\label{sec:stern}

\begin{definition}
Stern's diatomic series (see Lehmer~\cite{lehmer})
can be computed by the following procedure
(see figure~\ref{fig:stern}):
\begin{itemize}
\item start with an array of two values: $0$ and $1$,
\item repeatedly: insert between each pair of adjacent values,
      say $x$ and $y$, the value $x+y$,
\item the limiting array is Stern's diatomic series.
\end{itemize}
\end{definition}

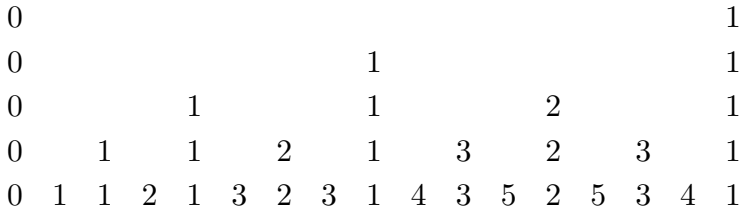
\begin{figure}[H]
    \centering
    \resizebox{10cm}{!}
    {
    \begin{tikzpicture}
                \node[scale=4,below] at (0,8) {0};
        \node[scale=4,below] at (32,8) {1};
        \node[scale=4,below] at (0,6) {0};
        \node[scale=4,below] at (16,6) {1};
        \node[scale=4,below] at (32,6) {1};
        \node[scale=4,below] at (0,4) {0};
        \node[scale=4,below] at (8,4) {1};
        \node[scale=4,below] at (16,4) {1};
        \node[scale=4,below] at (24,4) {2};
        \node[scale=4,below] at (32,4) {1};
        \node[scale=4,below] at (0,2) {0};
        \node[scale=4,below] at (4,2) {1};
        \node[scale=4,below] at (8,2) {1};
        \node[scale=4,below] at (12,2) {2};
        \node[scale=4,below] at (16,2) {1};
        \node[scale=4,below] at (20,2) {3};
        \node[scale=4,below] at (24,2) {2};
        \node[scale=4,below] at (28,2) {3};
        \node[scale=4,below] at (32,2) {1};
        \node[scale=4,below] at (0,0) {0};
        \node[scale=4,below] at (2,0) {1};
        \node[scale=4,below] at (4,0) {1};
        \node[scale=4,below] at (6,0) {2};
        \node[scale=4,below] at (8,0) {1};
        \node[scale=4,below] at (10,0) {3};
        \node[scale=4,below] at (12,0) {2};
        \node[scale=4,below] at (14,0) {3};
        \node[scale=4,below] at (16,0) {1};
        \node[scale=4,below] at (18,0) {4};
        \node[scale=4,below] at (20,0) {3};
        \node[scale=4,below] at (22,0) {5};
        \node[scale=4,below] at (24,0) {2};
        \node[scale=4,below] at (26,0) {5};
        \node[scale=4,below] at (28,0) {3};
        \node[scale=4,below] at (30,0) {4};
        \node[scale=4,below] at (32,0) {1};
    \end{tikzpicture}
    }
    \caption{The construction of Stern's diatomic series}
    \label{fig:stern}
\end{figure}

\begin{remark}
The substitutions $\sigma$ and $\tau$ will use a similar construction rule:
\begin{itemize}
\item start with \quotes{sites} decorated with \quotes{values}
      (in our context: taken from $\mathbb{F}_p$),
\item repeatedly: insert between each pair of \quotes{adjacent} sites,
      say decorated with values $x$ and $y$,
      a new site decorated with the value $x+y$.
\end{itemize}
\end{remark}

\begin{definition}
Stern's diatomic series is also known as Dijkstra's \quotes{fusc} function (see OEIS~\cite{oeis} sequence \seqnum{A002487}).
The \quotes{fusc} function satisfies:
\begin{equation}
    fusc(0) = 0
\end{equation}
\begin{equation}
    fusc(1) = 1
\end{equation}
\begin{equation}
    fusc(2n) = fusc(n)
\end{equation}
\begin{equation}
    fusc(2n+1) = fusc(n) + fusc(n+1)
\end{equation}
\end{definition}

\section{\texorpdfstring{Substitution $\sigma$}{Substitution sigma}}
\label{sec:sigma-cons}

This section presents the construction,
associated notations
and various properties of the substitution $\sigma$.

These notations and properties will allow us to build tilings of the plane
in sections~\ref{sec:sigma-tiling-s} and \ref{sec:sigma-tiling-h},
and provide tools to prove that these tilings are automatic,
self-similar and nonperiodic.

The substitution $\sigma$ is related to the OEIS~\cite{oeis} sequence \seqnum{A355855}.

\subsection{Tiles}
\label{sec:sigma-cons-tiles}

\begin{definition}
The tiles on which the substitution $\sigma$ operates
correspond to equilateral triangles
with unit side length
whose corners are decorated with elements of $\mathbb{F}_p$
and oriented either upwards or downwards
as depicted in figure~\ref{fig:tile}:
\begin{itemize}
\item the upward tile whose bottom left, bottom right and top corners
are respectively decorated with the values $x$, $y$ and $z$
will be denoted by $\bigtriangleup(x,y,z)$ (or $\bigtriangleup xyz$ when there is no ambiguity),
\item the downward tile whose top right, top left and bottom corners
are respectively decorated with the values $x$, $y$ and $z$
will be denoted by $\bigtriangledown(x,y,z)$ (or $\bigtriangledown xyz$ when there is no ambiguity).
\end{itemize}
\end{definition}

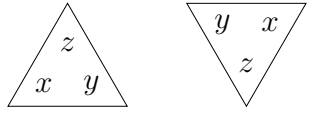
\begin{figure}[H]
    \centering
    \resizebox{4cm}{!}
    {
    \begin{tikzpicture}
    	\begin{scope}[yscale=.87,xslant=.5]
        \draw (0,0) -- (5,0) -- (0,5) -- cycle;
        \node[scale=1.5] at (1,1) {\Huge $x$};
        \node[scale=1.5] at (3,1) {\Huge $y$};
        \node[scale=1.5] at (1,3) {\Huge $z$};
        \draw (10,5) -- (5,5) -- (10,0) -- cycle;
        \node[scale=1.5] at (9,4) {\Huge $x$};
        \node[scale=1.5] at (7,4) {\Huge $y$};
        \node[scale=1.5] at (9,2) {\Huge $z$};
    	\end{scope}
    \end{tikzpicture}
    }
    \caption{The upward tile $\bigtriangleup xyz$
             and the downward tile $\bigtriangledown xyz$}
    \label{fig:tile}
\end{figure}

\begin{definition}
The sets of upward, downward and all tiles will be denoted, respectively,
by $U$, $D$ and $T$:
\begin{equation}
    U = \{ \bigtriangleup xyz \mid x, y, z \in \mathbb{F}_p \}
\end{equation}
\begin{equation}
    D = \{ \bigtriangledown xyz \mid x, y, z \in \mathbb{F}_p \}
\end{equation}
\begin{equation}
    T = U \cup D
\end{equation}
\end{definition}

\begin{definition}
We will also be interested in sets of tiles having at least one nonzero corner:
\begin{equation}
    U^* = U \setminus \{ \bigtriangleup 000 \}
\end{equation}
\begin{equation}
    D^* = D \setminus \{ \bigtriangledown 000 \}
\end{equation}
\begin{equation}
    T^* = U^* \cup D^*
\end{equation}
\end{definition}

\begin{definition}
\label{def:tile-matrix}
We will associate to any tile $t$
the row matrix $[t]$ with the values at its three corners
as follows:
$\forall x,y,z \in \mathbb{F}_p$:
\begin{equation}
    [\bigtriangleup xyz] = [\bigtriangledown xyz] = [x\ y\ z]
\end{equation}
\end{definition}

\subsection{Substitution rule}
\label{sec:sigma-cons-rule}

\begin{definition}
The substitution $\sigma$ transforms a tile into four tiles
as depicted in figures \ref{fig:subst-up} and \ref{fig:subst-down}.
\end{definition}

\begin{figure}[H]
    \centering
    \resizebox{8cm}{!}
    {
    \begin{tikzpicture}
    	\begin{scope}[yscale=.87,xslant=.5]
        \node at (5,2.5) {\Huge $\rightarrow$};

        \draw (0,0) -- (5,0) -- (0,5) -- cycle;
        \node at (1,1) {\Huge $x$};
        \node at (3,1) {\Huge $y$};
        \node at (1,3) {\Huge $z$};

        \draw (7,0) -- (12,0) -- (7,5) -- cycle;
        \node at (8,1) {\huge $x$};
        \node at (10,1) {\huge $x + y$};
        \node at (8,3) {\huge $x + z$};
        \draw (7,5) -- (12,5) -- (7,10) -- cycle;
        \node at (8,6) {\huge $x + z$};
        \node at (10,6) {\huge $y + z$};
        \node at (8,8) {\huge $z$};
        \draw (12,0) -- (17,0) -- (12,5) -- cycle;
        \node at (13,1) {\huge $x + y$};
        \node at (15,1) {\huge $y$};
        \node at (13,3) {\huge $y + z$};
        \draw (12,5) -- (7,5) -- (12,0) -- cycle;
        \node at (11,4) {\huge $y + z$};
        \node at (9,4) {\huge $x + z$};
        \node at (11,2) {\huge $x + y$};
    	\end{scope}
    \end{tikzpicture}
    }
    \caption{The substitution rule for the upward tile $\bigtriangleup xyz$}
    \label{fig:subst-up}
\end{figure}

\begin{figure}[H]
    \centering
    \resizebox{8cm}{!}
    {
    \begin{tikzpicture}
    	\begin{scope}[yscale=.87,xslant=.5]
        \node at (2.25,-2.5) {\Huge $\rightarrow$};

        \draw (0,0) -- (-5,0) -- (0,-5) -- cycle;
        \node at (-1,-1) {\Huge $x$};
        \node at (-3,-1) {\Huge $y$};
        \node at (-1,-3) {\Huge $z$};

        \draw (12,0) -- (7,0) -- (12,-5) -- cycle;
        \node at (11,-1) {\huge $x$};
        \node at (9,-1) {\huge $x + y$};
        \node at (11,-3) {\huge $x + z$};
        \draw (12,-5) -- (7,-5) -- (12,-10) -- cycle;
        \node at (11,-6) {\huge $x + z$};
        \node at (9,-6) {\huge $y + z$};
        \node at (11,-8) {\huge $z$};
        \draw (7,0) -- (2,0) -- (7,-5) -- cycle;
        \node at (6,-1) {\huge $x + y$};
        \node at (4,-1) {\huge $y$};
        \node at (6,-3) {\huge $y + z$};
        \draw (7,-5) -- (12,-5) -- (7,0) -- cycle;
        \node at (8,-4) {\huge $y + z$};
        \node at (10,-4) {\huge $x + z$};
        \node at (8,-2) {\huge $x + y$};
    	\end{scope}
    \end{tikzpicture}
    }
    \caption{The substitution rule for the downward tile $\bigtriangledown xyz$}
    \label{fig:subst-down}
\end{figure}

\begin{definition}
$\forall k \in \mathbb{N}$, $\forall t \in T$,
$\sigma^k(t)$ is called a \quotes{$k$-supertile} or simply a \quotes{supertile}.
A subset of tiles in a supertile is called a \quotes{patch}.
\end{definition}

\begin{definition}
As depicted in figure~\ref{fig:abcd}:
\begin{itemize}
\item $\forall u \in U$:
      the central, bottom left, bottom right and top tile in $\sigma(u)$
      will be denoted respectively by $\alpha(u)$, $\beta(u)$, $\gamma(u)$ and $\delta(u)$,
\item $\forall d \in D$:
      the central, top right, top left and bottom tile in $\sigma(d)$
      will be denoted respectively by $\alpha(d)$, $\beta(d)$, $\gamma(d)$ and $\delta(d)$.
\end{itemize}
\end{definition}

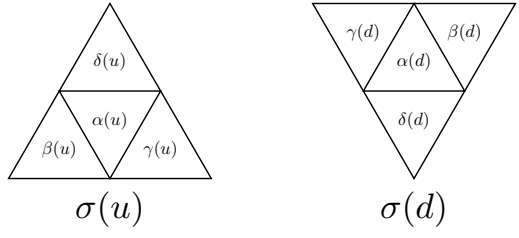
\begin{figure}[H]
    \centering
    \resizebox{7cm}{!}
    {
    \begin{tikzpicture}
    	\begin{scope}[yscale=.87,xslant=.5]
        \draw (0,0) -- (2,0) -- (0,2) -- cycle;
        \draw (1,0) -- (1,1) -- (0,1) -- cycle;
        \node[scale=.5] at (2/3,2/3) {$\alpha(u)$};
        \node[scale=.5] at (1/3,1/3) {$\beta(u)$};
        \node[scale=.5] at (1+1/3,1/3) {$\gamma(u)$};
        \node[scale=.5] at (1/3,1+1/3) {$\delta(u)$};
        \node[below] at (1,0) {$\sigma(u)$};

        \draw (4,0) -- (4,2) -- (2,2) -- cycle;
        \draw (4,1) -- (3,2) -- (3,1) -- cycle;
        \node[scale=.5] at (3+1/3,1+1/3) {$\alpha(d)$};
        \node[scale=.5] at (4-1/3,1+2/3) {$\beta(d)$};
        \node[scale=.5] at (3-1/3,1+2/3) {$\gamma(d)$};
        \node[scale=.5] at (4-1/3,0+2/3) {$\delta(d)$};
        \node[below] at (4,0) {$\sigma(d)$};
    	\end{scope}
    \end{tikzpicture}
    }
    \caption{Tiles in $1$-supertiles}
    \label{fig:abcd}
\end{figure}

\begin{lemma}
\label{lemma:sigma-matrix}
$\forall t \in T$,
using the notation introduced in definition~\ref{def:tile-matrix},
we can conveniently relate $\alpha(t)$, $\beta(t)$, $\gamma(t)$ and $\delta(t)$
to $t$
by means of multiplication by certain matrices $A$, $B$, $C$ and $D$
as follows:
\begin{equation}
    [\alpha(t)]
        = [t] \cdot \underbrace{
                \begin{bmatrix}
                    0 & 1 & 1    \\
                    1 & 0 & 1    \\
                    1 & 1 & 0
                \end{bmatrix}
              }_A
\end{equation}
\begin{equation}
    [\beta(t)]
        = [t] \cdot \underbrace{
                \begin{bmatrix}
                    1 & 1 & 1    \\
                    0 & 1 & 0    \\
                    0 & 0 & 1
                \end{bmatrix}
              }_B
\end{equation}
\begin{equation}
    [\gamma(t)]
        = [t] \cdot \underbrace{
                \begin{bmatrix}
                    1 & 0 & 0    \\
                    1 & 1 & 1    \\
                    0 & 0 & 1
                \end{bmatrix}
              }_C
\end{equation}
\begin{equation}
    [\delta(t)]
        = [t] \cdot \underbrace{
                \begin{bmatrix}
                    1 & 0 & 0    \\
                    0 & 1 & 0    \\
                    1 & 1 & 1
                \end{bmatrix}
              }_D
\end{equation}
\end{lemma}

\begin{proof}
$\forall x,y,z \in \mathbb{F}_p$:
\begin{equation}
    [\alpha(\bigtriangleup xyz)]
    = [\bigtriangledown(y+z,x+z,x+y)]
    = [y+z\ x+z\ x+y]
    = [x\ y\ z] \cdot A
    = [\bigtriangleup xyz] \cdot A
\end{equation}
\begin{equation}
    [\beta(\bigtriangleup xyz)]
    = [\bigtriangleup(x, x+y, x+z)]
    = [x\ x+y\ x+z]
    = [x\ y\ z] \cdot B
    = [\bigtriangleup xyz] \cdot B
\end{equation}
The other cases are similar.
\end{proof}

\begin{lemma}
\label{lemma:abcd-invertible}
The matrices $A$, $B$, $C$ and $D$ are invertible in $\mathbb{F}_p^{3 \times 3}$.
\end{lemma}

\begin{proof}
\begin{equation}
    \det(A) = 2
\end{equation}
\begin{equation}
    \det(B) = \det(C) = \det(D) = 1
\end{equation}
\end{proof}

\begin{remark}
The matrix $A$ is not invertible in $\mathbb{F}_2^{3 \times 3}$.
\end{remark}

\begin{lemma}
The functions $\alpha$, $\beta$, $\gamma$ and $\delta$
are permutations of $T$.
\end{lemma}

\begin{proof}
This is a consequence of lemmas \ref{lemma:sigma-matrix} and \ref{lemma:abcd-invertible}
and of the fact that $\alpha$ inverts tile orientations,
and $\beta$, $\gamma$ and $\delta$ preserve tile orientations.
\end{proof}

\begin{lemma}
\label{lemma:extreme}
$\forall k \in \mathbb{N}$, $\forall t \in T$,
the values at the three extreme corners of $\sigma^k(t)$
are exactly the values at the corners of $t$.
\end{lemma}

\begin{proof}
The property is trivially satisfied for all $0$-supertiles (\textit{i.e.} for all tiles).

Suppose that, for some $k \in \mathbb{N}$,
the property is satisfied for all $k$-supertiles.

Applying $\sigma$ again to some upward $k$-supertile $s$:
\begin{itemize}
\item let $t$ be the tile at the bottom left corner of $s$,
\item $\beta(t)$ is the tile at the bottom left corner of $\sigma(s)$,
\item $\forall x,y,z \in \mathbb{F}_p$,
      $\beta(\bigtriangleup xyz) = \bigtriangleup(x,x+y,x+z)$,
      so $\beta$ preserves values at bottom left corners,
\item hence the bottom left corners of $s$ and $\sigma(s)$ hold the same value,
\item the reasoning for the bottom right and top corners is similar.
\end{itemize}

The reasoning for downward $k$-supertiles is identical.
\end{proof}

\begin{definition}
\label{def:row}
Following lemma~\ref{lemma:extreme},
we will associate to any supertile $s$
the row matrix $[s]$ with the values at its three extreme corners
as follows:
$\forall k \in \mathbb{N}$,
$\forall t \in T$:
\begin{equation}
    [\sigma^k(t)] = [t]
\end{equation}
\end{definition}

\begin{lemma}
For any supertile, the corners that meet at the same point hold the same value.
\end{lemma}

\begin{proof}
The property is trivially satisfied for all $0$-supertiles (\textit{i.e.} for all tiles)
as only one corner can meet at a given point.

Suppose that, for some $k \in \mathbb{N}$,
the property is satisfied for all $k$-supertiles.

Applying $\sigma$ again to some $k$-supertile $s$,
\begin{itemize}
\item the corners added and meeting in the middle of an edge
      in contact with a single tile
      will hold the same value
      (see figures~\ref{fig:subst-up} and \ref{fig:subst-down}),
\item the corners added and meeting in the middle of a common edge
      to two tiles of $s$ only depend on the values
      at the ends of this edge, say $v$ and $w$,
      and they will hold the same value $v+w$
      (see figure~\ref{fig:subst-meet-same}).
\end{itemize}
\end{proof}

\begin{figure}[!ht]
    \centering
    \resizebox{8cm}{!}
    {
    \begin{tikzpicture}
    	\begin{scope}[yscale=.87,xslant=.5]
        \draw (0,5) -- (0,0) -- (5,0) -- (5,5) -- (0,5) -- (5,0);
        \node at (3,1) {\Huge $v$};
        \node at (1,3) {\Huge $w$};
        \node at (4,2) {\Huge $v$};
        \node at (2,4) {\Huge $w$};

        \node at (7,2.5) {\Huge $\rightarrow$};

        \draw (19,0) -- (19,10) -- (9,10) -- (9,0) -- (19,0) -- (9,10);
        \draw (9,5) -- (14,0) -- (14,10) -- (19,5) -- cycle;
        \node at (17,1) {\Huge $v$};
        \node at (18,2) {\Huge $v$};
        \node at (10,8) {\Huge $w$};
        \node at (11,9) {\Huge $w$};
        \node at (15,3) {\huge $v+w$};
        \node at (13,4) {\huge $v+w$};
        \node at (16,4) {\huge $v+w$};

        \node at (13,7) {\huge $v+w$};
        \node at (12,6) {\huge $v+w$};
        \node at (15,6) {\huge $v+w$};
    	\end{scope}
    \end{tikzpicture}
    }
    \caption{The corners that meet at the same point hold the same value.}
    \label{fig:subst-meet-same}
\end{figure}
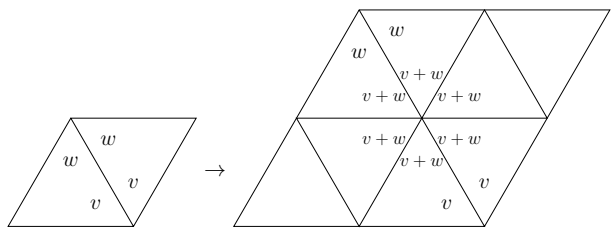

We will therefore, from now on, simplify the figures by indicating
just the values where the corners meet
(see figure~\ref{fig:subst-up-simp}).

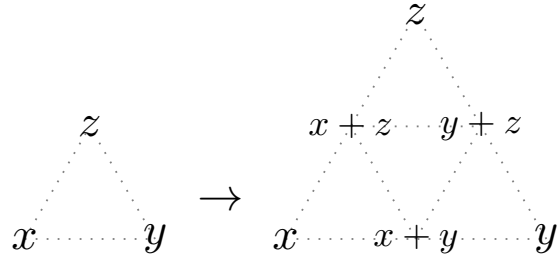
\begin{figure}[!ht]
    \centering
    \resizebox{8cm}{!}
    {
    \begin{tikzpicture}
    	\begin{scope}[yscale=.87,xslant=.5]
        \node at (1+1/3,1/3) {$\rightarrow$};
        \draw[dotted,gray] (0,0) -- (1,0);
        \draw[dotted,gray] (0,0) -- (0,1);
        \draw[dotted,gray] (1,0) -- (0,1);
        \draw[dotted,gray] (2,0) -- (4,0);
        \draw[dotted,gray] (2,0) -- (2,2);
        \draw[dotted,gray] (4,0) -- (2,2);
        \draw[dotted,gray] (2,1) -- (3,1);
        \draw[dotted,gray] (3,0) -- (3,1);
        \draw[dotted,gray] (3,0) -- (2,1);
        \node at (0,0) {$x$};
        \node at (0,1) {$z$};
        \node at (1,0) {$y$};
        \node at (2,0) {$x$};
        \node[scale=.8] at (2,1) {$x+z$};
        \node at (2,2) {$z$};
        \node[scale=.8] at (3,0) {$x+y$};
        \node[scale=.8] at (3,1) {$y+z$};
        \node at (4,0) {$y$};
    	\end{scope}
    \end{tikzpicture}
    }
    \caption{The \quotes{simplified} substitution rule for the upward tile $\bigtriangleup xyz$}
    \label{fig:subst-up-simp}
\end{figure}

\subsection{Examples}
\label{sec:sigma-cons-examples}

Figure~\ref{fig:samplenum} depicts $\sigma^k(\bigtriangleup 122)$ for $k = 0\twodots2$ (for $p = 3$); values are given numerically.

\begin{figure}[H]
    \centering
    \resizebox{10cm}{!}
    {
    \begin{tikzpicture}
    	\begin{scope}[yscale=.87,xslant=.5]
                \draw[dotted,gray] (0,0) -- (1,0);
        \draw[dotted,gray] (0,0) -- (0,1);
        \draw[dotted,gray] (1,0) -- (0,1);
        \node[scale=1.5] at (0,0) {1};
        \node[scale=1.5] at (1,0) {2};
        \node[scale=1.5] at (0,1) {2};
        \draw[dotted,gray] (2,0) -- (4,0);
        \draw[dotted,gray] (2,0) -- (2,2);
        \draw[dotted,gray] (4,0) -- (2,2);
        \draw[dotted,gray] (2,1) -- (3,1);
        \draw[dotted,gray] (3,0) -- (3,1);
        \draw[dotted,gray] (3,0) -- (2,1);
        \node[scale=1.5] at (2,0) {1};
        \node[scale=1.5] at (3,0) {0};
        \node[scale=1.5] at (2,1) {0};
        \node[scale=1.5] at (3,1) {1};
        \node[scale=1.5] at (2,2) {2};
        \node[scale=1.5] at (4,0) {2};
        \draw[dotted,gray] (5,0) -- (9,0);
        \draw[dotted,gray] (5,0) -- (5,4);
        \draw[dotted,gray] (9,0) -- (5,4);
        \draw[dotted,gray] (5,1) -- (8,1);
        \draw[dotted,gray] (6,0) -- (6,3);
        \draw[dotted,gray] (8,0) -- (5,3);
        \draw[dotted,gray] (5,2) -- (7,2);
        \draw[dotted,gray] (7,0) -- (7,2);
        \draw[dotted,gray] (7,0) -- (5,2);
        \draw[dotted,gray] (5,3) -- (6,3);
        \draw[dotted,gray] (8,0) -- (8,1);
        \draw[dotted,gray] (6,0) -- (5,1);
        \node[scale=1.5] at (5,0) {1};
        \node[scale=1.5] at (6,0) {1};
        \node[scale=1.5] at (5,1) {1};
        \node[scale=1.5] at (6,1) {0};
        \node[scale=1.5] at (5,2) {0};
        \node[scale=1.5] at (7,0) {0};
        \node[scale=1.5] at (6,2) {1};
        \node[scale=1.5] at (5,3) {2};
        \node[scale=1.5] at (6,3) {0};
        \node[scale=1.5] at (5,4) {2};
        \node[scale=1.5] at (7,2) {1};
        \node[scale=1.5] at (8,0) {2};
        \node[scale=1.5] at (7,1) {1};
        \node[scale=1.5] at (8,1) {0};
        \node[scale=1.5] at (9,0) {2};
    	\end{scope}
    \end{tikzpicture}
    }
    \caption{$\sigma^k(\bigtriangleup 122)$ for $k = 0\twodots2$ (for $p = 3$)}
    \label{fig:samplenum}
\end{figure}
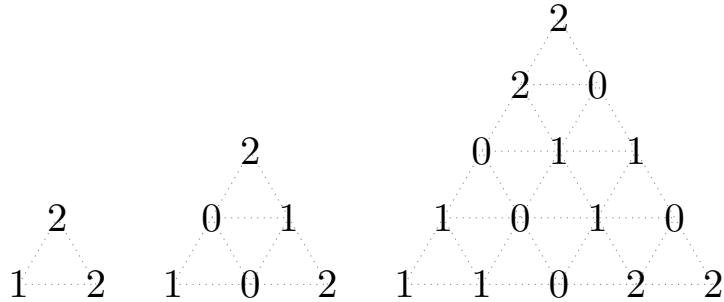

Replacing numerical values by colors
gives greater insight into the variety of patterns;
the next examples will use the following color schemes:
\begin{itemize}
\item for $p = 3$: white = $0$, blue = $1$ and red = $2$;
      see figures \ref{fig:sample-col-num}, \ref{fig:sample} and \ref{fig:sample3},
\item for $p = 5$: white = $0$, cyan = $1$ magenta = $2$, yellow = $3$ and black = $4$;
      see figures \ref{fig:sample2} and \ref{fig:sample4}.
\end{itemize}

\begin{figure}[H]
    \centering
    \resizebox{10cm}{!}
    {
    \begin{tikzpicture}
    	\begin{scope}[yscale=.87,xslant=.5]
        \fill[blue!50] (0,0) -- (0,1/2) -- (1/3,1/3) -- (1/2,0) -- cycle;
\fill[red!50] (1,0) -- (1/2,0) -- (1/3,1/3) -- (1/2,1/2) -- cycle;
\fill[red!50] (0,1) -- (0,1/2) -- (1/3,1/3) -- (1/2,1/2) -- cycle;
\fill[blue!50] (2,0) -- (2,1/2) -- (7/3,1/3) -- (5/2,0) -- cycle;
\fill[red!50] (4,0) -- (7/2,0) -- (10/3,1/3) -- (7/2,1/2) -- cycle;
\fill[blue!50] (3,1) -- (3,1/2) -- (10/3,1/3) -- (7/2,1/2) -- cycle;
\fill[blue!50] (3,1) -- (5/2,1) -- (7/3,4/3) -- (5/2,3/2) -- cycle;
\fill[red!50] (2,2) -- (2,3/2) -- (7/3,4/3) -- (5/2,3/2) -- cycle;
\fill[blue!50] (3,1) -- (3,1/2) -- (8/3,2/3) -- (5/2,1) -- cycle;
\fill[blue!50] (5,0) -- (5,1/2) -- (16/3,1/3) -- (11/2,0) -- cycle;
\fill[blue!50] (6,0) -- (11/2,0) -- (16/3,1/3) -- (11/2,1/2) -- cycle;
\fill[blue!50] (5,1) -- (5,1/2) -- (16/3,1/3) -- (11/2,1/2) -- cycle;
\fill[blue!50] (6,0) -- (6,1/2) -- (19/3,1/3) -- (13/2,0) -- cycle;
\fill[blue!50] (5,1) -- (5,3/2) -- (16/3,4/3) -- (11/2,1) -- cycle;
\fill[blue!50] (6,0) -- (6,1/2) -- (17/3,2/3) -- (11/2,1/2) -- cycle;
\fill[blue!50] (5,1) -- (11/2,1/2) -- (17/3,2/3) -- (11/2,1) -- cycle;
\fill[red!50] (8,0) -- (15/2,0) -- (22/3,1/3) -- (15/2,1/2) -- cycle;
\fill[blue!50] (7,1) -- (7,1/2) -- (22/3,1/3) -- (15/2,1/2) -- cycle;
\fill[red!50] (8,0) -- (8,1/2) -- (25/3,1/3) -- (17/2,0) -- cycle;
\fill[red!50] (9,0) -- (17/2,0) -- (25/3,1/3) -- (17/2,1/2) -- cycle;
\fill[blue!50] (7,1) -- (7,3/2) -- (22/3,4/3) -- (15/2,1) -- cycle;
\fill[blue!50] (7,2) -- (7,3/2) -- (22/3,4/3) -- (15/2,3/2) -- cycle;
\fill[red!50] (8,0) -- (8,1/2) -- (23/3,2/3) -- (15/2,1/2) -- cycle;
\fill[blue!50] (7,1) -- (15/2,1/2) -- (23/3,2/3) -- (15/2,1) -- cycle;
\fill[blue!50] (6,2) -- (11/2,2) -- (16/3,7/3) -- (11/2,5/2) -- cycle;
\fill[red!50] (5,3) -- (5,5/2) -- (16/3,7/3) -- (11/2,5/2) -- cycle;
\fill[blue!50] (6,2) -- (6,5/2) -- (19/3,7/3) -- (13/2,2) -- cycle;
\fill[blue!50] (7,2) -- (13/2,2) -- (19/3,7/3) -- (13/2,5/2) -- cycle;
\fill[red!50] (5,3) -- (5,7/2) -- (16/3,10/3) -- (11/2,3) -- cycle;
\fill[red!50] (5,4) -- (5,7/2) -- (16/3,10/3) -- (11/2,7/2) -- cycle;
\fill[blue!50] (6,2) -- (6,5/2) -- (17/3,8/3) -- (11/2,5/2) -- cycle;
\fill[red!50] (5,3) -- (11/2,5/2) -- (17/3,8/3) -- (11/2,3) -- cycle;
\fill[blue!50] (7,1) -- (7,1/2) -- (20/3,2/3) -- (13/2,1) -- cycle;
\fill[blue!50] (6,2) -- (6,3/2) -- (17/3,5/3) -- (11/2,2) -- cycle;
\fill[blue!50] (7,1) -- (7,3/2) -- (20/3,5/3) -- (13/2,3/2) -- cycle;
\fill[blue!50] (6,2) -- (13/2,3/2) -- (20/3,5/3) -- (13/2,2) -- cycle;
\fill[blue!50] (7,2) -- (7,3/2) -- (20/3,5/3) -- (13/2,2) -- cycle;
\fill[blue!50] (7,1) -- (13/2,1) -- (19/3,4/3) -- (13/2,3/2) -- cycle;
\fill[blue!50] (6,2) -- (6,3/2) -- (19/3,4/3) -- (13/2,3/2) -- cycle;
\node[scale=2] at (0,0) {$1$};
\node[scale=2] at (1,0) {$2$};
\node[scale=2] at (0,1) {$2$};
\draw[dotted] (0,0) -- (1,0) -- (0,1) -- cycle;
\node[scale=2] at (2,0) {$1$};
\node[scale=2] at (3,0) {$0$};
\node[scale=2] at (2,1) {$0$};
\draw[dotted] (2,0) -- (3,0) -- (2,1) -- cycle;
\node[scale=2] at (3,0) {$0$};
\node[scale=2] at (4,0) {$2$};
\node[scale=2] at (3,1) {$1$};
\draw[dotted] (3,0) -- (4,0) -- (3,1) -- cycle;
\node[scale=2] at (2,1) {$0$};
\node[scale=2] at (3,1) {$1$};
\node[scale=2] at (2,2) {$2$};
\draw[dotted] (2,1) -- (3,1) -- (2,2) -- cycle;
\node[scale=2] at (3,0) {$0$};
\node[scale=2] at (2,1) {$0$};
\node[scale=2] at (3,1) {$1$};
\draw[dotted] (3,0) -- (2,1) -- (3,1) -- cycle;
\node[scale=2] at (5,0) {$1$};
\node[scale=2] at (6,0) {$1$};
\node[scale=2] at (5,1) {$1$};
\draw[dotted] (5,0) -- (6,0) -- (5,1) -- cycle;
\node[scale=2] at (6,0) {$1$};
\node[scale=2] at (7,0) {$0$};
\node[scale=2] at (6,1) {$0$};
\draw[dotted] (6,0) -- (7,0) -- (6,1) -- cycle;
\node[scale=2] at (5,1) {$1$};
\node[scale=2] at (6,1) {$0$};
\node[scale=2] at (5,2) {$0$};
\draw[dotted] (5,1) -- (6,1) -- (5,2) -- cycle;
\node[scale=2] at (6,0) {$1$};
\node[scale=2] at (5,1) {$1$};
\node[scale=2] at (6,1) {$0$};
\draw[dotted] (6,0) -- (5,1) -- (6,1) -- cycle;
\node[scale=2] at (7,0) {$0$};
\node[scale=2] at (8,0) {$2$};
\node[scale=2] at (7,1) {$1$};
\draw[dotted] (7,0) -- (8,0) -- (7,1) -- cycle;
\node[scale=2] at (8,0) {$2$};
\node[scale=2] at (9,0) {$2$};
\node[scale=2] at (8,1) {$0$};
\draw[dotted] (8,0) -- (9,0) -- (8,1) -- cycle;
\node[scale=2] at (7,1) {$1$};
\node[scale=2] at (8,1) {$0$};
\node[scale=2] at (7,2) {$1$};
\draw[dotted] (7,1) -- (8,1) -- (7,2) -- cycle;
\node[scale=2] at (8,0) {$2$};
\node[scale=2] at (7,1) {$1$};
\node[scale=2] at (8,1) {$0$};
\draw[dotted] (8,0) -- (7,1) -- (8,1) -- cycle;
\node[scale=2] at (5,2) {$0$};
\node[scale=2] at (6,2) {$1$};
\node[scale=2] at (5,3) {$2$};
\draw[dotted] (5,2) -- (6,2) -- (5,3) -- cycle;
\node[scale=2] at (6,2) {$1$};
\node[scale=2] at (7,2) {$1$};
\node[scale=2] at (6,3) {$0$};
\draw[dotted] (6,2) -- (7,2) -- (6,3) -- cycle;
\node[scale=2] at (5,3) {$2$};
\node[scale=2] at (6,3) {$0$};
\node[scale=2] at (5,4) {$2$};
\draw[dotted] (5,3) -- (6,3) -- (5,4) -- cycle;
\node[scale=2] at (6,2) {$1$};
\node[scale=2] at (5,3) {$2$};
\node[scale=2] at (6,3) {$0$};
\draw[dotted] (6,2) -- (5,3) -- (6,3) -- cycle;
\node[scale=2] at (7,0) {$0$};
\node[scale=2] at (6,1) {$0$};
\node[scale=2] at (7,1) {$1$};
\draw[dotted] (7,0) -- (6,1) -- (7,1) -- cycle;
\node[scale=2] at (6,1) {$0$};
\node[scale=2] at (5,2) {$0$};
\node[scale=2] at (6,2) {$1$};
\draw[dotted] (6,1) -- (5,2) -- (6,2) -- cycle;
\node[scale=2] at (7,1) {$1$};
\node[scale=2] at (6,2) {$1$};
\node[scale=2] at (7,2) {$1$};
\draw[dotted] (7,1) -- (6,2) -- (7,2) -- cycle;
\node[scale=2] at (6,1) {$0$};
\node[scale=2] at (7,1) {$1$};
\node[scale=2] at (6,2) {$1$};
\draw[dotted] (6,1) -- (7,1) -- (6,2) -- cycle;
    	\end{scope}
    \end{tikzpicture}
    }
    \caption{Colored representation of $\sigma^k(\bigtriangleup 122)$ for $k = 0 \twodots 2$ (for $p = 3$)}
    \label{fig:sample-col-num}
\end{figure}

\begin{figure}[H]
    \centering
    \resizebox{10cm}{!}
    {
    \begin{tikzpicture}
    	\begin{scope}[yscale=.87,xslant=.5]
        \input{figure-sample.tikz}
    	\end{scope}
    \end{tikzpicture}
    }
    \caption{Colored representation of $\sigma^k(\bigtriangleup 122)$ for $k = 0 \twodots 3$ (for $p = 3$)}
    \label{fig:sample}
\end{figure}

\begin{figure}[H]
    \begin{center}
    \vstretch{1.732}{
        \includegraphics[height=7cm, keepaspectratio]{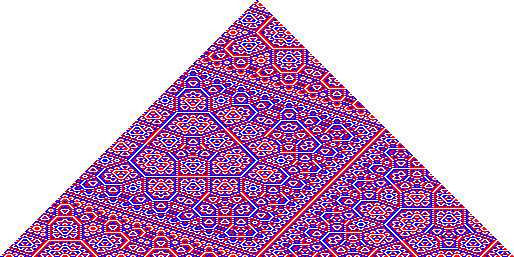}
    }
    \caption{Colored representation of $\sigma^8(\bigtriangleup 121)$ (for $p = 3$)}
    \label{fig:sample3}
    \end{center}
\end{figure}

\begin{figure}[H]
    \centering
    \resizebox{10cm}{!}
    {
    \begin{tikzpicture}
    	\begin{scope}[yscale=.87,xslant=.5]
        \input{figure-sample2.tikz}
    	\end{scope}
    \end{tikzpicture}
    }
    \caption{Colored representation of $\sigma^k(\bigtriangleup 123)$ for $k = 0\twodots3$ (for $p = 5$)}
    \label{fig:sample2}
\end{figure}

\begin{figure}[H]
    \begin{center}
    \vstretch{1.732}{
        \includegraphics[height=7cm, keepaspectratio]{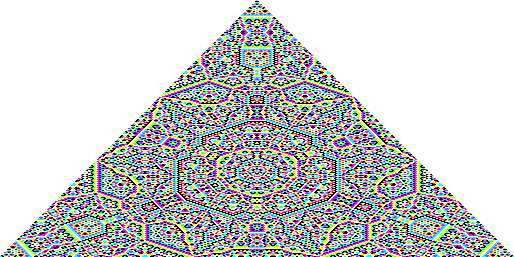}
    }
    \caption{Colored representation of $\sigma^8(\bigtriangleup 111)$ (for $p = 5$)}
    \label{fig:sample4}
    \end{center}
\end{figure}

\subsection{Additivity}
\label{sec:sigma-prop-additivity}

\begin{lemma}
$\forall k \in \mathbb{N}$,
$\forall m, x, y, z, x', y', z' \in \mathbb{F}_p$,
we have the following identities:

\begin{equation}
    \sigma^k(\bigtriangleup xyz) + \sigma^k(\bigtriangleup x'y'z') = \sigma^k(\bigtriangleup(x+x',y+y',z+z'))
\end{equation}
\begin{equation}
    m \cdot \sigma^k(\bigtriangleup xyz) = \sigma^k(\bigtriangleup(mx,my,mz))
\end{equation}

and also:

\begin{equation}
    \sigma^k(\bigtriangledown xyz) + \sigma^k(\bigtriangledown x'y'z') = \sigma^k(\bigtriangledown(x+x',y+y',z+z'))
\end{equation}
\begin{equation}
    m \cdot \sigma^k(\bigtriangledown xyz) = \sigma^k(\bigtriangledown(mx,my,mz))
\end{equation}

where additions and scalar multiplications
on supertiles are performed componentwise on each corner
(see figure~\ref{fig:additivity} and \ref{fig:scalar}).
\end{lemma}

\begin{proof}
The property is satisfied for all $0$-supertiles (\textit{i.e.} for all tiles);
$\forall m, x, y, z, x', y', z' \in \mathbb{F}_p$:
\begin{equation}
    \bigtriangleup xyz + \bigtriangleup x'y'z' = \bigtriangleup(x+x', y+y', z+z')
\end{equation}
\begin{equation}
    m \cdot \bigtriangleup xyz = \bigtriangleup(mx, my, mz)
\end{equation}

Suppose that, for some $k \in \mathbb{N}$,
the property is satisfied for all $k$-supertiles.

For any $k$-supertile $s$, $\sigma(s)$ is obtained
by applying $\sigma$ to each tile in $s$,
so we just have to prove that $\sigma$ is additive,
which is equivalent to prove that $\alpha$, $\beta$, $\gamma$ and $\delta$ are additive.

$\forall m, x, y, z, x', y', z' \in \mathbb{F}_p$:
\begin{align}
    \alpha(\bigtriangleup xyz + \bigtriangleup x'y'z')
    &= \alpha(\bigtriangleup(x+x', y+y', z+z'))                                    \\
    &= \bigtriangledown(y+y'+z+z', x+x'+z+z', x+x'+y+y')                           \\
    &= \bigtriangledown(y+z, x+z, x+y) + \bigtriangledown(y'+z', x'+z', x'+y')     \\
    &= \alpha(\bigtriangleup xyz) + \alpha(\bigtriangleup x'y'z')
\end{align}
\begin{align}
    \alpha(m \cdot \bigtriangleup xyz) &= \alpha(\bigtriangleup(mx, my, mz))       \\
                                       &= \bigtriangledown(my+mz, mx+mz, mx+my)    \\
                                       &= m \cdot \bigtriangledown(y+z, x+z, x+y)  \\
                                       &= m \cdot \alpha(\bigtriangleup xyz)
\end{align}
The reasoning for the other cases is identical.
\end{proof}

\begin{figure}[H]
    \centering
    \resizebox{8cm}{!}
    {
    \begin{tikzpicture}
    	\begin{scope}[yscale=.87,xslant=.5]
    	    \node[scale=3] at (4,2) {+};
    	    \node[scale=3] at (10,2) {=};
                \draw[dotted,gray] (0,0) -- (4,0);
        \draw[dotted,gray] (0,0) -- (0,4);
        \draw[dotted,gray] (4,0) -- (0,4);
        \draw[dotted,gray] (0,1) -- (3,1);
        \draw[dotted,gray] (1,0) -- (1,3);
        \draw[dotted,gray] (3,0) -- (0,3);
        \draw[dotted,gray] (0,2) -- (2,2);
        \draw[dotted,gray] (2,0) -- (2,2);
        \draw[dotted,gray] (2,0) -- (0,2);
        \draw[dotted,gray] (0,3) -- (1,3);
        \draw[dotted,gray] (3,0) -- (3,1);
        \draw[dotted,gray] (1,0) -- (0,1);
        \draw[dotted,gray] (6,0) -- (10,0);
        \draw[dotted,gray] (6,0) -- (6,4);
        \draw[dotted,gray] (10,0) -- (6,4);
        \draw[dotted,gray] (6,1) -- (9,1);
        \draw[dotted,gray] (7,0) -- (7,3);
        \draw[dotted,gray] (9,0) -- (6,3);
        \draw[dotted,gray] (6,2) -- (8,2);
        \draw[dotted,gray] (8,0) -- (8,2);
        \draw[dotted,gray] (8,0) -- (6,2);
        \draw[dotted,gray] (6,3) -- (7,3);
        \draw[dotted,gray] (9,0) -- (9,1);
        \draw[dotted,gray] (7,0) -- (6,1);
        \draw[dotted,gray] (12,0) -- (16,0);
        \draw[dotted,gray] (12,0) -- (12,4);
        \draw[dotted,gray] (16,0) -- (12,4);
        \draw[dotted,gray] (12,1) -- (15,1);
        \draw[dotted,gray] (13,0) -- (13,3);
        \draw[dotted,gray] (15,0) -- (12,3);
        \draw[dotted,gray] (12,2) -- (14,2);
        \draw[dotted,gray] (14,0) -- (14,2);
        \draw[dotted,gray] (14,0) -- (12,2);
        \draw[dotted,gray] (12,3) -- (13,3);
        \draw[dotted,gray] (15,0) -- (15,1);
        \draw[dotted,gray] (13,0) -- (12,1);
        \node[scale=2] at (0,0) {$1$};
        \node[scale=2] at (1,0) {$0$};
        \node[scale=2] at (0,1) {$2$};
        \node[scale=2] at (1,1) {$0$};
        \node[scale=2] at (0,2) {$1$};
        \node[scale=2] at (2,0) {$2$};
        \node[scale=2] at (1,2) {$2$};
        \node[scale=2] at (0,3) {$1$};
        \node[scale=2] at (1,3) {$1$};
        \node[scale=2] at (0,4) {$0$};
        \node[scale=2] at (2,2) {$1$};
        \node[scale=2] at (3,0) {$0$};
        \node[scale=2] at (2,1) {$0$};
        \node[scale=2] at (3,1) {$2$};
        \node[scale=2] at (4,0) {$1$};
        \node[scale=2] at (6,0) {$0$};
        \node[scale=2] at (7,0) {$2$};
        \node[scale=2] at (6,1) {$1$};
        \node[scale=2] at (7,1) {$0$};
        \node[scale=2] at (6,2) {$1$};
        \node[scale=2] at (8,0) {$2$};
        \node[scale=2] at (7,2) {$1$};
        \node[scale=2] at (6,3) {$2$};
        \node[scale=2] at (7,3) {$1$};
        \node[scale=2] at (6,4) {$1$};
        \node[scale=2] at (8,2) {$0$};
        \node[scale=2] at (9,0) {$1$};
        \node[scale=2] at (8,1) {$2$};
        \node[scale=2] at (9,1) {$2$};
        \node[scale=2] at (10,0) {$2$};
        \node[scale=2] at (12,0) {$1$};
        \node[scale=2] at (13,0) {$2$};
        \node[scale=2] at (12,1) {$0$};
        \node[scale=2] at (13,1) {$0$};
        \node[scale=2] at (12,2) {$2$};
        \node[scale=2] at (14,0) {$1$};
        \node[scale=2] at (13,2) {$0$};
        \node[scale=2] at (12,3) {$0$};
        \node[scale=2] at (13,3) {$2$};
        \node[scale=2] at (12,4) {$1$};
        \node[scale=2] at (14,2) {$1$};
        \node[scale=2] at (15,0) {$1$};
        \node[scale=2] at (14,1) {$2$};
        \node[scale=2] at (15,1) {$1$};
        \node[scale=2] at (16,0) {$0$};
    	\end{scope}
    \end{tikzpicture}
    }
    \caption{$\sigma^2(\bigtriangleup 110) + \sigma^2(\bigtriangleup 021) = \sigma^2(\bigtriangleup 101)$ (for $p = 3$)}
    \label{fig:additivity}
\end{figure}

\begin{figure}[H]
    \centering
    \resizebox{8cm}{!}
    {
    \begin{tikzpicture}
    	\begin{scope}[yscale=.87,xslant=.5]
    	    \node[scale=3] at (-2,2) {$2 \cdot$};
    	    \node[scale=3] at (4,2) {$=$};
        \node[scale=2] at (0,0) {$1$};
\node[scale=2] at (1,0) {$4$};
\node[scale=2] at (0,1) {$0$};
\draw[dotted] (0,0) -- (1,0) -- (0,1) -- cycle;
\node[scale=2] at (1,0) {$4$};
\node[scale=2] at (2,0) {$3$};
\node[scale=2] at (1,1) {$2$};
\draw[dotted] (1,0) -- (2,0) -- (1,1) -- cycle;
\node[scale=2] at (0,1) {$0$};
\node[scale=2] at (1,1) {$2$};
\node[scale=2] at (0,2) {$4$};
\draw[dotted] (0,1) -- (1,1) -- (0,2) -- cycle;
\node[scale=2] at (1,0) {$4$};
\node[scale=2] at (0,1) {$0$};
\node[scale=2] at (1,1) {$2$};
\draw[dotted] (1,0) -- (0,1) -- (1,1) -- cycle;
\node[scale=2] at (2,0) {$3$};
\node[scale=2] at (3,0) {$0$};
\node[scale=2] at (2,1) {$3$};
\draw[dotted] (2,0) -- (3,0) -- (2,1) -- cycle;
\node[scale=2] at (3,0) {$0$};
\node[scale=2] at (4,0) {$2$};
\node[scale=2] at (3,1) {$2$};
\draw[dotted] (3,0) -- (4,0) -- (3,1) -- cycle;
\node[scale=2] at (2,1) {$3$};
\node[scale=2] at (3,1) {$2$};
\node[scale=2] at (2,2) {$0$};
\draw[dotted] (2,1) -- (3,1) -- (2,2) -- cycle;
\node[scale=2] at (3,0) {$0$};
\node[scale=2] at (2,1) {$3$};
\node[scale=2] at (3,1) {$2$};
\draw[dotted] (3,0) -- (2,1) -- (3,1) -- cycle;
\node[scale=2] at (0,2) {$4$};
\node[scale=2] at (1,2) {$4$};
\node[scale=2] at (0,3) {$2$};
\draw[dotted] (0,2) -- (1,2) -- (0,3) -- cycle;
\node[scale=2] at (1,2) {$4$};
\node[scale=2] at (2,2) {$0$};
\node[scale=2] at (1,3) {$3$};
\draw[dotted] (1,2) -- (2,2) -- (1,3) -- cycle;
\node[scale=2] at (0,3) {$2$};
\node[scale=2] at (1,3) {$3$};
\node[scale=2] at (0,4) {$3$};
\draw[dotted] (0,3) -- (1,3) -- (0,4) -- cycle;
\node[scale=2] at (1,2) {$4$};
\node[scale=2] at (0,3) {$2$};
\node[scale=2] at (1,3) {$3$};
\draw[dotted] (1,2) -- (0,3) -- (1,3) -- cycle;
\node[scale=2] at (2,0) {$3$};
\node[scale=2] at (1,1) {$2$};
\node[scale=2] at (2,1) {$3$};
\draw[dotted] (2,0) -- (1,1) -- (2,1) -- cycle;
\node[scale=2] at (1,1) {$2$};
\node[scale=2] at (0,2) {$4$};
\node[scale=2] at (1,2) {$4$};
\draw[dotted] (1,1) -- (0,2) -- (1,2) -- cycle;
\node[scale=2] at (2,1) {$3$};
\node[scale=2] at (1,2) {$4$};
\node[scale=2] at (2,2) {$0$};
\draw[dotted] (2,1) -- (1,2) -- (2,2) -- cycle;
\node[scale=2] at (1,1) {$2$};
\node[scale=2] at (2,1) {$3$};
\node[scale=2] at (1,2) {$4$};
\draw[dotted] (1,1) -- (2,1) -- (1,2) -- cycle;
\node[scale=2] at (6,0) {$2$};
\node[scale=2] at (7,0) {$3$};
\node[scale=2] at (6,1) {$0$};
\draw[dotted] (6,0) -- (7,0) -- (6,1) -- cycle;
\node[scale=2] at (7,0) {$3$};
\node[scale=2] at (8,0) {$1$};
\node[scale=2] at (7,1) {$4$};
\draw[dotted] (7,0) -- (8,0) -- (7,1) -- cycle;
\node[scale=2] at (6,1) {$0$};
\node[scale=2] at (7,1) {$4$};
\node[scale=2] at (6,2) {$3$};
\draw[dotted] (6,1) -- (7,1) -- (6,2) -- cycle;
\node[scale=2] at (7,0) {$3$};
\node[scale=2] at (6,1) {$0$};
\node[scale=2] at (7,1) {$4$};
\draw[dotted] (7,0) -- (6,1) -- (7,1) -- cycle;
\node[scale=2] at (8,0) {$1$};
\node[scale=2] at (9,0) {$0$};
\node[scale=2] at (8,1) {$1$};
\draw[dotted] (8,0) -- (9,0) -- (8,1) -- cycle;
\node[scale=2] at (9,0) {$0$};
\node[scale=2] at (10,0) {$4$};
\node[scale=2] at (9,1) {$4$};
\draw[dotted] (9,0) -- (10,0) -- (9,1) -- cycle;
\node[scale=2] at (8,1) {$1$};
\node[scale=2] at (9,1) {$4$};
\node[scale=2] at (8,2) {$0$};
\draw[dotted] (8,1) -- (9,1) -- (8,2) -- cycle;
\node[scale=2] at (9,0) {$0$};
\node[scale=2] at (8,1) {$1$};
\node[scale=2] at (9,1) {$4$};
\draw[dotted] (9,0) -- (8,1) -- (9,1) -- cycle;
\node[scale=2] at (6,2) {$3$};
\node[scale=2] at (7,2) {$3$};
\node[scale=2] at (6,3) {$4$};
\draw[dotted] (6,2) -- (7,2) -- (6,3) -- cycle;
\node[scale=2] at (7,2) {$3$};
\node[scale=2] at (8,2) {$0$};
\node[scale=2] at (7,3) {$1$};
\draw[dotted] (7,2) -- (8,2) -- (7,3) -- cycle;
\node[scale=2] at (6,3) {$4$};
\node[scale=2] at (7,3) {$1$};
\node[scale=2] at (6,4) {$1$};
\draw[dotted] (6,3) -- (7,3) -- (6,4) -- cycle;
\node[scale=2] at (7,2) {$3$};
\node[scale=2] at (6,3) {$4$};
\node[scale=2] at (7,3) {$1$};
\draw[dotted] (7,2) -- (6,3) -- (7,3) -- cycle;
\node[scale=2] at (8,0) {$1$};
\node[scale=2] at (7,1) {$4$};
\node[scale=2] at (8,1) {$1$};
\draw[dotted] (8,0) -- (7,1) -- (8,1) -- cycle;
\node[scale=2] at (7,1) {$4$};
\node[scale=2] at (6,2) {$3$};
\node[scale=2] at (7,2) {$3$};
\draw[dotted] (7,1) -- (6,2) -- (7,2) -- cycle;
\node[scale=2] at (8,1) {$1$};
\node[scale=2] at (7,2) {$3$};
\node[scale=2] at (8,2) {$0$};
\draw[dotted] (8,1) -- (7,2) -- (8,2) -- cycle;
\node[scale=2] at (7,1) {$4$};
\node[scale=2] at (8,1) {$1$};
\node[scale=2] at (7,2) {$3$};
\draw[dotted] (7,1) -- (8,1) -- (7,2) -- cycle;
    	\end{scope}
    \end{tikzpicture}
    }
    \caption{$2 \cdot \sigma^2(\bigtriangleup 123) = \sigma^2(\bigtriangleup 241)$ (for $p = 5$)}
    \label{fig:scalar}
\end{figure}

\begin{remark}
These identities are reminiscent of generalized additivity
in some elementary cellular automata discovered by Wolfram~\cite{wolfram}.
\end{remark}

\subsection{Symmetries}
\label{sec:sigma-prop-symmetries}

$\forall k \in \mathbb{N}$,
$\forall x, y, z \in \mathbb{F}_p$:
\begin{itemize}
\item if $x = y = z$
      then $\sigma^k(\bigtriangleup xyz)$ and $\sigma^k(\bigtriangledown xyz)$
      have reflection and 3-fold rotational symmetries
      (see figure~\ref{fig:sample4}),
\item if $x = y$ or $y = z$ or $z = x$
      then $\sigma^k(\bigtriangleup xyz)$ and $\sigma^k(\bigtriangledown xyz)$
      have reflection symmetry
      (see figure~\ref{fig:sample3}),
\item if one of $x$, $y$ or $z$ is $0$
      and the two others values sum to $0$ (\textit{i.e.}, the two other values are opposite)
      then $\sigma^k(\bigtriangleup xyz)$ and $\sigma^k(\bigtriangledown xyz)$
      have reflection \quotes{odd} symmetry:
      the values at two symmetrical corners with respect to the bisector passing through the corner of value $0$ are opposite
      (see figure~\ref{fig:stationary}).
\end{itemize}

Some patches can induce other types of symmetry
(see for example remark~\ref{remark-6-fold-symmetry}).

\subsection{Positions within supertiles}
\label{sec:sigma-prop-positions}

\begin{definition}
Let:
\begin{equation}
    \Sigma = \{ \alpha, \beta, \gamma, \delta \}
\end{equation}
\end{definition}

\begin{remark}
\label{remark:alphabet}
$\Sigma$ will be used as an alphabet:
\begin{itemize}
\item $\forall k \in \mathbb{N}$,
      $\Sigma^k$ will denote the set of words of length $k$ over $\Sigma$,
\item $\varepsilon$ will denote the \quotes{empty word}
      (\textit{i.e.} the only element of $\Sigma^0$),
\item $\Sigma^* = \bigcup_{k \in \mathbb{N}} \Sigma^k$
      will denote the set of finite words over $\Sigma$,
\item $\forall d \in \Sigma$,
      $\forall k \in \mathbb{N}$,
      \quotes{$d^k$} will denote \quotes{$d$ repeated $k$ times},
\item $\forall k,k' \in \mathbb{N}$,
      $\forall w \in \Sigma^k$,
      $\forall w' \in \Sigma^{k'}$,
      \quotes{$ww'$} will denote the \quotes{concatenation of $w$ and $w'$}
      (and will belong to $\Sigma^{k+k'}$).
\end{itemize}
\end{remark}

\begin{definition}
$\forall k \in \mathbb{N}$,
a $(k+1)$-supertile contains four $k$-supertiles
(see figure~\ref{fig:position}):
\begin{itemize}
\item for an upward $(k+1)$-supertile $s$:
      the central, bottom left, bottom right and top $k$-supertiles within $s$
      will be respectively given positions $\alpha$, $\beta$, $\gamma$ and $\delta$,
\item for a downward $(k+1)$-supertile $s$:
      the central, top right, top left and bottom $k$-supertiles within $s$
      will be respectively given positions $\alpha$, $\beta$, $\gamma$ and $\delta$.
\end{itemize}
\end{definition}

\begin{figure}[H]
    \centering
    \resizebox{7cm}{!}
    {
    \begin{tikzpicture}
    	\begin{scope}[yscale=.87,xslant=.5]
        \draw (0,0) -- (2,0) -- (0,2) -- cycle;
        \draw (1,0) -- (1,1) -- (0,1) -- cycle;
        \node at (2/3,2/3) {$\alpha$};
        \node at (1/3,1/3) {$\beta$};
        \node at (1+1/3,1/3) {$\gamma$};
        \node at (1/3,1+1/3) {$\delta$};

        \draw (4,0) -- (4,2) -- (2,2) -- cycle;
        \draw (4,1) -- (3,2) -- (3,1) -- cycle;
        \node at (3+1/3,1+1/3) {$\alpha$};
        \node at (4-1/3,1+2/3) {$\beta$};
        \node at (3-1/3,1+2/3) {$\gamma$};
        \node at (4-1/3,0+2/3) {$\delta$};
    	\end{scope}
    \end{tikzpicture}
    }
    \caption{Positions within supertiles}
    \label{fig:position}
\end{figure}
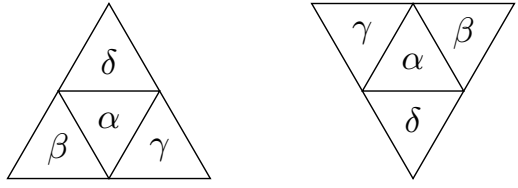

\begin{theo}
\label{theo:tile-position}
$\forall k \in \mathbb{N}$,
the position of any tile $t$ within a $k$-supertile is uniquely
defined by a word $w \in \Sigma^k$
such that $\forall n \in 0 \twodots k-1$,
$w_n$ is the position of the $(k-n-1)$-supertile containing $t$
within the $(k-n)$-supertile containing $t$
(see figure~\ref{fig:positions}).
\end{theo}

\begin{proof}
The property is satisfied for all $0$-supertiles (\textit{i.e.} for all tiles):
$0$-supertiles contain a unique tile whose position correspond to the empty word $\varepsilon$.

Suppose that, for some $k \in \mathbb{N}$,
the property is satisfied for all $k$-supertiles.

A tile $t$ in a $(k+1)$-supertile $s$
belongs to one of the four $k$-supertiles $s'$ in $s$.
If $s'$ is at position $d$ in $s$,
and $t$ is at position $w$ in $s'$,
then $t$ is at position $dw$ in $s$.
\end{proof}

\begin{figure}[H]
    \centering
    \resizebox{12cm}{!}
    {
    \begin{tikzpicture}
    	\begin{scope}[yscale=.87,xslant=.5]
        \input{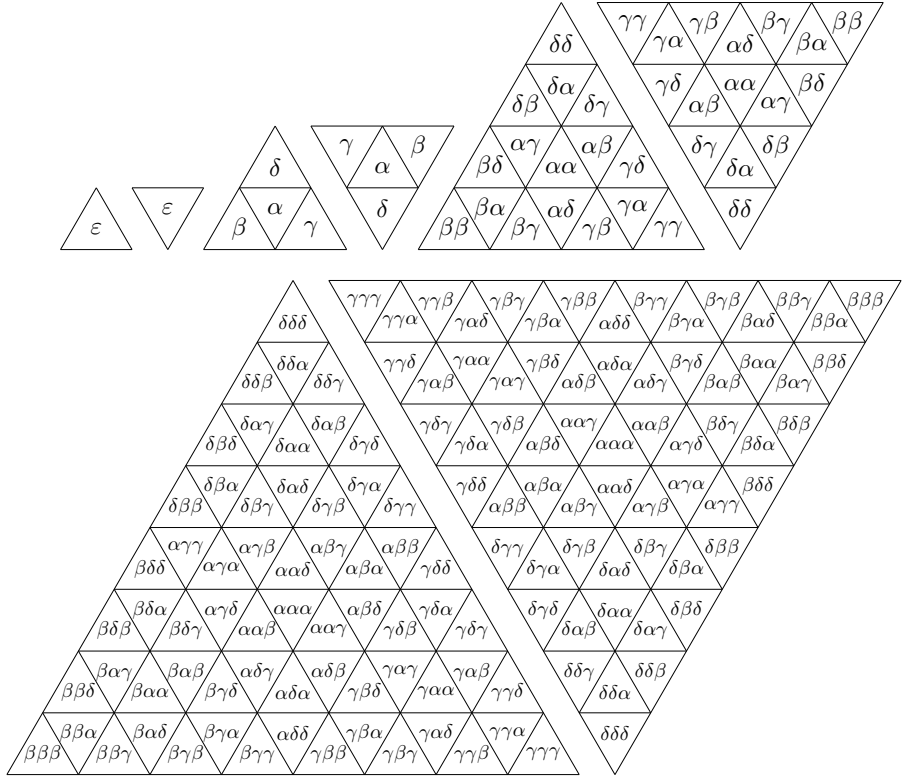}
    	\end{scope}
    \end{tikzpicture}
    }
    \caption{Positions within $k$-supertiles for $k = 0 \twodots 3$}
    \label{fig:positions}
\end{figure}

\begin{remark}
$\forall k \in \mathbb{N}$,
within a $k$-supertile $s$:
\begin{itemize}
\item the central tile has position $\alpha^k$,
\item the bottom left tile (when $s$ is upward)
      or the top right tile (when $s$ is downward)
      has position $\beta^k$,
\item the bottom right tile (when $s$ is upward)
      or the top left tile (when $s$ is downward)
      has position $\gamma^k$,
\item the top tile (when $s$ is upward)
      or the bottom tile (when $s$ is downward)
      has position $\delta^k$.
\end{itemize}
\end{remark}

\begin{lemma}
\label{lemma:flip}
A tile $t$ at position $w$ within a supertile $s$
has the same orientation as $s$
if and only if the number of $\alpha$'s in $w$ is even.
\end{lemma}

\begin{proof}
The position $w$ of a tile $t$ within a supertile $s$
can be interpreted as a descent from $s$ to $t$,
where each $\alpha$ induces an inversion of orientation.
Those inversions are canceled
when they occur an even number of times.
\end{proof}

\begin{definition}
Similarly to definition~\ref{def:row},
we will associate to any position $d \in \Sigma$ the square matrix $[d]$ as follows:
\begin{equation}
    [\alpha] = A
\end{equation}
\begin{equation}
    [\beta] = B
\end{equation}
\begin{equation}
    [\gamma] = C
\end{equation}
\begin{equation}
    [\delta] = D
\end{equation}

Also, $\forall k \in \mathbb{N}$, $\forall w \in \Sigma^k$:
\begin{equation}
    [w] = [w_0] \twodots [w_{k-1}]
\end{equation}
In particular:
\begin{equation}
    [\varepsilon] = I_3
\end{equation}
\end{definition}

\begin{theo}
\label{theo:sigma-prop-rebuilding}
For any supertile $s$, and any tile $t$ at position $w$ in $s$:
\begin{equation}
    [t] = [s] \cdot [w]
\end{equation}
\end{theo}

\begin{proof}
The property is satisfied for $0$-supertiles (\textit{i.e.} for tiles):
$\forall t \in T$:
\begin{equation}
    [t] \cdot [\varepsilon] = [t] \cdot I_3 = [t]
\end{equation}

Suppose that, for some $k \in \mathbb{N}$,
the property is satisfied for all $k$-supertiles.

For any $(k+1)$-supertile $s$, and any tile $t$ within $s$:
$t$ appears at position $w$ within some $k$-supertile $s'$ at position $d$ within $s$
(and the position of $t$ within $s$ is $dw$).
\begin{equation}
    [s'] = [s] \cdot [d]
\end{equation}
\begin{equation}
    [t] = [s'] \cdot [w]
\end{equation}
Hence:
\begin{equation}
    [t] = [s] \cdot [d][w] = [s] \cdot [dw]
\end{equation}
\end{proof}

\begin{lemma}
\label{lemma:unique}
A supertile $s$ is uniquely determined by one of its tile $t$
and the position $w$ of $t$ within $s$.
\end{lemma}

\begin{proof}
From theorem~\ref{theo:sigma-prop-rebuilding}, we can write:
\begin{equation}
    [s] = [t] \cdot [w]^{-1}
\end{equation}

Also, lemma~\ref{lemma:flip} relates the orientations of $t$ and $s$.
\end{proof}

\subsection{Patterns of zeros}
\label{sec:sigma-prop-patterns}

\begin{lemma}
A supertile $s$ contains an all-zero tile ($\bigtriangleup 000$ or $\bigtriangledown 000$)
if and only if $s$ was built starting from an all-zero tile.
\end{lemma}

\begin{proof}
If $s$ contains $t \in \{ \bigtriangleup 000, \bigtriangledown 000 \}$, say at position $w$, then:
\begin{equation}
    [s] = [t] \cdot [w]^{-1} = [0\ 0\ 0] \cdot [w]^{-1} = [0\ 0\ 0]
\end{equation}
\end{proof}

\begin{lemma}
$\forall m \in \mathbb{F}_p^*$,
$\forall k \in \mathbb{N}$,
$\forall t \in T$,
the $0$'s at the corners of $m \cdot \sigma^k(t)$ and $\sigma^k(t)$
are exactly located at the same places.
\end{lemma}

\begin{proof}
This is a consequence of additivity (see section~\ref{sec:sigma-prop-additivity}).
\end{proof}

Figure~\ref{fig:zero} represents $0$'s in $\sigma^9(\bigtriangleup 100)$
and $\sigma^9(\bigtriangleup 200)$ (for $p = 3$).

\begin{figure}[H]
    \begin{center}
    \vstretch{1.732}{
        \includegraphics[height=7cm, keepaspectratio]{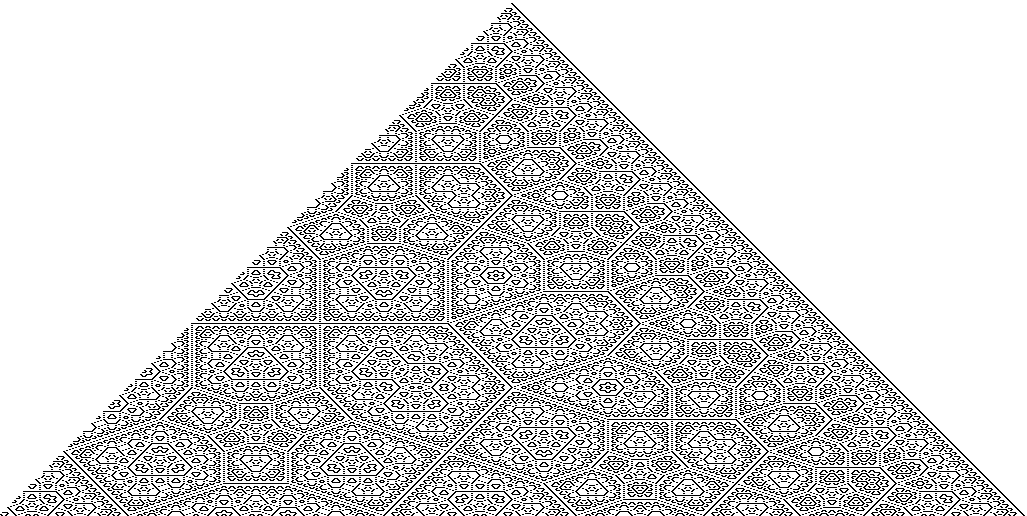}
    }
    \caption{Representation of $0$'s
             (as black pixels)
             in $\sigma^9(\bigtriangleup 100)$
             and $\sigma^9(\bigtriangleup 200)$ (for $p = 5$)}
    \label{fig:zero}
    \end{center}
\end{figure}

\begin{lemma}
Tiles with two $0$'s and some other value $c$
will yield lines of $0$'s and $c$'s parallel to the side holding the two $0$'s
upon repeated applications of $\sigma$.
This pattern is visible in figures~\ref{fig:sample}, \ref{fig:sample3} and \ref{fig:sample4}
as white and colored lines respectively for $0$'s and $c$'s.
\end{lemma}

\begin{proof}
See figure~\ref{fig:persistence-line}.
\end{proof}

\begin{figure}[H]
    \centering
    {
    \begin{tikzpicture}
    	\begin{scope}[yscale=.87,xslant=.5]
        \fill[green!30] (0,0) -- (1,0) -- (0,1) -- cycle;
        \draw[dotted,gray] (0,0) -- (1,0) -- (0,1) -- cycle;
        \node at (0,0) {$0$};
        \node at (1,0) {$0$};
        \node at (0,1) {$c$};

        \node at (2,0) {$\rightarrow$};

        \fill[green!30] (3,0) -- (5,0) -- (4,1) -- (3,1) -- cycle;
        \draw[dotted,gray] (3,0) -- (5,0) -- (3,2) -- cycle;
        \draw[dotted,gray] (4,0) -- (4,1) -- (3,1) -- cycle;
        \node at (3,0) {$0$};
        \node at (4,0) {$0$};
        \node at (5,0) {$0$};
        \node at (3,1) {$c$};
        \node at (4,1) {$c$};

        \node at (6,0) {$\rightarrow$};

        \fill[green!30] (7,0) -- (11,0) -- (10,1) -- (7,1) -- cycle;
        \draw[dotted,gray] (7,0) -- (11,0) -- (9,2) -- (7,2) -- cycle;
        \draw[dotted,gray] (7,2) -- (9,0) -- (9,2);
        \draw[dotted,gray] (7,1) -- (8,0) -- (8,2) -- (10,0) -- (10,1) -- cycle;
        \node at (7,0) {$0$};
        \node at (8,0) {$0$};
        \node at (9,0) {$0$};
        \node at (10,0) {$0$};
        \node at (11,0) {$0$};
        \node at (7,1) {$c$};
        \node at (8,1) {$c$};
        \node at (9,1) {$c$};
        \node at (10,1) {$c$};
    	\end{scope}
    \end{tikzpicture}
    }
    \caption{Lines of $0$'s and $c$'s}
    \label{fig:persistence-line}
\end{figure}
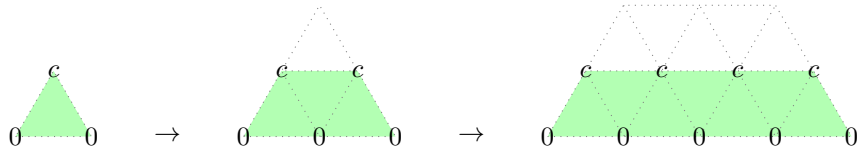

\begin{lemma}
Tiles with one $0$ and two opposite other values $+c$ and $-c$
will yield lines alternating $0$'s and opposite pairs $+c$ / $-c$
upon repeated applications of $\sigma$.
This pattern passes though the center of figure~\ref{fig:stationary}.
\end{lemma}

\begin{proof}
See figure~\ref{fig:persistence-alt}.
\end{proof}

\begin{figure}[H]
    \centering
    {
    \begin{tikzpicture}
    	\begin{scope}[yscale=.87,xslant=.5]
        \fill[green!30] (0,0) -- (1,0) -- (0,1) -- cycle;
        \draw[dotted,gray] (0,0) -- (1,0) -- (0,1) -- cycle;
        \node at (0,0) {$0$};
        \node at (1,0) {$-c$};
        \node at (0,1) {$+c$};

        \node at (2,0) {$\rightarrow$};

        \fill[green!30] (3,0) -- (4,0) -- (4,1) -- (3,1) -- cycle;
        \draw[dotted,gray] (3,0) -- (5,0) -- (3,2) -- cycle;
        \draw[dotted,gray] (4,0) -- (4,1) -- (3,1) -- cycle;
        \node at (3,0) {$0$};
        \node at (4,0) {$-c$};
        \node at (3,1) {$+c$};
        \node at (4,1) {$0$};

        \node at (6,0) {$\rightarrow$};

        \fill[green!30] (7,0) -- (8,0) -- (8,2) -- (9,2) -- (9,1) -- (7,1) -- cycle;
        \draw[dotted,gray] (7,0) -- (9,0) -- (9,2) -- (7,2) -- cycle;
        \draw[dotted,gray] (8,0) -- (7,1) -- (9,1) -- (8,2);
        \draw[dotted,gray] (9,0) -- (7,2) -- cycle;
        \node at (7,0) {$0$};
        \node at (8,1) {$0$};
        \node at (9,2) {$0$};
        \node at (8,0) {$-c$};
        \node at (9,1) {$-c$};
        \node at (7,1) {$+c$};
        \node at (8,2) {$+c$};
    	\end{scope}
    \end{tikzpicture}
    }
    \caption{Lines alternating $0$'s and opposite pairs $+c$ / $-c$}
    \label{fig:persistence-alt}
\end{figure}
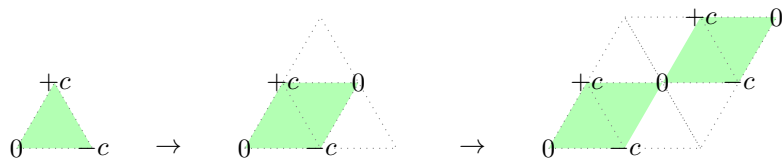

\begin{lemma}
Tiles with one $0$ remain in the corner holding that $0$
upon repeated applications of $\sigma$.
\end{lemma}

\begin{proof}
See figure~\ref{fig:persistence-one}.
\end{proof}

\begin{figure}[H]
    \centering
    {
    \begin{tikzpicture}
    	\begin{scope}[yscale=.87,xslant=.5]
        \fill[green!30] (0,0) -- (1,0) -- (0,1) -- cycle;
        \draw[dotted,gray] (0,0) -- (1,0) -- (0,1) -- cycle;
        \node at (0,0) {$0$};
        \node at (1,0) {$c$};
        \node at (0,1) {$d$};

        \node at (2,0) {$\rightarrow$};

        \fill[green!30] (3,0) -- (4,0) -- (3,1) -- cycle;
        \draw[dotted,gray] (3,0) -- (5,0) -- (3,2) -- cycle;
        \draw[dotted,gray] (4,0) -- (4,1) -- (3,1) -- cycle;
        \node at (3,0) {$0$};
        \node at (4,0) {$c$};
        \node at (3,1) {$d$};
    	\end{scope}
    \end{tikzpicture}
    }
    \caption{Persistence of $0$'s at corners}
    \label{fig:persistence-one}
\end{figure}
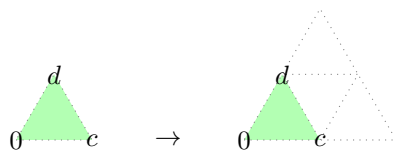

\begin{lemma}
$\forall k \in \mathbb{N}$,
$\forall t \in \{ \bigtriangleup 100, \bigtriangleup 010, \bigtriangleup 001,
                  \bigtriangledown 100, \bigtriangledown 010, \bigtriangledown 001 \}$,
the values at three corners forming an equilateral triangle
in $\sigma^k(t)$ with the same center as $\sigma^k(t)$
cannot all equal $0$ (see figure~\ref{fig:triangle-0}).
\end{lemma}

\begin{proof}
If this was the case, say for $\sigma^k(\bigtriangleup 100)$,
then, by symmetry, the values at the same three corners
in $\sigma^k(\bigtriangleup 010)$
and in $\sigma^k(\bigtriangleup 001)$
would also equal $0$.
Due to additivity ($\bigtriangleup xyz = x \cdot \bigtriangleup 100 + y \cdot \bigtriangleup 010 + z \cdot \bigtriangleup 001$),
the values at the same three corners
in any upward $k$-supertile would equal $0$.
This is a contradiction as from lemma~\ref{lemma:unique},
we can force a nonzero value at one of these three corners.

The reasoning for downward supertiles is similar.
\end{proof}

\begin{figure}[H]
    \centering
    {
    \begin{tikzpicture}
    	\begin{scope}[yscale=.87,xslant=.5]
        \draw (0,0) -- (7,0) -- (0,7) -- cycle;
        \draw[dotted] (2,1) -- (4,2) -- (1,4) -- cycle;

        \node[below,scale=2] at (3.5,0) {$\sigma^k(t)$};

        \node[scale=.5] at (2,1) [circle,fill=red] {};
        \node[scale=.5] at (4,2) [circle,fill=green] {};
        \node[scale=.5] at (1,4) [circle,fill=blue] {};

        \node[scale=.5] at (7/3,7/3) [circle,fill=gray] {};
    	\end{scope}
    \end{tikzpicture}
    }
    \caption{Three corners forming an equilateral triangle with the same center as $\sigma^k(t)$}
    \label{fig:triangle-0}
\end{figure}
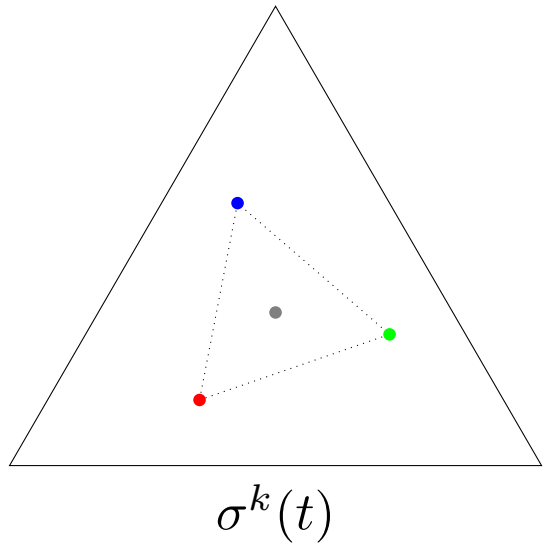

\subsection{Irreducibility}
\label{sec:sigma-prop-irreducibility}

\begin{theo}
$(T^*, \sigma)$ is irreducible.
\end{theo}

\begin{proof}
$\forall t, t' \in T$,
let the relation $t \sim t'$
mean that $\exists k \in \mathbb{N}$ such that $t'$ appears in $\sigma^k(t)$.

The relation \quotes{$\sim$} is naturally reflexive and transitive:
\begin{itemize}
\item $t$ appears in $\sigma^0(t)$,
\item if $t'$ appears in $\sigma^k(t)$
      and $t''$ appears in $\sigma^{k'}(t')$,
      then $t''$ appears in $\sigma^{k+k'}(t)$.
\end{itemize}

\begin{lemma}
\label{lemma:irr1}
$\forall x,y,z \in \mathbb{F}_p$,
$\bigtriangleup xyz \sim \bigtriangleup(x,x+y,x+z)$.
\end{lemma}

$\bigtriangleup(x,x+y,x+z)$ appears in $\sigma(\bigtriangleup xyz)$
(see figure~\ref{fig:subst-up}).

\begin{lemma}
\label{lemma:irr2}
$\forall x \in \mathbb{F}_p^*$,
$\forall d,y,z \in \mathbb{F}_p$,
$\bigtriangleup xyz \sim \bigtriangleup(x,d+y,d+z)$.
\end{lemma}

Applying lemma~\ref{lemma:irr1} $k$ times,
we have $\bigtriangleup xyz \sim \bigtriangleup(x,kx+y,kx+z)$.
As $x$ is invertible, $k \rightarrow kx$ runs through $\mathbb{F}_p$.

\begin{lemma}
\label{lemma:irr3}
$\forall y \in \mathbb{F}_p^*$,
$\forall d,x,z \in \mathbb{F}_p$,
$\bigtriangleup xyz \sim \bigtriangleup(d+x,y,d+z)$.
\end{lemma}

Similar to lemma~\ref{lemma:irr2}.

\begin{lemma}
\label{lemma:irr4}
$\forall z \in \mathbb{F}_p^*$,
$\forall d,x,y \in \mathbb{F}_p$,
$\bigtriangleup xyz \sim \bigtriangleup(d+x,d+y,z)$.
\end{lemma}

Similar to lemma~\ref{lemma:irr2}.

\begin{lemma}
\label{lemma:irr5}
$\forall x,y,z \in \mathbb{F}_p$,
when $(y,z) \ne (0,0)$, $\bigtriangleup xyz \sim \bigtriangleup(x+2,y,z)$.
\end{lemma}

To show $\bigtriangleup xyz \sim \bigtriangleup(x+2,y,z)$,
we will apply each of the lemmas \ref{lemma:irr3}, \ref{lemma:irr4} and \ref{lemma:irr5} once,
in some appropriate order, based on the following identity:
\begin{equation}
    \bigtriangleup 200
    = \bigtriangleup 101
    + \bigtriangleup 110
    - \bigtriangleup 011
\end{equation}

Figure~\ref{fig:irrecucibility-cases} describes the different
cases to consider and table~\ref{tab:irrecucibility-order-lemmas}
indicates an appropriate order of lemmas for each case.

\begin{figure}[H]
    \centering
    \resizebox{10cm}{!}
    {
    \begin{tikzpicture}

        \node at (-1.5,0.5) {\Huge $y = 0$};
        \node at (-1.5,1.5) {\Huge $y = 1$};
        \node at (-1.5,3.5) {\Huge $y \ne 0, 1$};

        \node[rotate=-45, above right] at (.25,-.5) {\huge $x = 0$};
        \node[rotate=-45, above right] at (2.25,-.5) {\huge $x \ne 0$};

        \node[rotate=-45, above right] at (6+.25,-.5) {\huge $x = 0$};
        \node[rotate=-45, above right] at (6+2.25,-.5) {\huge $x \ne 0, p-1$};
        \node[rotate=-45, above right] at (6+4.25,-.5) {\huge $x = p-1$};

        \node[rotate=-45, above right] at (12+.25,-.5) {\huge $x = 0$};
        \node[rotate=-45, above right] at (12+2.25,-.5) {\huge $x \ne 0$};

        \fill[violet!20] (0,1) rectangle (1,5);
        \fill[green!20] (1,1) rectangle (5,5);
        \draw (1,1) -- (1,5);
        \draw[pattern=north west lines, pattern color=gray] (0,0) rectangle (5,1);
        \node at (2.5,5.5) {\Huge $z = 0$};
        \draw[thick] (0,0) rectangle (5,5);
        \node at (.5,3) {\Huge d};
        \node at (3,3) {\Huge b};

        \fill[yellow!20] (7,0) rectangle (11,1);
        \fill[yellow!20] (7,2) rectangle (11,5);
        \fill[blue!20] (6,1) rectangle (7,5);
        \fill[blue!20] (7,1) rectangle (10,2);
        \fill[violet!20] (10,1) rectangle (11,2);
        \fill[red!20] (6,0) rectangle (7,1);
        \draw (7,0) -- (7,1);
        \draw (6,1) -- (11,1);
        \draw (10,1) -- (10,2);
        \draw (11,2) -- (7,2) -- (7,5);
        \node at (8.5,5.5) {\Huge $z = 1$};
        \draw[thick] (6,0) rectangle (11,5);
        \node at (9,3.5) {\Huge a};
        \node at (9,.5) {\Huge a};
        \node at (6.5,.5) {\Huge e};
        \node at (6.5,1.5) {\Huge c};
        \node at (10.5,1.5) {\Huge d};

        \fill[yellow!20] (13,0) rectangle (17,1);
        \fill[yellow!20] (13,2) rectangle (17,5);
        \fill[blue!20] (12,1) rectangle (13,5);
        \fill[green!20] (13,1) rectangle (17,2);
        \fill[red!20] (12,0) rectangle (13,1);
        \draw (13,0) -- (13,5);
        \draw (12,1) -- (17,1);
        \draw (13,2) -- (17,2);
        \node at (14.5,5.5) {\Huge $z \ne 0, 1$};
        \draw[thick] (12,0) rectangle (17,5);
        \node at (15,3.5) {\Huge a};
        \node at (15,.5) {\Huge a};
        \node at (15,1.5) {\Huge b};
        \node at (12.5,3) {\Huge c};
        \node at (12.5,.5) {\Huge e};

    \end{tikzpicture}
    }
    \caption{Different cases for lemma \ref{lemma:irr5}}
    \label{fig:irrecucibility-cases}
\end{figure}
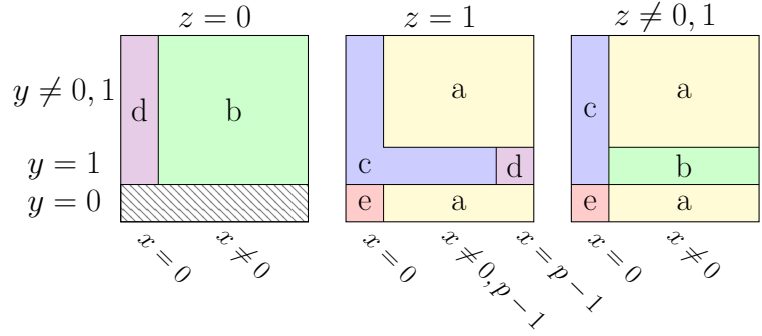

\begin{table} [!htb]
\begin{tabular}{|c|l|}
    \hline
    \rule[-1ex]{0pt}{2.5ex} Case & Order of lemmas \\
    \hline
    \rule[-1ex]{0pt}{2.5ex} a & $\bigtriangleup xyz \stackrel{\ref{lemma:irr2}}{\sim} \bigtriangleup(x,y - 1,z - 1) \stackrel{\ref{lemma:irr3}}{\sim} \bigtriangleup(x + 1,y - 1,z) \stackrel{\ref{lemma:irr4}}{\sim} \bigtriangleup(x + 2,y,z)$ \\
    \rule[-1ex]{0pt}{2.5ex} b & $\bigtriangleup xyz \stackrel{\ref{lemma:irr2}}{\sim} \bigtriangleup(x,y - 1,z - 1) \stackrel{\ref{lemma:irr4}}{\sim} \bigtriangleup(x + 1,y,z - 1) \stackrel{\ref{lemma:irr3}}{\sim} \bigtriangleup(x + 2,y,z)$ \\
    \rule[-1ex]{0pt}{2.5ex} c & $\bigtriangleup xyz \stackrel{\ref{lemma:irr3}}{\sim} \bigtriangleup(x + 1,y,z + 1) \stackrel{\ref{lemma:irr2}}{\sim} \bigtriangleup(x + 1,y - 1,z) \stackrel{\ref{lemma:irr4}}{\sim} \bigtriangleup(x + 2,y,z)$ \\
    \rule[-1ex]{0pt}{2.5ex} d & $\bigtriangleup xyz \stackrel{\ref{lemma:irr3}}{\sim} \bigtriangleup(x + 1,y,z + 1) \stackrel{\ref{lemma:irr4}}{\sim} \bigtriangleup(x + 2,y + 1,z + 1) \stackrel{\ref{lemma:irr2}}{\sim} \bigtriangleup(x + 2,y,z)$ \\
    \rule[-1ex]{0pt}{2.5ex} e & $\bigtriangleup xyz \stackrel{\ref{lemma:irr4}}{\sim} \bigtriangleup(x + 1,y + 1,z) \stackrel{\ref{lemma:irr3}}{\sim} \bigtriangleup(x + 2,y + 1,z + 1) \stackrel{\ref{lemma:irr2}}{\sim} \bigtriangleup(x + 2,y,z)$ \\
    \hline
\end{tabular}
\caption{Order of lemmas for lemma \ref{lemma:irr5}}
\label{tab:irrecucibility-order-lemmas}
\end{table}

\begin{lemma}
\label{lemma:irr6}
$\forall x,x',y,z \in \mathbb{F}_p$,
when $(y,z) \ne (0,0)$, $\bigtriangleup xyz \sim \bigtriangleup x'yz$.
\end{lemma}

Applying lemma~\ref{lemma:irr5} $k$ times,
we have $\bigtriangleup xyz \sim \bigtriangleup(x+2k,y,z)$.
As $2$ is invertible, $k \rightarrow x+2k$ runs through $\mathbb{F}_p$.

\begin{lemma}
\label{lemma:irr7}
$\forall x,y,y',z \in \mathbb{F}_p$,
when $(x,z) \ne (0,0)$, $\bigtriangleup xyz \sim \bigtriangleup xy'z$.
\end{lemma}

Similar to lemma \ref{lemma:irr6}.

\begin{lemma}
\label{lemma:irr8}
$\forall x,y,z,z' \in \mathbb{F}_p$,
when $(x,y) \ne (0,0)$, $\bigtriangleup xyz \sim \bigtriangleup xyz'$.
\end{lemma}

Similar to lemma \ref{lemma:irr6}.

\begin{lemma}
\label{lemma:irr9}
$\forall x,y,z \in \mathbb{F}_p$,
when $(x,y,z) \ne (0,0,0)$,
$\bigtriangleup xyz \sim \bigtriangleup 111$
and $\bigtriangleup 111 \sim \bigtriangleup xyz$.
\end{lemma}

If $x \ne 0$,
then $\bigtriangleup xyz \sim \bigtriangleup xy1
                            \sim \bigtriangleup x11
                            \sim \bigtriangleup 111$
and $\bigtriangleup 111 \sim \bigtriangleup x11
                           \sim \bigtriangleup xy1
                           \sim \bigtriangleup xyz$.

The other cases ($y \ne 0$ and $z \ne 0$) are similar.

\begin{lemma}
\label{lemma:irr10}
$\forall x,y,z \in \mathbb{F}_p$,
when $(x,y,z) \ne (0,0,0)$,
$\bigtriangledown xyz \sim \bigtriangledown 111$
and $\bigtriangledown 111 \sim \bigtriangledown xyz$.
\end{lemma}

Similar to lemma \ref{lemma:irr9}.

\vspace{\baselineskip}

Noting that $\bigtriangleup 111 \sim \bigtriangledown 222$
and that $\bigtriangledown 111 \sim \bigtriangleup 222$,
and using lemmas \ref{lemma:irr9} and \ref{lemma:irr10},
we can now prove that $\forall t, t' \in T^*$, $t \sim t'$:
as depicted in figure~\ref{fig:irrecucibility-conclusion},
we have a \quotes{$\sim$} path from any tile $t$ to any tile $t'$.
\end{proof}

\begin{figure}[H]
    \centering
    \resizebox{10cm}{!}
    {
    \begin{tikzpicture}

        \node (A) at (0,2) {$\bigtriangleup xyz$};
        \node (B) at (2,2) {$\bigtriangleup 222$};
        \node (C) at (4,2) {$\bigtriangledown x'y'z'$};

        \node (D) at (1,1) {$\bigtriangleup 111$};
        \node (E) at (3,1) {$\bigtriangledown 111$};

        \node (F) at (0,0) {$\bigtriangleup x'y'z'$};
        \node (G) at (2,0) {$\bigtriangledown 222$};
        \node (H) at (4,0) {$\bigtriangledown xyz$};

        \draw[<->] (A) -- (D);
        \draw[<->] (D) -- (F);
        \draw[<->] (C) -- (E);
        \draw[<->] (E) -- (H);
        \draw[<->] (D) -- (B);
        \draw[<->] (E) -- (G);
        \draw[->] (D) -- (G);
        \draw[->] (E) -- (B);

    \end{tikzpicture}
    }
    \caption{$\forall t, t' \in T^*$, $t \sim t'$}
    \label{fig:irrecucibility-conclusion}
\end{figure}
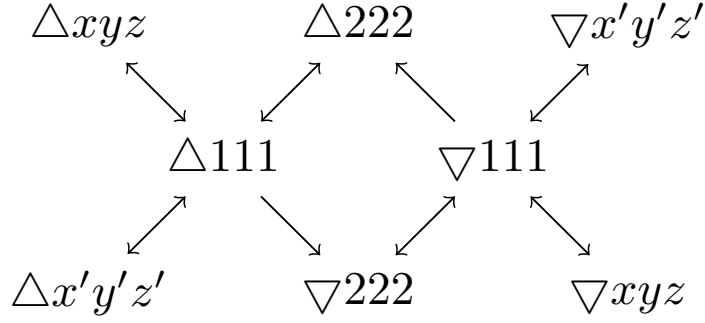

\subsection{Primitivity}
\label{sec:sigma-prop-primitivity}

\begin{theo}
$(T^*, \sigma)$ is primitive.
\end{theo}

\begin{proof}
$\forall t, t' \in T^*$,
let $d(t, t')$ denote the least $k \in \mathbb{N}$
such that $t'$ appears in $\sigma^k(t)$;
$d$ is well defined as $(T^*, \sigma)$ is irreducible.

The tile $\bigtriangleup 011$ is particular as $\bigtriangleup 011$ appears in $\sigma(\bigtriangleup 011)$.

Let:
\begin{align}
    m = \max_{t \in T^*} d(t, \bigtriangleup 011)    \\
    n = \max_{t \in T^*} d(\bigtriangleup 011, t)
\end{align}

$\forall t, t' \in T^*$,
$t'$ appears in $\sigma^{m+n}(t)$:
\begin{equation}
    t
    \underbrace{
        \xrightarrow{\sigma^{d(t, \bigtriangleup 011)}}
        \bigtriangleup 011
        \xrightarrow{\sigma^{m - d(t, \bigtriangleup 011) + n - d(\bigtriangleup 011, t')}}
        \bigtriangleup 011
        \xrightarrow{\sigma^{d(\bigtriangleup 011, t')}}
    }_{\xrightarrow{\sigma^{m+n}}}
    t'
\end{equation}
\end{proof}

Figures~\ref{fig:patch3} and \ref{fig:patch5}
present examples of patches that appear in all sufficiently large supertiles
respectively for $p = 3$ and $p = 5$.

\begin{figure}[H]
    \begin{center}
    \vstretch{1.732}{
        \includegraphics[height=5cm, keepaspectratio]{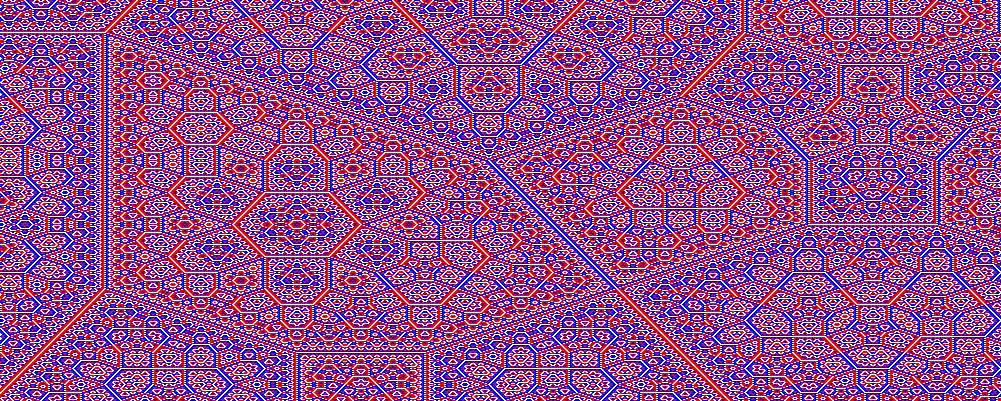}
    }
    \caption{Colored representation of a patch (for $p = 3$)}
    \label{fig:patch3}
    \end{center}
\end{figure}

\begin{figure}[H]
    \begin{center}
    \vstretch{1.732}{
        \includegraphics[height=5cm, keepaspectratio]{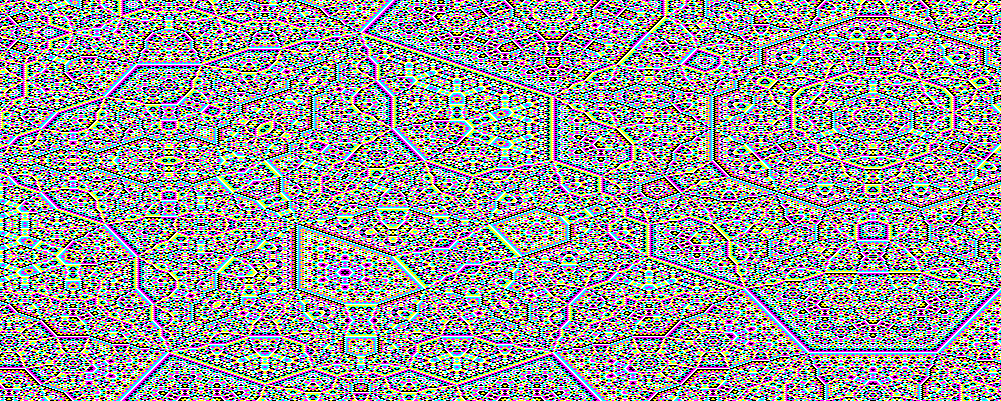}
    }
    \caption{Colored representation of patch (for $p = 5$)}
    \label{fig:patch5}
    \end{center}
\end{figure}

\section{\texorpdfstring{Tiling $S_t$}{Tiling St}}
\label{sec:sigma-tiling-s}

\subsection{Construction}
\label{sec:sigma-tiling-s-cons}

\begin{definition}
$\forall t \in T$,
$\forall k \in \mathbb{N}$,
let:
\begin{equation}
    s_t(k) = \sigma^k(\alpha^{-k}(t))
\end{equation}
\end{definition}

\begin{lemma}
$\forall t \in T$,
the sequence $\{s_t(k)\}_{k \in \mathbb{N}}$ defines a family of nested supertiles
sharing the same central tile $t$.
\end{lemma}

\begin{proof}
$\forall k \in \mathbb{N}$:
the central $k$-supertile of $s_t(k+1)$ satisfies:
\begin{equation}
    [\alpha(s_t(k+1))] = [\alpha^{-k-1}(t)][\alpha]
                       = [t][\alpha^{-k-1}][\alpha]
                       = [t][\alpha^{-k}]
                       = [s_t(k)]
\end{equation}
so $s_t(k)$ is \quotes{nested} in the center of $s_t(k+1)$.

Also, the central tile of $s_t(k)$ (at position $\alpha^k$) satisfies:
\begin{equation}
    [\alpha^k(s_t)] = [s_t] \cdot [\alpha^k]
                    = [t] \cdot [\alpha^{-k}] \cdot [\alpha^k]
                    = [t]
\end{equation}
\end{proof}

\begin{lemma}
\label{lemma:s-limit}
$\forall t \in T$,
the sequence $\{s_t(k)\}_{k \in \mathbb{N}}$ has a limit, say $S_t$,
that constitutes a tiling of the plane.
\end{lemma}

\begin{proof}
The sequence $\{s_t(k)\}_{k \in \mathbb{N}}$ is a family of nested supertiles,
whose incircle radius tends to infinity,
hence, by arguments applied in Smilansky and Solomon~\cite{multiscale},
its limit exists in the sense of \quotes{Chabauty–Fell topology}
and corresponds to a tiling of the plane.
\end{proof}

Figure \ref{fig:stationary} presents the first steps
of the construction of $S_{\bigtriangleup 021}$ (for $p = 3$);
the thick lines correspond to the boundaries of the first nested supertiles;
the color scheme is the one used in section~\ref{sec:sigma-cons-examples}.

\begin{figure}[H]
    \centering
    \resizebox{12cm}{!}
    {
    \begin{tikzpicture}
    	\begin{scope}[yscale=.87,xslant=.5]
        \input{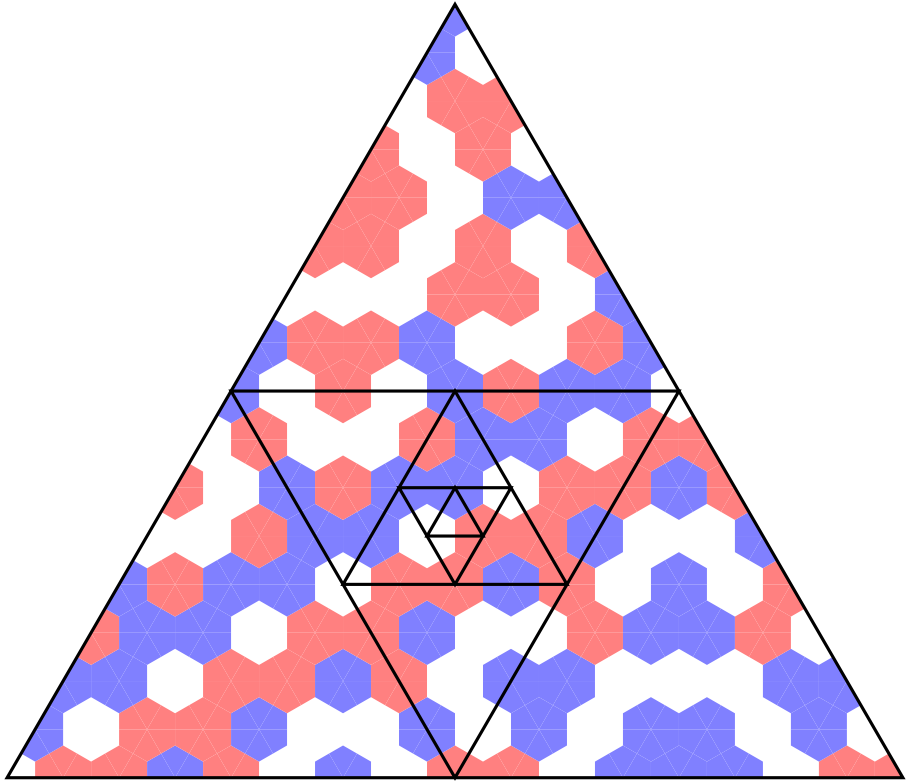}
    	\end{scope}
    \end{tikzpicture}
    }
    \caption{Colored representation of $s_{\bigtriangleup 021}(k)$
             (as nested supertiles) for $k = 0 \twodots 4$ (for $p = 3$)}
    \label{fig:stationary}
\end{figure}

\begin{remark}
$\forall t \in T$,
$S_t$ inherits the symmetries of $t$.
\end{remark}

\begin{remark}
$\forall m \in \mathbb{F}_p$,
$\forall t, t' \in U$ (or $\forall t, t' \in D$):
\begin{equation}
    S_t + S_{t'} = S_{t+t'}
\end{equation}
\begin{equation}
    m \cdot S_t = S_{mt}
\end{equation}
\end{remark}

\subsection{Automaticity}
\label{sec:sigma-tiling-s-automaticity}

For an in-depth introduction to automaticity, see Allouche and Shallit \cite{automatic}.

\begin{definition}
Let $M$ be the deterministic finite automaton with the following characteristics:
\begin{itemize}
\item the set of states is $Q = \{+1, -1\} \times \mathbb{F}_p^{3 \times 3}$,
\item the input alphabet is $\Sigma$,
\item the initial state is $q_0 = (+1, I_3)$,
\item the transitions from state $(f, m)$ are performed as follows:\\
\begin{tabular}{|c|l|}
    \hline
    \rule[-1ex]{0pt}{2.5ex} Input symbol & Output state \\
    \hline
    \rule[-1ex]{0pt}{2.5ex} $\alpha$ & $(+f, [\alpha]^{-1} \cdot m \cdot [\alpha])$ \\
    \rule[-1ex]{0pt}{2.5ex} $\beta$  & $(-f, [\alpha]^{-1} \cdot m \cdot [\beta])$ \\
    \rule[-1ex]{0pt}{2.5ex} $\gamma$ & $(-f, [\alpha]^{-1} \cdot m \cdot [\gamma])$ \\
    \rule[-1ex]{0pt}{2.5ex} $\delta$ & $(-f, [\alpha]^{-1} \cdot m \cdot [\delta])$ \\
    \hline
\end{tabular}
\item $\forall w \in \Sigma^*$,
      the final state of the automaton after processing the input $w$,
      say $(f, m)$, encodes a tile $t'$ as follows:
\begin{itemize}
\item if $f = +1$ then $t$ and $t'$ have the same orientation,
\item if $f = -1$ then $t$ and $t'$ have opposite orientations,
\item also:
\begin{equation}
    [t'] = [t] \cdot [m]
\end{equation}
\end{itemize}
\end{itemize}
\end{definition}

\begin{theo}
$\forall t \in T$,
the tiling $S_t$ is $4$-automatic and can be computed with $M$.
\end{theo}

\begin{proof}
The incircle radius of nested prototiles $\{ s_t(k) \}_{k \in \mathbb{N}}$
increases without limit, so for any tile $t'$ in the plane,
$\exists k \in \mathbb{N}$ such that $t'$ belongs to $s_t(k)$.

As seen in theorem~\ref{theo:tile-position},
the position of $t'$ within $s_t(k)$
is characterized by a word $w \in \Sigma^k$.

Suppose that the final state of $M$ after processing $w$ equals $(f, m)$.

By construction, $m$ satisfies:
\begin{equation}
    m = \underbrace{[\alpha]^{-1} \ldots [\alpha]^{-1}}_{k \mathrm{\ times}}
        \cdot
        I_3
        \cdot
        [w_0] \ldots [w_{k-1}]
      = [\alpha]^{-k} \cdot [w]
\end{equation}

Also, $t'$ is at position $w$ within $s_t(k)$, so:
\begin{equation}
    [t'] = [s_t(k)] \cdot [w]
         = [\alpha^{-k}(t)] \cdot [w]
         = [t] \cdot [\alpha]^{-k} \cdot [w]
         = [t] \cdot m
\end{equation}
So $m$ gives access to $[t']$.

\vspace{\baselineskip}

By construction, $f = +1$ if and only if $w$ contains
an even cumulative number of $\beta$'s, $\gamma$'s and $\delta$'s.

From lemma~\ref{lemma:flip}:
\begin{itemize}
\item $t$ has the same orientation as $s_t(k)$
      if and only if $k$ is even ($t$ is at position $\alpha^k$),
\item $t'$ has the same orientation as $s_t(k)$
      if and only if $w$ contains an even number of $\alpha$'s,
\item so $t$ and $t'$ have the same orientation
      if and only if $k$ and the number of $\alpha$'s in $w$ have the same parity,
\item this is equivalent to say that
      the cumulative number of $\beta$'s, $\gamma$'s and $\delta$'s in $w$ is even.
\end{itemize}
So $f$ gives access to the orientation of $t'$.
\end{proof}

\begin{remark}
If $k$ is not minimal,
then the input will have extra leading $\alpha$'s;
this will not alter the final state as
$q_0 \xrightarrow{\alpha} q_0$.
\end{remark}

\begin{remark}
$\#\Sigma = 4$, so $S_t$ is \quotes{4}-automatic.
\end{remark}

\begin{remark}
Once the orientation of the central tile is fixed,
we can interpret the position $w$ of any tile in the plane
as a base-$4$ integer,
by assigning the values $0$, $1$, $2$ and $3$
to $\alpha$, $\beta$, $\gamma$ and $\delta$ respectively
(see figure~\ref{fig:plane-coord}).
\end{remark}

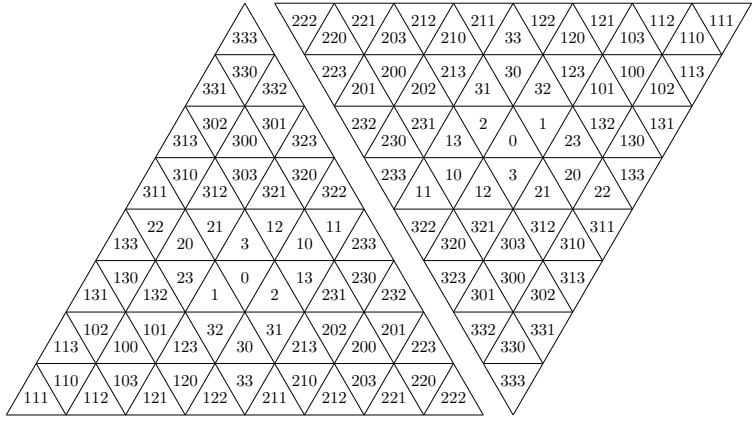
\begin{figure}[H]
    \centering
    \resizebox{10cm}{!}
    {
    \begin{tikzpicture}
    	\begin{scope}[yscale=.87,xslant=.5]
        \node[scale=.8] at (37/6,-35/6) {$0$};\node[scale=.8] at (35/6,-37/6) {$1$};\node[scale=.8] at (41/6,-37/6) {$2$};\node[scale=.8] at (35/6,-31/6) {$3$};\node[scale=.8] at (41/6,-31/6) {$10$};\node[scale=.8] at (43/6,-29/6) {$11$};\node[scale=.8] at (37/6,-29/6) {$12$};\node[scale=.8] at (43/6,-35/6) {$13$};\node[scale=.8] at (29/6,-31/6) {$20$};\node[scale=.8] at (31/6,-29/6) {$21$};\node[scale=.8] at (25/6,-29/6) {$22$};\node[scale=.8] at (31/6,-35/6) {$23$};\node[scale=.8] at (41/6,-43/6) {$30$};\node[scale=.8] at (43/6,-41/6) {$31$};\node[scale=.8] at (37/6,-41/6) {$32$};\node[scale=.8] at (43/6,-47/6) {$33$};\node[scale=.8] at (29/6,-43/6) {$100$};\node[scale=.8] at (31/6,-41/6) {$101$};\node[scale=.8] at (25/6,-41/6) {$102$};\node[scale=.8] at (31/6,-47/6) {$103$};\node[scale=.8] at (25/6,-47/6) {$110$};\node[scale=.8] at (23/6,-49/6) {$111$};\node[scale=.8] at (29/6,-49/6) {$112$};\node[scale=.8] at (23/6,-43/6) {$113$};\node[scale=.8] at (37/6,-47/6) {$120$};\node[scale=.8] at (35/6,-49/6) {$121$};\node[scale=.8] at (41/6,-49/6) {$122$};\node[scale=.8] at (35/6,-43/6) {$123$};\node[scale=.8] at (25/6,-35/6) {$130$};\node[scale=.8] at (23/6,-37/6) {$131$};\node[scale=.8] at (29/6,-37/6) {$132$};\node[scale=.8] at (23/6,-31/6) {$133$};\node[scale=.8] at (53/6,-43/6) {$200$};\node[scale=.8] at (55/6,-41/6) {$201$};\node[scale=.8] at (49/6,-41/6) {$202$};\node[scale=.8] at (55/6,-47/6) {$203$};\node[scale=.8] at (49/6,-47/6) {$210$};\node[scale=.8] at (47/6,-49/6) {$211$};\node[scale=.8] at (53/6,-49/6) {$212$};\node[scale=.8] at (47/6,-43/6) {$213$};\node[scale=.8] at (61/6,-47/6) {$220$};\node[scale=.8] at (59/6,-49/6) {$221$};\node[scale=.8] at (65/6,-49/6) {$222$};\node[scale=.8] at (59/6,-43/6) {$223$};\node[scale=.8] at (49/6,-35/6) {$230$};\node[scale=.8] at (47/6,-37/6) {$231$};\node[scale=.8] at (53/6,-37/6) {$232$};\node[scale=.8] at (47/6,-31/6) {$233$};\node[scale=.8] at (29/6,-19/6) {$300$};\node[scale=.8] at (31/6,-17/6) {$301$};\node[scale=.8] at (25/6,-17/6) {$302$};\node[scale=.8] at (31/6,-23/6) {$303$};\node[scale=.8] at (25/6,-23/6) {$310$};\node[scale=.8] at (23/6,-25/6) {$311$};\node[scale=.8] at (29/6,-25/6) {$312$};\node[scale=.8] at (23/6,-19/6) {$313$};\node[scale=.8] at (37/6,-23/6) {$320$};\node[scale=.8] at (35/6,-25/6) {$321$};\node[scale=.8] at (41/6,-25/6) {$322$};\node[scale=.8] at (35/6,-19/6) {$323$};\node[scale=.8] at (25/6,-11/6) {$330$};\node[scale=.8] at (23/6,-13/6) {$331$};\node[scale=.8] at (29/6,-13/6) {$332$};\node[scale=.8] at (23/6,-7/6) {$333$};\draw(7/2,-17/2) -- (23/2,-17/2);
\draw(7/2,-17/2) -- (7/2,-1/2);
\draw(23/2,-17/2) -- (7/2,-1/2);
\draw(7/2,-15/2) -- (21/2,-15/2);
\draw(9/2,-17/2) -- (9/2,-3/2);
\draw(21/2,-17/2) -- (7/2,-3/2);
\draw(7/2,-13/2) -- (19/2,-13/2);
\draw(11/2,-17/2) -- (11/2,-5/2);
\draw(19/2,-17/2) -- (7/2,-5/2);
\draw(7/2,-11/2) -- (17/2,-11/2);
\draw(13/2,-17/2) -- (13/2,-7/2);
\draw(17/2,-17/2) -- (7/2,-7/2);
\draw(7/2,-9/2) -- (15/2,-9/2);
\draw(15/2,-17/2) -- (15/2,-9/2);
\draw(15/2,-17/2) -- (7/2,-9/2);
\draw(7/2,-7/2) -- (13/2,-7/2);
\draw(17/2,-17/2) -- (17/2,-11/2);
\draw(13/2,-17/2) -- (7/2,-11/2);
\draw(7/2,-5/2) -- (11/2,-5/2);
\draw(19/2,-17/2) -- (19/2,-13/2);
\draw(11/2,-17/2) -- (7/2,-13/2);
\draw(7/2,-3/2) -- (9/2,-3/2);
\draw(21/2,-17/2) -- (21/2,-15/2);
\draw(9/2,-17/2) -- (7/2,-15/2);
\node[scale=.8] at (28/3,-19/6) {$0$};\node[scale=.8] at (29/3,-17/6) {$1$};\node[scale=.8] at (26/3,-17/6) {$2$};\node[scale=.8] at (29/3,-23/6) {$3$};\node[scale=.8] at (26/3,-23/6) {$10$};\node[scale=.8] at (25/3,-25/6) {$11$};\node[scale=.8] at (28/3,-25/6) {$12$};\node[scale=.8] at (25/3,-19/6) {$13$};\node[scale=.8] at (32/3,-23/6) {$20$};\node[scale=.8] at (31/3,-25/6) {$21$};\node[scale=.8] at (34/3,-25/6) {$22$};\node[scale=.8] at (31/3,-19/6) {$23$};\node[scale=.8] at (26/3,-11/6) {$30$};\node[scale=.8] at (25/3,-13/6) {$31$};\node[scale=.8] at (28/3,-13/6) {$32$};\node[scale=.8] at (25/3,-7/6) {$33$};\node[scale=.8] at (32/3,-11/6) {$100$};\node[scale=.8] at (31/3,-13/6) {$101$};\node[scale=.8] at (34/3,-13/6) {$102$};\node[scale=.8] at (31/3,-7/6) {$103$};\node[scale=.8] at (34/3,-7/6) {$110$};\node[scale=.8] at (35/3,-5/6) {$111$};\node[scale=.8] at (32/3,-5/6) {$112$};\node[scale=.8] at (35/3,-11/6) {$113$};\node[scale=.8] at (28/3,-7/6) {$120$};\node[scale=.8] at (29/3,-5/6) {$121$};\node[scale=.8] at (26/3,-5/6) {$122$};\node[scale=.8] at (29/3,-11/6) {$123$};\node[scale=.8] at (34/3,-19/6) {$130$};\node[scale=.8] at (35/3,-17/6) {$131$};\node[scale=.8] at (32/3,-17/6) {$132$};\node[scale=.8] at (35/3,-23/6) {$133$};\node[scale=.8] at (20/3,-11/6) {$200$};\node[scale=.8] at (19/3,-13/6) {$201$};\node[scale=.8] at (22/3,-13/6) {$202$};\node[scale=.8] at (19/3,-7/6) {$203$};\node[scale=.8] at (22/3,-7/6) {$210$};\node[scale=.8] at (23/3,-5/6) {$211$};\node[scale=.8] at (20/3,-5/6) {$212$};\node[scale=.8] at (23/3,-11/6) {$213$};\node[scale=.8] at (16/3,-7/6) {$220$};\node[scale=.8] at (17/3,-5/6) {$221$};\node[scale=.8] at (14/3,-5/6) {$222$};\node[scale=.8] at (17/3,-11/6) {$223$};\node[scale=.8] at (22/3,-19/6) {$230$};\node[scale=.8] at (23/3,-17/6) {$231$};\node[scale=.8] at (20/3,-17/6) {$232$};\node[scale=.8] at (23/3,-23/6) {$233$};\node[scale=.8] at (32/3,-35/6) {$300$};\node[scale=.8] at (31/3,-37/6) {$301$};\node[scale=.8] at (34/3,-37/6) {$302$};\node[scale=.8] at (31/3,-31/6) {$303$};\node[scale=.8] at (34/3,-31/6) {$310$};\node[scale=.8] at (35/3,-29/6) {$311$};\node[scale=.8] at (32/3,-29/6) {$312$};\node[scale=.8] at (35/3,-35/6) {$313$};\node[scale=.8] at (28/3,-31/6) {$320$};\node[scale=.8] at (29/3,-29/6) {$321$};\node[scale=.8] at (26/3,-29/6) {$322$};\node[scale=.8] at (29/3,-35/6) {$323$};\node[scale=.8] at (34/3,-43/6) {$330$};\node[scale=.8] at (35/3,-41/6) {$331$};\node[scale=.8] at (32/3,-41/6) {$332$};\node[scale=.8] at (35/3,-47/6) {$333$};\draw(12,-1/2) -- (4,-1/2);
\draw(12,-1/2) -- (12,-17/2);
\draw(4,-1/2) -- (12,-17/2);
\draw(12,-3/2) -- (5,-3/2);
\draw(11,-1/2) -- (11,-15/2);
\draw(5,-1/2) -- (12,-15/2);
\draw(12,-5/2) -- (6,-5/2);
\draw(10,-1/2) -- (10,-13/2);
\draw(6,-1/2) -- (12,-13/2);
\draw(12,-7/2) -- (7,-7/2);
\draw(9,-1/2) -- (9,-11/2);
\draw(7,-1/2) -- (12,-11/2);
\draw(12,-9/2) -- (8,-9/2);
\draw(8,-1/2) -- (8,-9/2);
\draw(8,-1/2) -- (12,-9/2);
\draw(12,-11/2) -- (9,-11/2);
\draw(7,-1/2) -- (7,-7/2);
\draw(9,-1/2) -- (12,-7/2);
\draw(12,-13/2) -- (10,-13/2);
\draw(6,-1/2) -- (6,-5/2);
\draw(10,-1/2) -- (12,-5/2);
\draw(12,-15/2) -- (11,-15/2);
\draw(5,-1/2) -- (5,-3/2);
\draw(11,-1/2) -- (12,-3/2);
    	\end{scope}
    \end{tikzpicture}
    }
    \caption{Positions in the plane as base-$4$ integers
             (the central tile has position $0$)}
    \label{fig:plane-coord}
\end{figure}

\subsection{Self-similarity}
\label{sec:sigma-tiling-s-self-similarity}

\begin{theo}
$\forall t \in T^*$,
the tiling $S_t$ is self-similar.
\end{theo}

\begin{proof}
$\forall k \in \mathbb{N}$:
\begin{equation}
    \sigma(s_t(k)) = s_{\alpha(t)}(k+1)
\end{equation}
So, to the limit when $k \rightarrow \infty$:
\begin{equation}
    \sigma(S_t) = S_{\alpha(t)}
\end{equation}

Let $h_p$ be defined as:
\begin{equation}
    h_p = \left\{ \begin{array}{cl}
                3           & \mathrm{if\ } p = 3 \\
                ord_p(4)    & \mathrm{if\ } p > 3
          \end{array} \right.
\end{equation}

When $p = 3$:
\begin{equation}
    A^{2h_3} = A^6 = I_3
\end{equation}

$\forall n \in \mathbb{N}$, we have (in $\mathbb{Z}^{3 \times 3}$):
\begin{equation}
    A^{2n} = \frac{4^n+2}{3} \cdot I_3 + \frac{4^n-1}{3} \cdot A
\end{equation}

So when $p > 3$, in $F_p$:
\begin{equation}
    4^{2h_p} = 1
\end{equation}
And:
\begin{equation}
    A^{2h_p} = \frac{1+2}{3} \cdot I_3 + \frac{1-1}{3} \cdot A = I_3
\end{equation}

So,
$\forall k \in \mathbb{N}$:
\begin{equation}
    \alpha^{2h_p} = id_T
\end{equation}

Applying $\sigma^{2h_p}$ to $S_t$:
\begin{equation}
    \sigma^{2h_p}(S_t) = S_{\alpha^{2h_p}(t)} = S_t
\end{equation}

$S_t$ is self-similar with respect to $\sigma^{2h_p}$.
\end{proof}

\subsection{Nonperiodicity}
\label{sec:sigma-tiling-s-nonperiodicity}

\begin{theo}
\label{theo:nonperiod-3}
$\forall t \in T^*$,
the tiling $S_t$ is nonperiodic.
\end{theo}

\begin{proof}
As a consequence of primitivity and self-similarity,
$S_t$ contains all supertiles $\sigma^k(t')$ where $k \in \mathbb{N}$ and $t' \in T^*$.

\vspace{\baselineskip}

In particular,
$\forall k \in \mathbb{N}$,
$S_t$ contains $\sigma^k(\bigtriangleup 110)$.

\vspace{\baselineskip}

As depicted in figure~\ref{fig:nonperiodic-1}:
\begin{itemize}
\item $\sigma(\bigtriangleup 110)$ contains a diamond patch combining
      $\bigtriangleup 121$ and $\bigtriangledown 112$,
\item $\sigma^{k+1}(\bigtriangleup 110)$ contains a diamond patch combining
      $\bigtriangleup (k+1,k+2,1)$ and $\bigtriangledown (k+1,1,k+2)$,
\item we will fix $k = p-2$ to continue.
\end{itemize}

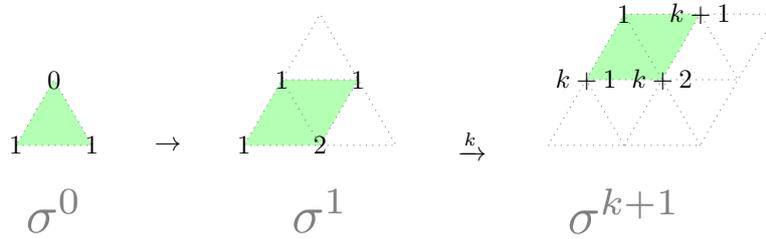
\begin{figure}[H]
    \centering
    {
    \begin{tikzpicture}
    	\begin{scope}[yscale=.87,xslant=.5]
        \node[scale=2,gray] at (1,-1) {$\sigma^0$};
        \node[scale=2,gray] at (4.5,-1) {$\sigma^1$};
        \node[scale=2,gray] at (8.5,-1) {$\sigma^{k+1}$};

        \fill[green!30] (0,0) -- (1,0) -- (0,1) -- cycle;
        \draw[dotted,gray] (0,0) -- (1,0) -- (0,1) -- cycle;
        \node at (0,0) {$1$};
        \node at (1,0) {$1$};
        \node at (0,1) {$0$};

        \node at (2,0) {$\rightarrow$};

        \fill[green!30] (3,0) -- (4,0) -- (4,1) -- (3,1) -- cycle;
        \draw[dotted,gray] (3,0) -- (5,0) -- (3,2) -- cycle;
        \draw[dotted,gray] (4,0) -- (4,1) -- (3,1) -- cycle;
        \node at (3,0) {$1$};
        \node at (3,1) {$1$};
        \node at (4,1) {$1$};
        \node at (4,0) {$2$};

        \node at (6,0) {$\xrightarrow{k}$};

        \fill[green!30] (7,1) -- (8,1) -- (8,2) -- (7,2) -- cycle;
        \draw[dotted,gray] (7,0) -- (9,0) -- (9,2) -- (7,2) -- cycle;
        \draw[dotted,gray] (8,0) -- (8,2);
        \draw[dotted,gray] (9,0) -- (7,2);
        \draw[dotted,gray] (8,0) -- (7,1) -- (9,1) -- (8,2);
        \node at (7,2) {$1$};
        \node at (8,2) {$k+1$};
        \node at (7,1) {$k+1$};
        \node at (8,1) {$k+2$};
    	\end{scope}
    \end{tikzpicture}
    }
    \caption{From $\bigtriangleup 110$ to diamond patch}
    \label{fig:nonperiodic-1}
\end{figure}

As depicted in figure~\ref{fig:nonperiodic-2}:
\begin{itemize}
\item $\sigma^{p-1}(\bigtriangleup 110)$ contains a diamond patch combining
      $\bigtriangleup (-1,0,1)$ and $\bigtriangledown (-1,1,0)$,
\item $\sigma^p(\bigtriangleup 110)$ contains a diamond patch combining
      $\bigtriangleup 011$ and $\bigtriangledown 011$.
\item $\sigma^{p+1}(\bigtriangleup 110)$ contains a hexagonal patch
      with $1$'s around a central $2$.
\end{itemize}

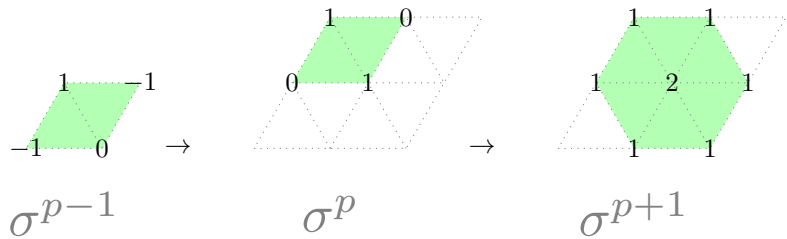
\begin{figure}[H]
    \centering
    {
    \begin{tikzpicture}
    	\begin{scope}[yscale=.87,xslant=.5]
        \node[scale=2,gray] at (1,-1) {$\sigma^{p-1}$};
        \node[scale=2,gray] at (4.5,-1) {$\sigma^p$};
        \node[scale=2,gray] at (8.5,-1) {$\sigma^{p+1}$};

        \fill[green!30] (0,0) -- (1,0) -- (1,1) -- (0,1) -- cycle;
        \draw[dotted,gray] (0,0) -- (1,0) -- (1,1) -- (0,1) -- cycle;
        \draw[dotted,gray] (1,0) -- (0,1);
        \node at (0,0) {$-1$};
        \node at (1,0) {$0$};
        \node at (0,1) {$1$};
        \node at (1,1) {$-1$};

        \node at (2,0) {$\rightarrow$};

        \fill[green!30] (3,1) -- (4,1) -- (4,2) -- (3,2) -- cycle;
        \draw[dotted,gray] (3,0) -- (5,0) -- (5,2) -- (3,2) -- cycle;
        \draw[dotted,gray] (4,0) -- (4,2);
        \draw[dotted,gray] (5,0) -- (3,2);
        \draw[dotted,gray] (4,0) -- (3,1) -- (5,1) -- (4,2);
        \node at (3,2) {$1$};
        \node at (4,2) {$0$};
        \node at (3,1) {$0$};
        \node at (4,1) {$1$};

        \node at (6,0) {$\rightarrow$};

        \fill[green!30] (8,0) -- (9,0) -- (9,1) -- (8,2) -- (7,2) -- (7,1) -- cycle;
        \draw[dotted,gray] (7,0) -- (9,0) -- (9,2) -- (7,2) -- cycle;
        \draw[dotted,gray] (8,0) -- (8,2);
        \draw[dotted,gray] (9,0) -- (7,2);
        \draw[dotted,gray] (8,0) -- (7,1) -- (9,1) -- (8,2);
        \node at (7,2) {$1$};
        \node at (8,2) {$1$};
        \node at (7,1) {$1$};
        \node at (8,1) {$2$};
        \node at (9,1) {$1$};
        \node at (8,0) {$1$};
        \node at (9,0) {$1$};
    	\end{scope}
    \end{tikzpicture}
    }
    \caption{From diamond patch to hexagonal patch}
    \label{fig:nonperiodic-2}
\end{figure}

As depicted in figure~\ref{fig:nonperiodic-3}:
\begin{itemize}
\item $\sigma^{k'+p+1}(\bigtriangleup 110)$ contains a hexagonal patch
      with a central $2$,
      surrounded by values $2k'+1$,
      themselves surrounded by values $2k'-1$ and $4k'-2$,
\item we fix $k' = \frac{p-1}{2}$ to continue.
\end{itemize}

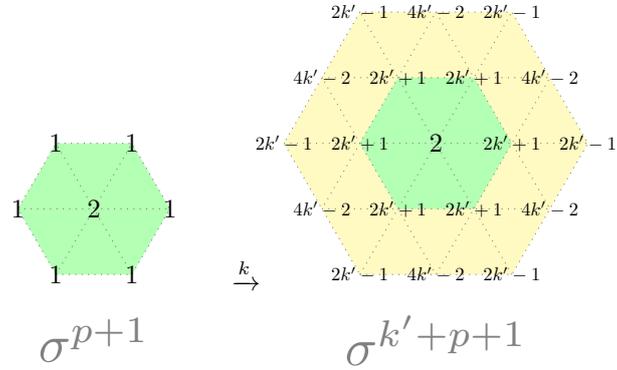
\begin{figure}[H]
    \centering
    {
    \begin{tikzpicture}
    	\begin{scope}[yscale=.87,xslant=.5]
        \node[scale=2,gray] at (2,-1) {$\sigma^{p+1}$};
        \node[scale=2,gray] at (6.5,-1) {$\sigma^{k'+p+1}$};

        \fill[green!30] (1,0) -- (2,0) -- (2,1) -- (1,2) -- (0,2) -- (0,1) -- cycle;
        \draw[dotted,gray] (1,0) -- (2,0) -- (2,1) -- (1,2) -- (0,2) -- (0,1) -- cycle;
        \draw[dotted,gray] (1,0) -- (1,2);
        \draw[dotted,gray] (2,0) -- (0,2);
        \draw[dotted,gray] (0,1) -- (2,1);
        \node at (0,2) {$1$};
        \node at (1,2) {$1$};
        \node at (0,1) {$1$};
        \node at (1,1) {$2$};
        \node at (2,1) {$1$};
        \node at (1,0) {$1$};
        \node at (2,0) {$1$};

        \node at (3.5,0) {$\xrightarrow{k}$};

        \fill[yellow!30] (5,0) -- (7,0) -- (7,2) -- (5,4) -- (3,4) -- (3,2) -- cycle;
        \fill[green!30] (5,1) -- (6,1) -- (6,2) -- (5,3) -- (4,3) -- (4,2) -- cycle;
        \draw[dotted,gray] (5,0) -- (7,0) -- (7,2) -- (5,4) -- (3,4) -- (3,2) -- cycle;
        \draw[dotted,gray] (5,0) -- (5,4);
        \draw[dotted,gray] (7,0) -- (3,4);
        \draw[dotted,gray] (3,2) -- (7,2);
        \draw[dotted,gray] (4,1) -- (7,1) -- (4,4) -- cycle;
        \draw[dotted,gray] (3,3) -- (6,3) -- (6,0) -- cycle;

        \node at (5,2) {$2$};

        \node[scale=.7] at (4,2) {$2k'+1$};
        \node[scale=.7] at (6,2) {$2k'+1$};
        \node[scale=.7] at (4,3) {$2k'+1$};
        \node[scale=.7] at (5,3) {$2k'+1$};
        \node[scale=.7] at (5,1) {$2k'+1$};
        \node[scale=.7] at (6,1) {$2k'+1$};

        \node[scale=.7] at (3,2) {$2k'-1$};
        \node[scale=.7] at (7,2) {$2k'-1$};
        \node[scale=.7] at (3,4) {$2k'-1$};
        \node[scale=.7] at (5,4) {$2k'-1$};
        \node[scale=.7] at (5,0) {$2k'-1$};
        \node[scale=.7] at (7,0) {$2k'-1$};

        \node[scale=.7] at (6,0) {$4k'-2$};
        \node[scale=.7] at (7,1) {$4k'-2$};
        \node[scale=.7] at (6,3) {$4k'-2$};
        \node[scale=.7] at (4,4) {$4k'-2$};
        \node[scale=.7] at (3,3) {$4k'-2$};
        \node[scale=.7] at (4,1) {$4k'-2$};
    	\end{scope}
    \end{tikzpicture}
    }
    \caption{Iterating substitution $\sigma$ on hexagonal patches}
    \label{fig:nonperiodic-3}
\end{figure}

As depicted in figure~\ref{fig:nonperiodic-4}:
\begin{itemize}
\item $\sigma^{\frac{3p+1}{2}}(\bigtriangleup 110)$
      contains a hexagonal path of $0$'s
      only in contact with nonzero values,
\item this hexagonal path of $0$'s
      doubles in size at each application of $\sigma$
      (see section~\ref{sec:sigma-prop-patterns}),
\item $\sigma^{w+\frac{3p+1}{2}}(\bigtriangleup 110)$
      contains a hexagonal path of $0$'s
      with side length $2^w$
      only in contact with nonzero values.
\end{itemize}

\begin{figure}[H]
    \centering
    {
    \begin{tikzpicture}
    	\begin{scope}[yscale=.87,xslant=.5]
        \node[scale=2,gray] at (3.5,-1) {$\sigma^{\frac{3p+1}{2}}$};

        \fill[green!30] (2,0) -- (4,0) -- (4,2) -- (2,4) -- (0,4) -- (0,2) -- cycle;
        \hexgrid{2}{0}{2};

        \node at (2,2) {$2$};

        \node at (1,2) {$0$};
        \node at (3,2) {$0$};
        \node at (1,3) {$0$};
        \node at (2,3) {$0$};
        \node at (2,1) {$0$};
        \node at (3,1) {$0$};

        \node[scale=.7] at (0,2) {$p-2$};
        \node[scale=.7] at (4,2) {$p-2$};
        \node[scale=.7] at (0,4) {$p-2$};
        \node[scale=.7] at (2,4) {$p-2$};
        \node[scale=.7] at (2,0) {$p-2$};
        \node[scale=.7] at (4,0) {$p-2$};

        \node[scale=.7] at (3,0) {$2p-4$};
        \node[scale=.7] at (4,1) {$2p-4$};
        \node[scale=.7] at (3,3) {$2p-4$};
        \node[scale=.7] at (1,4) {$2p-4$};
        \node[scale=.7] at (0,3) {$2p-4$};
        \node[scale=.7] at (1,1) {$2p-4$};

        \node at (6,0) {$\rightarrow$};

        \node[scale=2,gray] at (10.5,-1) {$\sigma^{\frac{3p+3}{2}}$};

        \fill[green!30] (8,1) -- (11,1) -- (11,4) -- (8,7) -- (5,7) -- (5,4) -- cycle;
        \fill[white] (8,3) -- (9,3) -- (9,4) -- (8,5) -- (7,5) -- (7,4) -- cycle;
        \hexgrid{8}{0}{4};

        \node at (8,3) {$2$};
        \node at (9,3) {$2$};
        \node at (9,4) {$2$};
        \node at (8,5) {$2$};
        \node at (7,5) {$2$};
        \node at (7,4) {$2$};

        \node at (8,2) {$0$};
        \node at (9,2) {$0$};
        \node at (10,2) {$0$};
        \node at (10,3) {$0$};
        \node at (10,4) {$0$};
        \node at (9,5) {$0$};
        \node at (8,6) {$0$};
        \node at (7,6) {$0$};
        \node at (6,6) {$0$};
        \node at (6,5) {$0$};
        \node at (6,4) {$0$};
        \node at (7,3) {$0$};

        \node[scale=.7] at (8,1) {$p-2$};
        \node[scale=.7] at (9,1) {$2p-4$};
        \node[scale=.7] at (10,1) {$2p-4$};
        \node[scale=.7] at (11,1) {$p-2$};
        \node[scale=.7] at (11,2) {$2p-4$};
        \node[scale=.7] at (11,3) {$2p-4$};
        \node[scale=.7] at (11,4) {$p-2$};
        \node[scale=.7] at (10,5) {$2p-4$};
        \node[scale=.7] at (9,6) {$2p-4$};
        \node[scale=.7] at (8,7) {$p-2$};
        \node[scale=.7] at (7,7) {$2p-4$};
        \node[scale=.7] at (6,7) {$2p-4$};
        \node[scale=.7] at (5,7) {$p-2$};
        \node[scale=.7] at (5,6) {$2p-4$};
        \node[scale=.7] at (5,5) {$2p-4$};
        \node[scale=.7] at (5,4) {$p-2$};
        \node[scale=.7] at (6,3) {$2p-4$};
        \node[scale=.7] at (7,2) {$2p-4$};

    	\end{scope}
    \end{tikzpicture}
    }
    \caption{Hexagonal paths of $0$'s}
    \label{fig:nonperiodic-4}
\end{figure}
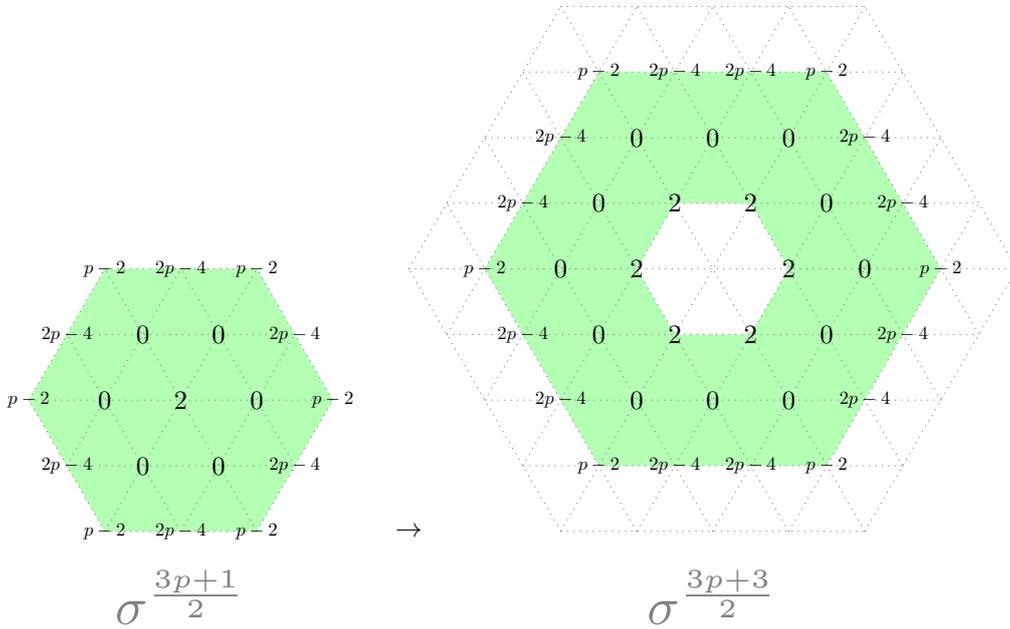

Suppose that $S_t = S_t + x$ for some vector $x \ne 0$.
We can find a hexagonal path of $0$'s, say $h$, with an incircle with diameter $> \lvert x \lvert$.
$h$ would necessarily collide with $h + x$
(and this would induce forbidden all-zero tiles),
hence $S_t$ cannot be periodic (see figure~\ref{fig:nonperiodic-5}).
\end{proof}

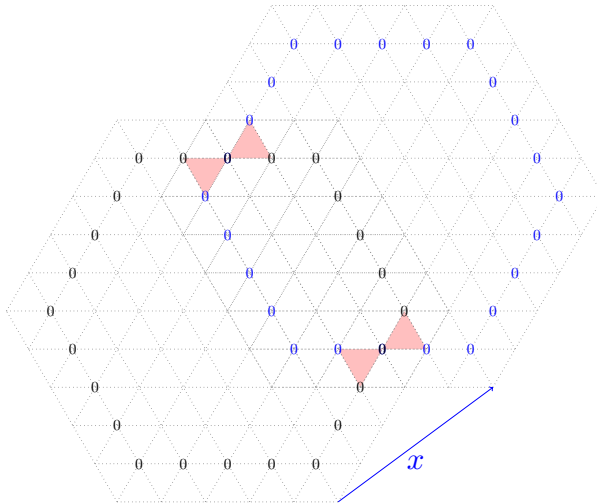
\begin{figure}[H]
    \centering
    \resizebox{8cm}{!}
    {
    \begin{tikzpicture}
    	\begin{scope}[yscale=.87,xslant=.5]
    	    \fill[red!25] (3,3) -- (4,3) -- (4,2) -- cycle;
    	    \fill[red!25] (5,3) -- (4,3) -- (4,4) -- cycle;
    	    \fill[red!25] (-2,7) -- (-2,8) -- (-3,8) -- cycle;
    	    \fill[red!25] (-2,9) -- (-2,8) -- (-1,8) -- cycle;

    	    \hexgrid{0}{-1}{5};
        \hexgrid{2}{2}{5};

    	    \begin{scope}[text=blue]
            \hexOfZeros{2}{3}{4};
        \end{scope}
        \hexOfZeros{0}{0}{4};

        \draw[->,blue] (5,-1) -- (7,2);
        \node[scale=2,below,blue] at (6,.5) {$x$};
    	\end{scope}
    \end{tikzpicture}
    }
    \caption{Collision of hexagonal paths of $0$'s}
    \label{fig:nonperiodic-5}
\end{figure}

\section{\texorpdfstring{Tiling $H_{b,c}$}{Tiling Hbc}}
\label{sec:sigma-tiling-h}

\subsection{Construction}
\label{sec:sigma-tiling-h-cons}

\begin{definition}
$\forall b \in \mathbb{F}_p$,
$\forall c \in \mathbb{F}_p^*$,
let $h_{b,c}$
be the hexagonal patch of six tiles with the value $c$
at the center and six $b$'s around
as depicted in figure~\ref{fig:hexagon-h}.
\end{definition}

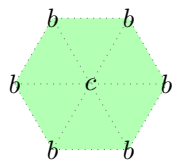
\begin{figure}[H]
    \centering
    {
    \begin{tikzpicture}
    	\begin{scope}[yscale=.87,xslant=.5]
        \fill[green!30] (1,0) -- (2,0) -- (2,1) -- (1,2) -- (0,2) -- (0,1) -- cycle;
        \draw[dotted,gray] (1,0) -- (2,0) -- (2,1) -- (1,2) -- (0,2) -- (0,1) -- cycle;
        \draw[dotted,gray] (1,0) -- (1,2);
        \draw[dotted,gray] (2,0) -- (0,2);
        \draw[dotted,gray] (0,1) -- (2,1);
        \node at (0,2) {$b$};
        \node at (1,2) {$b$};
        \node at (0,1) {$b$};
        \node at (1,1) {$c$};
        \node at (2,1) {$b$};
        \node at (1,0) {$b$};
        \node at (2,0) {$b$};
    	\end{scope}
    \end{tikzpicture}
    }
    \caption{Patch $h_{b,c}$}
    \label{fig:hexagon-h}
\end{figure}

\begin{lemma}
$\forall t \in T^*$,
$b \in \mathbb{F}_p$,
$\forall c \in \mathbb{F}_p^*$,
$h_{b,c}$ appears in $\sigma^k(t)$ for some $k \in \mathbb{N}$.
\end{lemma}

\begin{proof}
From section~\ref{sec:sigma-tiling-s-nonperiodicity},
we know that this is the case for $h_{1,2}$.

Also, $\forall k \in \mathbb{N}$,
$h_{1+2k,2}$ appears in $\sigma^k(h_{1,2})$;
as $k \rightarrow 1+2k$ runs through $\mathbb{F}_p$,
we can find all the patches $h_{b,2}$ with $b \in \mathbb{F}_p$.

Through additivity (see section~\ref{sec:sigma-prop-additivity}),
we can write $h_{b,c} = \frac{c}{2} \cdot h_{\frac{2b}{c},2}$.
\end{proof}

\begin{lemma}
$b \in \mathbb{F}_p$,
$\forall k \in \mathbb{N}$,
$\forall t \in T^*$,
$h_{b,0}$ does not appear in $\sigma^k(t)$.
\end{lemma}

\begin{proof}

\end{proof}
$\forall b \in \mathbb{F}_p$,
let $k$ be the least integer such that $h_{b,0}$ appears in $\sigma^k(t)$.

Note that $b \ne 0$ as an all-zero tile cannot appear in $\sigma^k(t)$.

If the $b$'s were generated during the $k$-th application of $\sigma$,
then $h_{b,0}$ would appear in $\sigma^{k-1}(t)$:
this is a contradiction as $k$ was supposed to be minimal.

If the central $0$ was generated during the $k$-th application of $\sigma$,
then this $0$ would be the sum of two surrounding and opposite $b$'s:
this is a contradiction as $2b \ne 0$ (as $p$ is an odd prime number).

\begin{lemma}
\label{lemma:BCDI}
\begin{equation}
    B^p = C^p = D^p = I_3
\end{equation}
\end{lemma}

\begin{proof}
\begin{equation}
   B^k = \begin{bmatrix}
                1 & 1 & 1    \\
                0 & 1 & 0    \\
                0 & 0 & 1
            \end{bmatrix}^k
       = \begin{bmatrix}
                1 & k & k    \\
                0 & 1 & 0    \\
                0 & 0 & 1
            \end{bmatrix}
\end{equation}
Hence:
\begin{equation}
   B^p = \begin{bmatrix}
                1 & p & p    \\
                0 & 1 & 0    \\
                0 & 0 & 1
            \end{bmatrix}
       = \begin{bmatrix}
                1 & 0 & 0    \\
                0 & 1 & 0    \\
                0 & 0 & 1
            \end{bmatrix}
       = I_3
\end{equation}

The cases for $C$ and $D$ are similar.
\end{proof}

\begin{lemma}
\label{lemma:bcdi}
\begin{equation}
    \beta^p = \gamma^p = \delta^p = id_T
\end{equation}
\end{lemma}

\begin{proof}
This follows from lemma~\ref{lemma:BCDI}
and the fact that $\beta$, $\gamma$ and $\delta$ preserve tile orientations.
\end{proof}

\begin{remark}
\label{remark:corner}
This means that $\forall k \in \mathbb{N}$, $\forall t \in T$,
$t$ appears at the three extreme corners of $\sigma^{kp}(t)$;
see right supertile in figure~\ref{fig:sample}.
\end{remark}

\begin{definition}
$\forall b \in \mathbb{F}_p$,
$\forall c \in \mathbb{F}_p^*$,
$\forall k \in \mathbb{N}$,
let:
\begin{equation}
    h_{b,c}(k) = \sigma^{kp}(h_{b,c})
\end{equation}
\end{definition}

\begin{lemma}
$\forall b \in \mathbb{F}_p$,
$\forall c \in \mathbb{F}_p^*$,
the sequence $\{ h_{b,c}(k) \}_{k \in \mathbb{N}}$
defines a sequence of nested patches
sharing sharing the same central value $c$ surrounded by $b$’s.
\end{lemma}

\begin{proof}
$h_{b,c}(0)$ is composed of six tiles $s_0, \ldots, s_5$
sharing the same central value $c$ surrounded by $b$’s.

Following remark~\ref{remark:corner},
$h_{b,c}(1)$ is composed of six supertiles $\sigma^p(s_0), \ldots, \sigma^p(s_5)$
sharing the same central value $c$ surrounded by $b$’s
(see figure~\ref{fig:h0-h1}).

So $h_{b,c}(0)$ is nested in $h_{b,c}(1)$.\\
Applying $\sigma^p$ again,
$h_{b,c}(1)$ is nested in $h_{b,c}(2)$,
$h_{b,c}(2)$ is nested in $h_{b,c}(3)$,
etc.
\end{proof}

\begin{figure}[H]
    \centering
    \resizebox{10cm}{!}
    {
    \begin{tikzpicture}
    	\begin{scope}[yscale=.87,xslant=.5]
        \fill[green!30] (1,0)
 -- (0,1)
 -- (-1,1)
 -- (-1,0)
 -- (0,-1)
 -- (1,-1)
 -- cycle;
\fill[gray!30] (4,0) -- (5,0) -- (4,1) -- cycle;
\fill[gray!30] (0,4) -- (0,5) -- (1,4) -- cycle;
\fill[gray!30] (0,4) -- (0,5) -- (-1,5) -- cycle;
\fill[gray!30] (-4,4) -- (-5,5) -- (-4,5) -- cycle;
\fill[gray!30] (-4,4) -- (-5,5) -- (-5,4) -- cycle;
\fill[gray!30] (-4,0) -- (-5,0) -- (-5,1) -- cycle;
\fill[gray!30] (-4,0) -- (-5,0) -- (-4,-1) -- cycle;
\fill[gray!30] (0,-4) -- (0,-5) -- (-1,-4) -- cycle;
\fill[gray!30] (0,-4) -- (0,-5) -- (1,-5) -- cycle;
\fill[gray!30] (4,-4) -- (5,-5) -- (4,-5) -- cycle;
\fill[gray!30] (4,-4) -- (5,-5) -- (5,-4) -- cycle;
\fill[gray!30] (4,0) -- (5,0) -- (5,-1) -- cycle;
\draw[dotted] (0,0) -- (5,0) -- (0,5);
\draw[dotted] (1,0) -- (0,1);
\node at (1,0) {$b$};
\node at (1/3,1/3) {$s_0$};
\node[scale=2] at (5/3,5/3) {$\sigma^p(s_0)$};
\node at (13/3,1/3) {$s_0$};
\node at (4,0) {$c$};
\node at (5,0) {$b$};
\node at (4,1) {$b$};
\node at (1/3,13/3) {$s_0$};
\node at (1,4) {$b$};
\draw[dotted] (0,0) -- (0,5) -- (-5,5);
\draw[dotted] (0,1) -- (-1,1);
\node at (0,1) {$b$};
\node at (-1/3,2/3) {$s_1$};
\node[scale=2] at (-5/3,10/3) {$\sigma^p(s_1)$};
\node at (-1/3,14/3) {$s_1$};
\node at (0,4) {$c$};
\node at (0,5) {$b$};
\node at (-1,5) {$b$};
\node at (-13/3,14/3) {$s_1$};
\node at (-4,5) {$b$};
\draw[dotted] (0,0) -- (-5,5) -- (-5,0);
\draw[dotted] (-1,1) -- (-1,0);
\node at (-1,1) {$b$};
\node at (-2/3,1/3) {$s_2$};
\node[scale=2] at (-10/3,5/3) {$\sigma^p(s_2)$};
\node at (-14/3,13/3) {$s_2$};
\node at (-4,4) {$c$};
\node at (-5,5) {$b$};
\node at (-5,4) {$b$};
\node at (-14/3,1/3) {$s_2$};
\node at (-5,1) {$b$};
\draw[dotted] (0,0) -- (-5,0) -- (0,-5);
\draw[dotted] (-1,0) -- (0,-1);
\node at (-1,0) {$b$};
\node at (-1/3,-1/3) {$s_3$};
\node[scale=2] at (-5/3,-5/3) {$\sigma^p(s_3)$};
\node at (-13/3,-1/3) {$s_3$};
\node at (-4,0) {$c$};
\node at (-5,0) {$b$};
\node at (-4,-1) {$b$};
\node at (-1/3,-13/3) {$s_3$};
\node at (-1,-4) {$b$};
\draw[dotted] (0,0) -- (0,-5) -- (5,-5);
\draw[dotted] (0,-1) -- (1,-1);
\node at (0,-1) {$b$};
\node at (1/3,-2/3) {$s_4$};
\node[scale=2] at (5/3,-10/3) {$\sigma^p(s_4)$};
\node at (1/3,-14/3) {$s_4$};
\node at (0,-4) {$c$};
\node at (0,-5) {$b$};
\node at (1,-5) {$b$};
\node at (13/3,-14/3) {$s_4$};
\node at (4,-5) {$b$};
\draw[dotted] (0,0) -- (5,-5) -- (5,0);
\draw[dotted] (1,-1) -- (1,0);
\node at (1,-1) {$b$};
\node at (2/3,-1/3) {$s_5$};
\node[scale=2] at (10/3,-5/3) {$\sigma^p(s_5)$};
\node at (14/3,-13/3) {$s_5$};
\node at (4,-4) {$c$};
\node at (5,-5) {$b$};
\node at (5,-4) {$b$};
\node at (14/3,-1/3) {$s_5$};
\node at (5,-1) {$b$};
\node at (0,0) {$c$};
    	\end{scope}
    \end{tikzpicture}
    }
    \caption{$h_{b,c}(0)$ is nested in $h_{b,c}(1)$}
    \label{fig:h0-h1}
\end{figure}
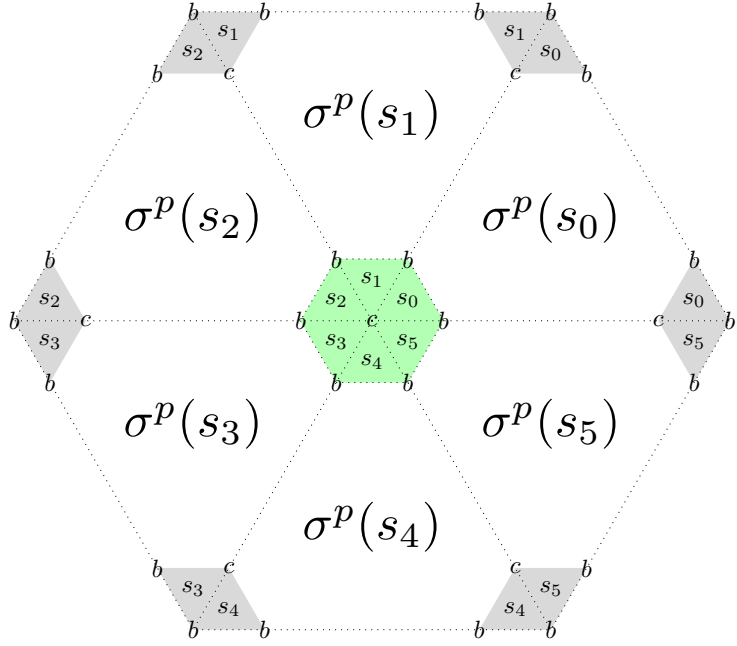

\begin{lemma}
$\forall b \in \mathbb{F}_p$,
$\forall c \in \mathbb{F}_p^*$,
the sequence $\{ h_{b,c}(k) \}_{k \in \mathbb{N}}$ has a limit, say $H_{b,c}$,
that constitutes a tiling of the plane.
\end{lemma}

\begin{proof}
The arguments are similar to that of lemma~\ref{lemma:s-limit}.
\end{proof}

Figure~\ref{fig:stationary-h} presents the first steps
of the construction of $H_{2,1}$ (for $p = 3$);
the thick lines correspond to the boundaries of the first nested patches;
the color scheme is the one used in section~\ref{sec:sigma-cons-examples}.

Figure~\ref{fig:stationary-h-2} presents the central part of $H_{0,1}$.

\begin{remark}
\label{remark-6-fold-symmetry}
$\forall b \in \mathbb{F}_p$,
$\forall c \in \mathbb{F}_p^*$,
$H_{b,c}$ has reflection and 6-fold rotational symmetries.
\end{remark}

\begin{figure}[H]
    \centering
    \resizebox{10cm}{!}
    {
    \begin{tikzpicture}
    	\begin{scope}[yscale=.87,xslant=.5]
        \input{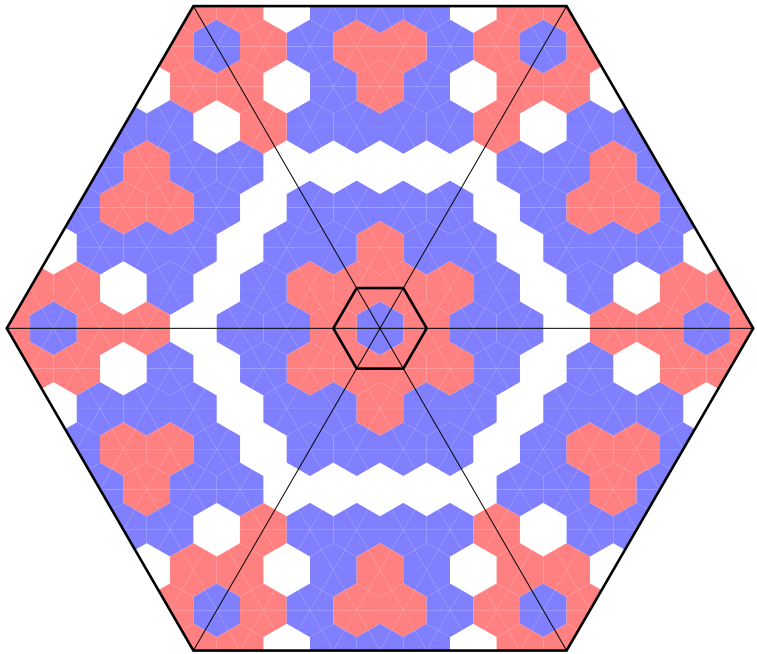}

        \draw (-8,0) -- (8,0);
        \draw (0,-8) -- (0,8);
        \draw (8,-8) -- (-8,8);
        \draw[ultra thick] (0,-1) -- (1,-1) -- (1,0) -- (0,1) -- (-1,1) -- (-1,0) -- cycle;
        \draw[ultra thick] (0,-8) -- (8,-8) -- (8,0) -- (0,8) -- (-8,8) -- (-8,0) -- cycle;
    	\end{scope}
    \end{tikzpicture}
    }
    \caption{Colored representation of $h_{2,1}(k)$ for $k = 0, 1$ (for $p = 3$)}
    \label{fig:stationary-h}
\end{figure}

\begin{figure}[H]
    \begin{center}
    \vstretch{1.732}{
        \includegraphics[height=7cm, keepaspectratio]{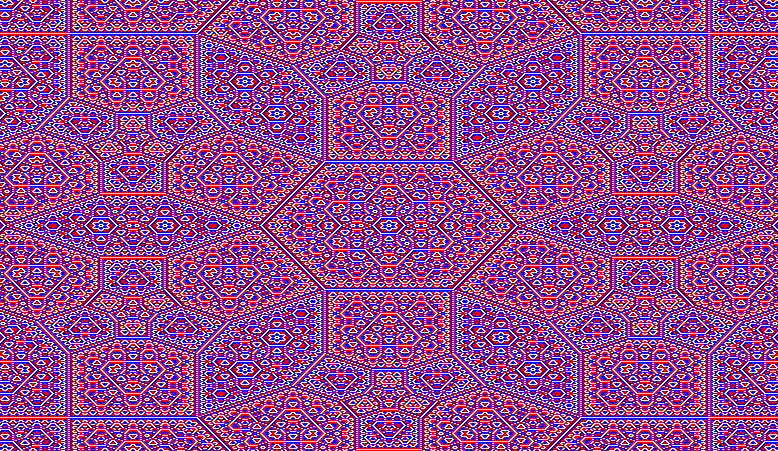}
    }
    \caption{Colored representation of the central part of $H_{0,1}$ (for $p = 3$)}
    \label{fig:stationary-h-2}
    \end{center}
\end{figure}

\begin{remark}
$\forall m, c, c' \in \mathbb{F}_p^*$,
$\forall b, b' \in \mathbb{F}_p$:
\begin{equation}
    H_{b,c} + H_{b',c'} = H_{b+b',c+c'} \mathrm{\ (provided\ that\ }c + c' \ne 0 \mathrm{)}
\end{equation}
\begin{equation}
    m \cdot H_{b,c}= H_{mb,mc}
\end{equation}
\end{remark}

\subsection{Automaticity}
\label{sec:sigma-tiling-h-automaticity}

\begin{definition}
Let $N$ be the deterministic finite automaton with the following characteristics:
\begin{itemize}
\item the set of states is $R = \{+1, -1\} \times \mathbb{F}_p^{3 \times 3}$,
\item the input alphabet is $\Sigma$,
\item the initial state is $r_0 = (+1, I_3)$,
\item the transitions from state $(f, m)$ are performed as follows:\\
\begin{tabular}{|c|l|}
    \hline
    \rule[-1ex]{0pt}{2.5ex} Input symbol & Output state \\
    \hline
    \rule[-1ex]{0pt}{2.5ex} $\alpha$ & $(-f, [\beta]^{-1} \cdot m \cdot [\alpha])$ \\
    \rule[-1ex]{0pt}{2.5ex} $\beta$  & $(+f, [\beta]^{-1} \cdot m \cdot [\beta])$ \\
    \rule[-1ex]{0pt}{2.5ex} $\gamma$ & $(+f, [\beta]^{-1} \cdot m \cdot [\gamma])$ \\
    \rule[-1ex]{0pt}{2.5ex} $\delta$ & $(+f, [\beta]^{-1} \cdot m \cdot [\delta])$ \\
    \hline
\end{tabular}
\item $\forall w \in \Sigma^*$,
      the final state of the automaton after processing the input $w$,
      say $(f, m)$,
      encodes a tile $t$ as follow:
\begin{itemize}
\item if $f = +1$, then $t$ is an upward tile,
\item if $f = -1$, then $t$ is a downward tile,
\item also:
\begin{equation}
    [t] = [c\ b\ b] \cdot m
\end{equation}
\end{itemize}
     \end{itemize}
\end{definition}

\begin{theo}
$\forall b \in \mathbb{F}_p$,
$\forall c \in \mathbb{F}_p^*$,
the tiling $H_{b,c}$ is $4$-automatic
and can be computed with $N$.
\end{theo}

\begin{proof}
As depicted in figure~\ref{fig:automaticity-h}:
\begin{itemize}
\item $H_{b,c}$ can be divided into six sectors
      (related to the six tiles composing the central hexagonal patch $h_{b,c}$),
\item any tile $t_s$ in sector $s$ belong to a family $t_0, \ldots, t_5$
      where $t_r$ belongs to sector $r$
      and is obtained by rotating counterclockwise $t_0$ by $60 \cdot r$ degrees around the origin.
\end{itemize}

\begin{figure}[H]
    \centering
    \resizebox{10cm}{!}
    {
    \begin{tikzpicture}
    	\begin{scope}[yscale=.87,xslant=.5]
        \hexgrid{0}{-5}{5};

        \fill[green!30] (1,0) -- (0,1) -- (-1,1) -- (-1,0) -- (0,-1) -- (1,-1) -- cycle;

                \draw[dotted] (0,0) -- (7,0);
        \node[scale=2] at (3,3) {Sector $0$};
        \fill[yellow!50] (2,1) -- (3,1) -- (2,2) -- cycle;
        \node at (7/3,4/3) {$t_0$};
        \node at (2,1) {$x$};
        \node at (3,1) {$y$};
        \node at (2,2) {$z$};
        \draw[dotted] (0,0) -- (0,7);
        \node[scale=2] at (-3,6) {Sector $1$};
        \fill[yellow!50] (-1,3) -- (-1,4) -- (-2,4) -- cycle;
        \node at (-4/3,11/3) {$t_1$};
        \node at (-1,3) {$x$};
        \node at (-1,4) {$y$};
        \node at (-2,4) {$z$};
        \draw[dotted] (0,0) -- (-7,7);
        \node[scale=2] at (-6,3) {Sector $2$};
        \fill[yellow!50] (-3,2) -- (-4,3) -- (-4,2) -- cycle;
        \node at (-11/3,7/3) {$t_2$};
        \node at (-3,2) {$x$};
        \node at (-4,3) {$y$};
        \node at (-4,2) {$z$};
        \draw[dotted] (0,0) -- (-7,0);
        \node[scale=2] at (-3,-3) {Sector $3$};
        \fill[yellow!50] (-2,-1) -- (-3,-1) -- (-2,-2) -- cycle;
        \node at (-7/3,-4/3) {$t_3$};
        \node at (-2,-1) {$x$};
        \node at (-3,-1) {$y$};
        \node at (-2,-2) {$z$};
        \draw[dotted] (0,0) -- (0,-7);
        \node[scale=2] at (3,-6) {Sector $4$};
        \fill[yellow!50] (1,-3) -- (1,-4) -- (2,-4) -- cycle;
        \node at (4/3,-11/3) {$t_4$};
        \node at (1,-3) {$x$};
        \node at (1,-4) {$y$};
        \node at (2,-4) {$z$};
        \draw[dotted] (0,0) -- (7,-7);
        \node[scale=2] at (6,-3) {Sector $5$};
        \fill[yellow!50] (3,-2) -- (4,-3) -- (4,-2) -- cycle;
        \node at (11/3,-7/3) {$t_5$};
        \node at (3,-2) {$x$};
        \node at (4,-3) {$y$};
        \node at (4,-2) {$z$};

        \node[scale=1.5] at (0,0) {$c$};
        \node[scale=1.5] at (1,0) {$b$};
        \node[scale=1.5] at (0,1) {$b$};
        \node[scale=1.5] at (-1,1) {$b$};
        \node[scale=1.5] at (-1,0) {$b$};
        \node[scale=1.5] at (0,-1) {$b$};
        \node[scale=1.5] at (1,-1) {$b$};
    	\end{scope}
    \end{tikzpicture}
    }
    \caption{$H_{b,c}$ divided into six sectors}
    \label{fig:automaticity-h}
\end{figure}
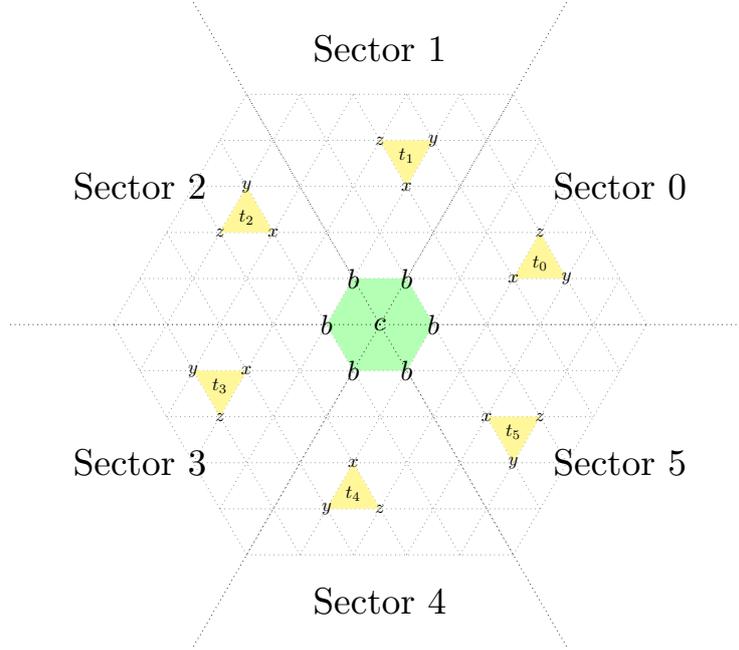

The incircle radius of nested prototiles $\{ h_{b,c}(k) \}_{k \in \mathbb{N}}$ increases without limit,
so $\forall r \in 0 \twodots 5$, for any tile $t_r$ in sector $r$ in the plane,
$\exists k \in \mathbb{N}$
such that $t_r$ belongs to $h_{b,c}(k)$.

$t_0$ belongs to $\sigma^{kp}(\bigtriangleup bcc)$.

As seen in theorem~\ref{theo:tile-position},
the position of $t_0$ within $\sigma^{kp}(\bigtriangleup bcc)$
is characterized by a word $w \in \Sigma^{kp}$.

Suppose that the final state of $N$ after processing $w$ equals $(f, m)$.

By construction, $m$ satisfies:
\begin{equation}
    m = \underbrace{[\beta]^{-1} \ldots [\beta]^{-1}}_{kp \mathrm{\ times}}
        \cdot
        I_3
        \cdot
        [w_0] \ldots [w_{kp-1}]
      = \underbrace{[\beta]^{-kp}}_{=id_T \mathrm{\ (see\ lemma~\ref{lemma:bcdi})}} \cdot [w]
      = [w]
\end{equation}

Also, $t_0$ is at position $w$ within $\sigma^{kp}(\bigtriangleup bcc)$, so:
\begin{equation}
    [t_0] = [\sigma^{kp}(\bigtriangleup bcc)] \cdot [w]
          = [\bigtriangleup bcc] \cdot [w]
          = [b\ c\ c] \cdot [w]
          = [b\ c\ c] \cdot [m]
\end{equation}
So $m$ gives access to $[t_0]$.

\vspace{\baselineskip}

By construction, $f = +1$ if and only if $w$ contains
an even number of $\alpha$'s.

From lemma~\ref{lemma:flip}:
\begin{itemize}
\item $t_0$ has the same orientation as $\bigtriangleup bcc$
      (\textit{i.e.} is an upward tile)
      if and only if the number of $\alpha$'s in $w$ is even.
\end{itemize}
So $f$ gives access to the orientation of $t_0$.

\vspace{\baselineskip}

If $r \ne 0$, then $t_r$ can be deduced from $t_0$ as follows:
\begin{itemize}
\item
\begin{equation}
    [t_r] = [t_0] \cdot \begin{bmatrix}
                            0 & 0 & 1    \\
                            1 & 0 & 0    \\
                            0 & 1 & 0
                        \end{bmatrix}^r
\end{equation}
\item $t_0$ and $t_r$ have the same orientation
      if and only if $r$ is even.
\end{itemize}
\end{proof}

\begin{remark}
If $k$ is not minimal,
then the input will have extra leading $\beta$'s;
this will not alter the final state as
$r_0 \xrightarrow{\beta} r_0$.
\end{remark}

\begin{remark}
$\#\Sigma = 4$, so $H_{b,c}$ is \quotes{4}-automatic.
\end{remark}

\begin{remark}
We can interpret the position $w$ of any tile in sector $0$
as a base-$4$ integer,
by assigning the values $1$, $0$, $2$ and $3$
to $\alpha$, $\beta$, $\gamma$ and $\delta$ respectively;
tiles whose position does not contain $\alpha$'s form a Sierpi\'nski triangle
(see figure~\ref{fig:sector-coord}).
\end{remark}

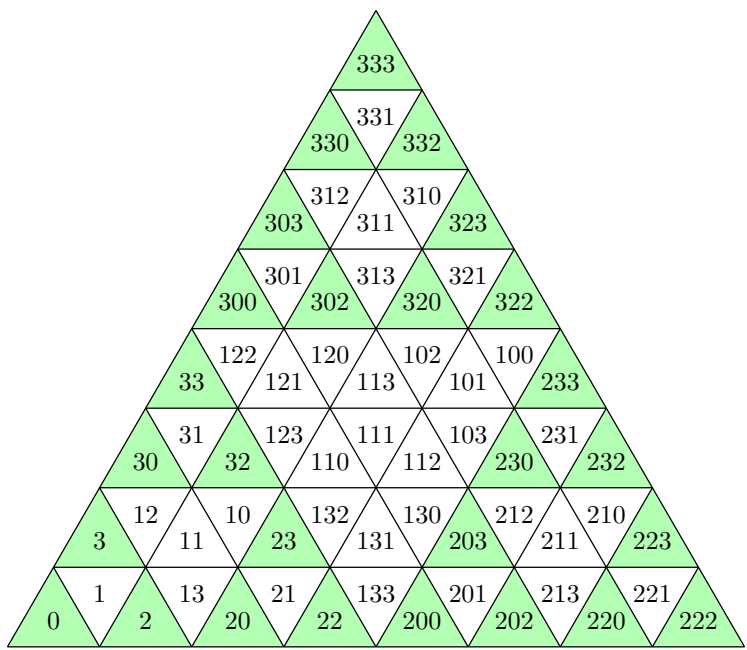
\begin{figure}[H]
    \centering
    \resizebox{10cm}{!}
    {
    \begin{tikzpicture}
    	\begin{scope}[yscale=.87,xslant=.5]
        \fill[green!30] (0,0) -- (1,0) -- (0,1) -- cycle;
\node[scale=.8] at (1/3,1/3) {$0$};
\node[scale=.8] at (2/3,2/3) {$1$};
\fill[green!30] (1,0) -- (2,0) -- (1,1) -- cycle;
\node[scale=.8] at (4/3,1/3) {$2$};
\fill[green!30] (0,1) -- (1,1) -- (0,2) -- cycle;
\node[scale=.8] at (1/3,4/3) {$3$};
\node[scale=.8] at (5/3,5/3) {$10$};
\node[scale=.8] at (4/3,4/3) {$11$};
\node[scale=.8] at (2/3,5/3) {$12$};
\node[scale=.8] at (5/3,2/3) {$13$};
\fill[green!30] (2,0) -- (3,0) -- (2,1) -- cycle;
\node[scale=.8] at (7/3,1/3) {$20$};
\node[scale=.8] at (8/3,2/3) {$21$};
\fill[green!30] (3,0) -- (4,0) -- (3,1) -- cycle;
\node[scale=.8] at (10/3,1/3) {$22$};
\fill[green!30] (2,1) -- (3,1) -- (2,2) -- cycle;
\node[scale=.8] at (7/3,4/3) {$23$};
\fill[green!30] (0,2) -- (1,2) -- (0,3) -- cycle;
\node[scale=.8] at (1/3,7/3) {$30$};
\node[scale=.8] at (2/3,8/3) {$31$};
\fill[green!30] (1,2) -- (2,2) -- (1,3) -- cycle;
\node[scale=.8] at (4/3,7/3) {$32$};
\fill[green!30] (0,3) -- (1,3) -- (0,4) -- cycle;
\node[scale=.8] at (1/3,10/3) {$33$};
\node[scale=.8] at (11/3,11/3) {$100$};
\node[scale=.8] at (10/3,10/3) {$101$};
\node[scale=.8] at (8/3,11/3) {$102$};
\node[scale=.8] at (11/3,8/3) {$103$};
\node[scale=.8] at (7/3,7/3) {$110$};
\node[scale=.8] at (8/3,8/3) {$111$};
\node[scale=.8] at (10/3,7/3) {$112$};
\node[scale=.8] at (7/3,10/3) {$113$};
\node[scale=.8] at (5/3,11/3) {$120$};
\node[scale=.8] at (4/3,10/3) {$121$};
\node[scale=.8] at (2/3,11/3) {$122$};
\node[scale=.8] at (5/3,8/3) {$123$};
\node[scale=.8] at (11/3,5/3) {$130$};
\node[scale=.8] at (10/3,4/3) {$131$};
\node[scale=.8] at (8/3,5/3) {$132$};
\node[scale=.8] at (11/3,2/3) {$133$};
\fill[green!30] (4,0) -- (5,0) -- (4,1) -- cycle;
\node[scale=.8] at (13/3,1/3) {$200$};
\node[scale=.8] at (14/3,2/3) {$201$};
\fill[green!30] (5,0) -- (6,0) -- (5,1) -- cycle;
\node[scale=.8] at (16/3,1/3) {$202$};
\fill[green!30] (4,1) -- (5,1) -- (4,2) -- cycle;
\node[scale=.8] at (13/3,4/3) {$203$};
\node[scale=.8] at (17/3,5/3) {$210$};
\node[scale=.8] at (16/3,4/3) {$211$};
\node[scale=.8] at (14/3,5/3) {$212$};
\node[scale=.8] at (17/3,2/3) {$213$};
\fill[green!30] (6,0) -- (7,0) -- (6,1) -- cycle;
\node[scale=.8] at (19/3,1/3) {$220$};
\node[scale=.8] at (20/3,2/3) {$221$};
\fill[green!30] (7,0) -- (8,0) -- (7,1) -- cycle;
\node[scale=.8] at (22/3,1/3) {$222$};
\fill[green!30] (6,1) -- (7,1) -- (6,2) -- cycle;
\node[scale=.8] at (19/3,4/3) {$223$};
\fill[green!30] (4,2) -- (5,2) -- (4,3) -- cycle;
\node[scale=.8] at (13/3,7/3) {$230$};
\node[scale=.8] at (14/3,8/3) {$231$};
\fill[green!30] (5,2) -- (6,2) -- (5,3) -- cycle;
\node[scale=.8] at (16/3,7/3) {$232$};
\fill[green!30] (4,3) -- (5,3) -- (4,4) -- cycle;
\node[scale=.8] at (13/3,10/3) {$233$};
\fill[green!30] (0,4) -- (1,4) -- (0,5) -- cycle;
\node[scale=.8] at (1/3,13/3) {$300$};
\node[scale=.8] at (2/3,14/3) {$301$};
\fill[green!30] (1,4) -- (2,4) -- (1,5) -- cycle;
\node[scale=.8] at (4/3,13/3) {$302$};
\fill[green!30] (0,5) -- (1,5) -- (0,6) -- cycle;
\node[scale=.8] at (1/3,16/3) {$303$};
\node[scale=.8] at (5/3,17/3) {$310$};
\node[scale=.8] at (4/3,16/3) {$311$};
\node[scale=.8] at (2/3,17/3) {$312$};
\node[scale=.8] at (5/3,14/3) {$313$};
\fill[green!30] (2,4) -- (3,4) -- (2,5) -- cycle;
\node[scale=.8] at (7/3,13/3) {$320$};
\node[scale=.8] at (8/3,14/3) {$321$};
\fill[green!30] (3,4) -- (4,4) -- (3,5) -- cycle;
\node[scale=.8] at (10/3,13/3) {$322$};
\fill[green!30] (2,5) -- (3,5) -- (2,6) -- cycle;
\node[scale=.8] at (7/3,16/3) {$323$};
\fill[green!30] (0,6) -- (1,6) -- (0,7) -- cycle;
\node[scale=.8] at (1/3,19/3) {$330$};
\node[scale=.8] at (2/3,20/3) {$331$};
\fill[green!30] (1,6) -- (2,6) -- (1,7) -- cycle;
\node[scale=.8] at (4/3,19/3) {$332$};
\fill[green!30] (0,7) -- (1,7) -- (0,8) -- cycle;
\node[scale=.8] at (1/3,22/3) {$333$};
\draw(0,0) -- (8,0);
\draw(0,0) -- (0,8);
\draw(8,0) -- (0,8);
\draw(0,1) -- (7,1);
\draw(1,0) -- (1,7);
\draw(7,0) -- (0,7);
\draw(0,2) -- (6,2);
\draw(2,0) -- (2,6);
\draw(6,0) -- (0,6);
\draw(0,3) -- (5,3);
\draw(3,0) -- (3,5);
\draw(5,0) -- (0,5);
\draw(0,4) -- (4,4);
\draw(4,0) -- (4,4);
\draw(4,0) -- (0,4);
\draw(0,5) -- (3,5);
\draw(5,0) -- (5,3);
\draw(3,0) -- (0,3);
\draw(0,6) -- (2,6);
\draw(6,0) -- (6,2);
\draw(2,0) -- (0,2);
\draw(0,7) -- (1,7);
\draw(7,0) -- (7,1);
\draw(1,0) -- (0,1);
    	\end{scope}
    \end{tikzpicture}
    }
    \caption{Positions in the sector $0$ as base-$4$ integers
             (the bottom left tile has position $0$)}
    \label{fig:sector-coord}
\end{figure}

\subsection{Self-similarity}
\label{sec:sigma-tiling-h-self-similarity}

\begin{theo}
$\forall b \in \mathbb{F}_p$,
$\forall c \in \mathbb{F}_p^*$,
the tiling $H_{b,c}$ is self-similar.
\end{theo}

\begin{proof}
\begin{equation}
    \sigma^p(H_{b,c}) = H_{b,c}
\end{equation}
\end{proof}

\subsection{Nonperiodicity}
\label{sec:sigma-tiling-h-nonperiodicity}

\begin{theo}
$\forall b \in \mathbb{F}_p$,
$\forall c \in \mathbb{F}_p^*$,
the tiling $H_{b,c}$ is nonperiodic.
\end{theo}

\begin{proof}
Similar to section~\ref{sec:sigma-tiling-s-nonperiodicity}.
\end{proof}

\section{One-dimensional case}

The structure of this section
is similar to that of sections~\ref{sec:sigma-cons} to \ref{sec:sigma-tiling-h}.

For reasons of simplicity,
we will omit the proofs
when they do not present any particular difficulty.

\subsection{\texorpdfstring{Substitution $\tau$}{Substitution tau}}
\label{sec:tau-cons}

\subsubsection{Tiles}
\label{sec:tau-cons-tiles}

\begin{definition}
Tiles on which the substitution $\tau$ operates
correspond to horizontal unit segments whose endpoints are
decorated with elements of $\mathbb{F}_p$
as depicted in figure~\ref{fig:stern-tile}:
\begin{itemize}
\item the tile whose left and right endpoints
      are decorated respectively with the values $x$ and $y$
      will be denoted by $\square(x,y)$
      (or $\square\ xy$ when there is no ambiguity).
\end{itemize}
\end{definition}

\begin{table} [!htb]
    \centering
    \resizebox{2cm}{!}
    {
    \begin{tikzpicture}
        \draw (0,1) -- (5,1);
        \node[scale=1] at (0,1) [circle,fill=black] {};
        \node[scale=1] at (5,1) [circle,fill=black] {};
        \node[scale=2] at (.5,.5) {$x$};
        \node[scale=2] at (4.5,.5) {$y$};
    \end{tikzpicture}
    }
\caption{The tile $\square\ xy$}
\label{fig:stern-tile}
\end{table}

\begin{definition}
The set of tiles will be denoted by $P$:
\begin{equation}
    P = \{ \square\ xy \mid x, y \in \mathbb{F}_p \}
\end{equation}
\end{definition}

\begin{definition}
We will also be interested in the set of tiles having at least one nonzero endpoint:
\begin{equation}
    P^* = P \setminus \{ \square\ 00 \}
\end{equation}
\end{definition}

\begin{definition}
\label{def:tile-matrix-tau}
We will associate to any tile $t$
the row matrix $[t]$ with the values at its two endpoints
as follows:
$\forall x,y \in \mathbb{F}_p$:
\begin{equation}
    [\square\ xy] = [x\ y]
\end{equation}
\end{definition}

\subsubsection{Substitution rule}
\label{sec:tau-cons-rule}

\begin{definition}
The substitution $\tau$ transforms a tile $t = \square\ xy$
into two tiles $\lambda(t) = \square(x,x+y)$ and $\rho(t) = \square(x+y,y)$
as depicted in figure~\ref{fig:tau}.
\end{definition}

\begin{figure}[H]
    \centering
    \resizebox{10cm}{!}
    {
    \begin{tikzpicture}
        \draw (-1,1) -- (4,1);
        \node[scale=1] at (-1,1) [circle,fill=black] {};
        \node[scale=1] at (4,1) [circle,fill=black] {};
        \node[scale=2] at (-.5,.5) {$x$};
        \node[scale=2] at (3.5,.5) {$y$};
        \node[scale=2] at (1.5,1.5) {$t$};

        \node[scale=2] at (5.5,1) {$\rightarrow$};

        \draw (7,1) -- (12,1);
        \node[scale=1] at (7,1) [circle,fill=black] {};
        \node[scale=1] at (12,1) [circle,fill=black] {};
        \node[scale=2] at (7.5,.5) {$x$};
        \node[scale=2] at (11,.5) {$x+y$};
        \node[scale=2] at (9.5,1.5) {$\lambda(t)$};

        \draw (12,1) -- (17,1);
        \node[scale=1] at (17,1) [circle,fill=black] {};
        \node[scale=2] at (13,.5) {$x+y$};
        \node[scale=2] at (16.5,.5) {$y$};
        \node[scale=2] at (14.5,1.5) {$\rho(t)$};
    \end{tikzpicture}
    }
    \caption{The substitution rule for the tile $\square\ xy$}
    \label{fig:tau}
\end{figure}
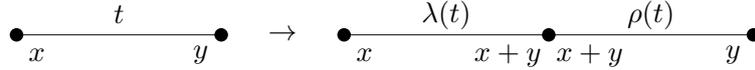

\begin{lemma}
$\forall k \in \mathbb{N}^*$,
$\forall t \in P$,
using the notation introduced in definition~\ref{def:tile-matrix-tau},
we can conveniently relate $\lambda(t)$ and $\rho(t)$
to $t$
by means of multiplication by certain matrices $L$ and $R$
as follows:
\begin{equation}
    [\lambda(t)]
        = [t] \cdot \underbrace{
            \begin{bmatrix}
                1 & 1    \\
                0 & 1
            \end{bmatrix}
        }_L
\end{equation}
\begin{equation}
    [\rho(t)]
        = [t] \cdot \underbrace{
            \begin{bmatrix}
                1 & 0    \\
                1 & 1
            \end{bmatrix}
        }_R
\end{equation}
\end{lemma}

\begin{lemma}
\label{lemma:lr-invertible}
The matrices $L$ and $R$ are invertible in $\mathbb{F}_p^{2 \times 2}$.
\end{lemma}

\begin{lemma}
The functions $\lambda$ and $\rho$ are permutations of P.
\end{lemma}

\begin{lemma}
\label{lemma:extreme-tau}
$\forall k \in \mathbb{N}$,
$\forall t \in P$,
the values at the two extreme endpoints of $\tau^k(t)$
are exactly the values at the endpoints of $t$.
\end{lemma}

\begin{definition}
\label{def:tau-row}
Following lemma~\ref{lemma:extreme-tau},
we will associate to any supertile $s$
the row matrix $[s]$ with the values at its two extreme endpoints
as follows:
$\forall k \in \mathbb{N}$,
$\forall x,y \in \mathbb{F}_p$:
\begin{equation}
    [\tau^k(\square\ xy)] = [x\ y]
\end{equation}
\end{definition}

\begin{lemma}
Endpoints that meet at the same point hold the same value.
\end{lemma}

We will therefore, from now on, simplify the figures by indicating
just the values where the endpoints meet
(see figure~\ref{fig:tau-simple}).

\begin{figure}[H]
    \centering
    \resizebox{10cm}{!}
    {
    \begin{tikzpicture}
        \draw (0,0) -- (1,0);
        \node[below] at (0,0) {$x$};
        \node[below] at (1,0) {$y$};
        \node[scale=.2] at (0,0) [circle,fill=black] {};
        \node[scale=.2] at (1,0) [circle,fill=black] {};

        \node at (2,0) {$\rightarrow$};

        \draw (3,0) -- (5,0);
        \node[below] at (3,0) {$x$};
        \node[below] at (4,0) {$x+y$};
        \node[below] at (5,0) {$y$};
        \node[scale=.2] at (3,0) [circle,fill=black] {};
        \node[scale=.2] at (4,0) [circle,fill=black] {};
        \node[scale=.2] at (5,0) [circle,fill=black] {};
    \end{tikzpicture}
    }
    \caption{The \quotes{simplified} substitution rule for the tile $\square\ xy$}
    \label{fig:tau-simple}
\end{figure}
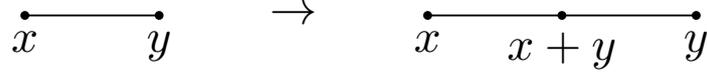

\begin{remark}
$\forall k \in \mathbb{N}$,
$\forall x, y, z \in \mathbb{F}_p$,
the bottom border of $\sigma^k(\bigtriangleup xyz)$
matches $\tau^k(\square\ xy)$
(see figure~\ref{fig:sigma-tau}).
\end{remark}

\begin{figure}[H]
    \centering
    \resizebox{10cm}{!}
    {
    \begin{tikzpicture}
    \begin{scope}[yscale=.87,xslant=.5]
        \node[scale=.5] at (1,-2) [circle,fill=black] {};
\node[scale=.5] at (2,-2) [circle,fill=black] {};
\node[scale=.5] at (3,-2) [circle,fill=black] {};
\node[scale=.5] at (4,-2) [circle,fill=black] {};
\node[scale=.5] at (4,-2) [circle,fill=black] {};
\node[scale=.5] at (5,-2) [circle,fill=black] {};
\node[scale=.5] at (4,-2) [circle,fill=black] {};
\node[scale=.5] at (6,-2) [circle,fill=black] {};
\node[scale=.5] at (7,-2) [circle,fill=black] {};
\node[scale=.5] at (7,-2) [circle,fill=black] {};
\node[scale=.5] at (8,-2) [circle,fill=black] {};
\node[scale=.5] at (7,-2) [circle,fill=black] {};
\node[scale=.5] at (8,-2) [circle,fill=black] {};
\node[scale=.5] at (9,-2) [circle,fill=black] {};
\node[scale=.5] at (9,-2) [circle,fill=black] {};
\node[scale=.5] at (10,-2) [circle,fill=black] {};
\node[scale=.5] at (9,-2) [circle,fill=black] {};
\node[scale=.5] at (8,-2) [circle,fill=black] {};
\node[scale=2] at (0,0) {$1$};
\node[scale=2] at (1,0) {$2$};
\node[scale=2] at (0,1) {$2$};
\draw[dotted] (0,0) -- (1,0) -- (0,1) -- cycle;
\node[below,scale=2] at (1,-2) {$1$};
\node[below,scale=2] at (2,-2) {$2$};
\draw (1,-2) -- (2,-2);
\node[scale=2] at (2,0) {$1$};
\node[scale=2] at (3,0) {$0$};
\node[scale=2] at (2,1) {$0$};
\draw[dotted] (2,0) -- (3,0) -- (2,1) -- cycle;
\node[below,scale=2] at (3,-2) {$1$};
\node[below,scale=2] at (4,-2) {$0$};
\draw (3,-2) -- (4,-2);
\node[scale=2] at (3,0) {$0$};
\node[scale=2] at (4,0) {$2$};
\node[scale=2] at (3,1) {$1$};
\draw[dotted] (3,0) -- (4,0) -- (3,1) -- cycle;
\node[below,scale=2] at (4,-2) {$0$};
\node[below,scale=2] at (5,-2) {$2$};
\draw (4,-2) -- (5,-2);
\node[scale=2] at (2,1) {$0$};
\node[scale=2] at (3,1) {$1$};
\node[scale=2] at (2,2) {$2$};
\draw[dotted] (2,1) -- (3,1) -- (2,2) -- cycle;
\node[scale=2] at (3,0) {$0$};
\node[scale=2] at (2,1) {$0$};
\node[scale=2] at (3,1) {$1$};
\draw[dotted] (3,0) -- (2,1) -- (3,1) -- cycle;
\node[below,scale=2] at (4,-2) {$0$};
\node[scale=2] at (5,0) {$1$};
\node[scale=2] at (6,0) {$1$};
\node[scale=2] at (5,1) {$1$};
\draw[dotted] (5,0) -- (6,0) -- (5,1) -- cycle;
\node[below,scale=2] at (6,-2) {$1$};
\node[below,scale=2] at (7,-2) {$1$};
\draw (6,-2) -- (7,-2);
\node[scale=2] at (6,0) {$1$};
\node[scale=2] at (7,0) {$0$};
\node[scale=2] at (6,1) {$0$};
\draw[dotted] (6,0) -- (7,0) -- (6,1) -- cycle;
\node[below,scale=2] at (7,-2) {$1$};
\node[below,scale=2] at (8,-2) {$0$};
\draw (7,-2) -- (8,-2);
\node[scale=2] at (5,1) {$1$};
\node[scale=2] at (6,1) {$0$};
\node[scale=2] at (5,2) {$0$};
\draw[dotted] (5,1) -- (6,1) -- (5,2) -- cycle;
\node[scale=2] at (6,0) {$1$};
\node[scale=2] at (5,1) {$1$};
\node[scale=2] at (6,1) {$0$};
\draw[dotted] (6,0) -- (5,1) -- (6,1) -- cycle;
\node[below,scale=2] at (7,-2) {$1$};
\node[scale=2] at (7,0) {$0$};
\node[scale=2] at (8,0) {$2$};
\node[scale=2] at (7,1) {$1$};
\draw[dotted] (7,0) -- (8,0) -- (7,1) -- cycle;
\node[below,scale=2] at (8,-2) {$0$};
\node[below,scale=2] at (9,-2) {$2$};
\draw (8,-2) -- (9,-2);
\node[scale=2] at (8,0) {$2$};
\node[scale=2] at (9,0) {$2$};
\node[scale=2] at (8,1) {$0$};
\draw[dotted] (8,0) -- (9,0) -- (8,1) -- cycle;
\node[below,scale=2] at (9,-2) {$2$};
\node[below,scale=2] at (10,-2) {$2$};
\draw (9,-2) -- (10,-2);
\node[scale=2] at (7,1) {$1$};
\node[scale=2] at (8,1) {$0$};
\node[scale=2] at (7,2) {$1$};
\draw[dotted] (7,1) -- (8,1) -- (7,2) -- cycle;
\node[scale=2] at (8,0) {$2$};
\node[scale=2] at (7,1) {$1$};
\node[scale=2] at (8,1) {$0$};
\draw[dotted] (8,0) -- (7,1) -- (8,1) -- cycle;
\node[below,scale=2] at (9,-2) {$2$};
\node[scale=2] at (5,2) {$0$};
\node[scale=2] at (6,2) {$1$};
\node[scale=2] at (5,3) {$2$};
\draw[dotted] (5,2) -- (6,2) -- (5,3) -- cycle;
\node[scale=2] at (6,2) {$1$};
\node[scale=2] at (7,2) {$1$};
\node[scale=2] at (6,3) {$0$};
\draw[dotted] (6,2) -- (7,2) -- (6,3) -- cycle;
\node[scale=2] at (5,3) {$2$};
\node[scale=2] at (6,3) {$0$};
\node[scale=2] at (5,4) {$2$};
\draw[dotted] (5,3) -- (6,3) -- (5,4) -- cycle;
\node[scale=2] at (6,2) {$1$};
\node[scale=2] at (5,3) {$2$};
\node[scale=2] at (6,3) {$0$};
\draw[dotted] (6,2) -- (5,3) -- (6,3) -- cycle;
\node[scale=2] at (7,0) {$0$};
\node[scale=2] at (6,1) {$0$};
\node[scale=2] at (7,1) {$1$};
\draw[dotted] (7,0) -- (6,1) -- (7,1) -- cycle;
\node[below,scale=2] at (8,-2) {$0$};
\node[scale=2] at (6,1) {$0$};
\node[scale=2] at (5,2) {$0$};
\node[scale=2] at (6,2) {$1$};
\draw[dotted] (6,1) -- (5,2) -- (6,2) -- cycle;
\node[scale=2] at (7,1) {$1$};
\node[scale=2] at (6,2) {$1$};
\node[scale=2] at (7,2) {$1$};
\draw[dotted] (7,1) -- (6,2) -- (7,2) -- cycle;
\node[scale=2] at (6,1) {$0$};
\node[scale=2] at (7,1) {$1$};
\node[scale=2] at (6,2) {$1$};
\draw[dotted] (6,1) -- (7,1) -- (6,2) -- cycle;
    \end{scope}
    \end{tikzpicture}
    }
    \caption{$\sigma^k(\bigtriangleup 122)$ and $\tau^k(\square\ 12)$ for $k = 0 \twodots 2$ (for $p = 3$)}
    \label{fig:sigma-tau}
\end{figure}
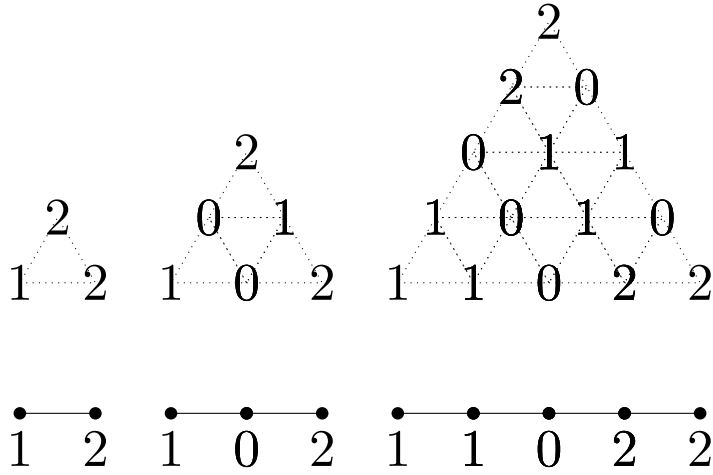

\subsubsection{Examples}
\label{sec:tau-cons-examples}

Figure~\ref{fig:sample-tau} depicts $\tau^k(\square\ 12)$ for $k = 0 \twodots 4$ numerically (for $p = 3$).

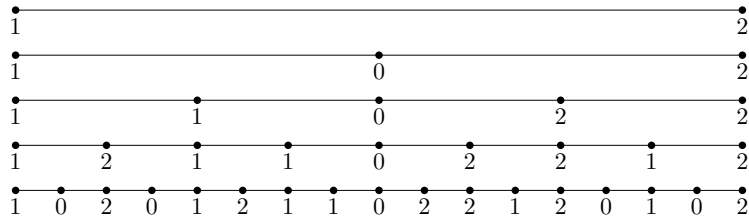
\begin{figure}[H]
    \centering
    \resizebox{10cm}{!}
    {
    \begin{tikzpicture}
                \draw (0,8) -- (32,8);
        \node[scale=1] at (0,8) [circle,fill=black] {};
        \node[scale=3,below] at (0,8) {1};
        \node[scale=1] at (32,8) [circle,fill=black] {};
        \node[scale=3,below] at (32,8) {2};
        \draw (0,6) -- (32,6);
        \node[scale=1] at (0,6) [circle,fill=black] {};
        \node[scale=3,below] at (0,6) {1};
        \node[scale=1] at (16,6) [circle,fill=black] {};
        \node[scale=3,below] at (16,6) {0};
        \node[scale=1] at (32,6) [circle,fill=black] {};
        \node[scale=3,below] at (32,6) {2};
        \draw (0,4) -- (32,4);
        \node[scale=1] at (0,4) [circle,fill=black] {};
        \node[scale=3,below] at (0,4) {1};
        \node[scale=1] at (8,4) [circle,fill=black] {};
        \node[scale=3,below] at (8,4) {1};
        \node[scale=1] at (16,4) [circle,fill=black] {};
        \node[scale=3,below] at (16,4) {0};
        \node[scale=1] at (24,4) [circle,fill=black] {};
        \node[scale=3,below] at (24,4) {2};
        \node[scale=1] at (32,4) [circle,fill=black] {};
        \node[scale=3,below] at (32,4) {2};
        \draw (0,2) -- (32,2);
        \node[scale=1] at (0,2) [circle,fill=black] {};
        \node[scale=3,below] at (0,2) {1};
        \node[scale=1] at (4,2) [circle,fill=black] {};
        \node[scale=3,below] at (4,2) {2};
        \node[scale=1] at (8,2) [circle,fill=black] {};
        \node[scale=3,below] at (8,2) {1};
        \node[scale=1] at (12,2) [circle,fill=black] {};
        \node[scale=3,below] at (12,2) {1};
        \node[scale=1] at (16,2) [circle,fill=black] {};
        \node[scale=3,below] at (16,2) {0};
        \node[scale=1] at (20,2) [circle,fill=black] {};
        \node[scale=3,below] at (20,2) {2};
        \node[scale=1] at (24,2) [circle,fill=black] {};
        \node[scale=3,below] at (24,2) {2};
        \node[scale=1] at (28,2) [circle,fill=black] {};
        \node[scale=3,below] at (28,2) {1};
        \node[scale=1] at (32,2) [circle,fill=black] {};
        \node[scale=3,below] at (32,2) {2};
        \draw (0,0) -- (32,0);
        \node[scale=1] at (0,0) [circle,fill=black] {};
        \node[scale=3,below] at (0,0) {1};
        \node[scale=1] at (2,0) [circle,fill=black] {};
        \node[scale=3,below] at (2,0) {0};
        \node[scale=1] at (4,0) [circle,fill=black] {};
        \node[scale=3,below] at (4,0) {2};
        \node[scale=1] at (6,0) [circle,fill=black] {};
        \node[scale=3,below] at (6,0) {0};
        \node[scale=1] at (8,0) [circle,fill=black] {};
        \node[scale=3,below] at (8,0) {1};
        \node[scale=1] at (10,0) [circle,fill=black] {};
        \node[scale=3,below] at (10,0) {2};
        \node[scale=1] at (12,0) [circle,fill=black] {};
        \node[scale=3,below] at (12,0) {1};
        \node[scale=1] at (14,0) [circle,fill=black] {};
        \node[scale=3,below] at (14,0) {1};
        \node[scale=1] at (16,0) [circle,fill=black] {};
        \node[scale=3,below] at (16,0) {0};
        \node[scale=1] at (18,0) [circle,fill=black] {};
        \node[scale=3,below] at (18,0) {2};
        \node[scale=1] at (20,0) [circle,fill=black] {};
        \node[scale=3,below] at (20,0) {2};
        \node[scale=1] at (22,0) [circle,fill=black] {};
        \node[scale=3,below] at (22,0) {1};
        \node[scale=1] at (24,0) [circle,fill=black] {};
        \node[scale=3,below] at (24,0) {2};
        \node[scale=1] at (26,0) [circle,fill=black] {};
        \node[scale=3,below] at (26,0) {0};
        \node[scale=1] at (28,0) [circle,fill=black] {};
        \node[scale=3,below] at (28,0) {1};
        \node[scale=1] at (30,0) [circle,fill=black] {};
        \node[scale=3,below] at (30,0) {0};
        \node[scale=1] at (32,0) [circle,fill=black] {};
        \node[scale=3,below] at (32,0) {2};
    \end{tikzpicture}
    }
    \caption{$\tau^k(\square\ 12)$ for $k = 0 \twodots 4$ (for $p = 3$)}
    \label{fig:sample-tau}
\end{figure}

Figure~\ref{fig:sample-tau-2} depicts $\tau^k(\square\ 12)$ for $k = 0 \twodots 4$ graphically with colors (for $p = 3$).

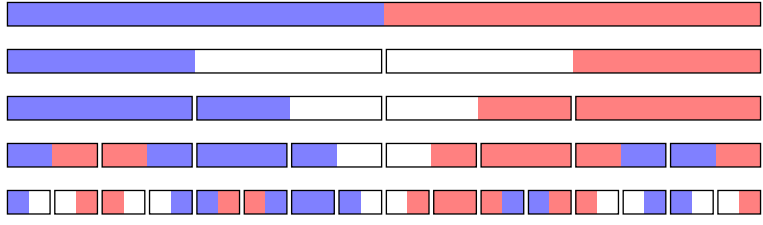
\begin{figure}[H]
    \centering
    \resizebox{10cm}{!}
    {
    \begin{tikzpicture}
                \fill[blue!50] (0+.1,8) rectangle (16,9);
        \fill[red!50] (16,8) rectangle (32-.1,9);
        \draw[ultra thick] (0+.1,8) rectangle (32-.1,9);
        \fill[blue!50] (0+.1,6) rectangle (8,7);
        \fill[red!50] (24,6) rectangle (32-.1,7);
        \draw[ultra thick] (0+.1,6) rectangle (16-.1,7);
        \draw[ultra thick] (16+.1,6) rectangle (32-.1,7);
        \fill[blue!50] (0+.1,4) rectangle (4,5);
        \fill[blue!50] (4,4) rectangle (8-.1,5);
        \fill[blue!50] (8+.1,4) rectangle (12,5);
        \fill[red!50] (20,4) rectangle (24-.1,5);
        \fill[red!50] (24+.1,4) rectangle (28,5);
        \fill[red!50] (28,4) rectangle (32-.1,5);
        \draw[ultra thick] (0+.1,4) rectangle (8-.1,5);
        \draw[ultra thick] (8+.1,4) rectangle (16-.1,5);
        \draw[ultra thick] (16+.1,4) rectangle (24-.1,5);
        \draw[ultra thick] (24+.1,4) rectangle (32-.1,5);
        \fill[blue!50] (0+.1,2) rectangle (2,3);
        \fill[red!50] (2,2) rectangle (4-.1,3);
        \fill[red!50] (4+.1,2) rectangle (6,3);
        \fill[blue!50] (6,2) rectangle (8-.1,3);
        \fill[blue!50] (8+.1,2) rectangle (10,3);
        \fill[blue!50] (10,2) rectangle (12-.1,3);
        \fill[blue!50] (12+.1,2) rectangle (14,3);
        \fill[red!50] (18,2) rectangle (20-.1,3);
        \fill[red!50] (20+.1,2) rectangle (22,3);
        \fill[red!50] (22,2) rectangle (24-.1,3);
        \fill[red!50] (24+.1,2) rectangle (26,3);
        \fill[blue!50] (26,2) rectangle (28-.1,3);
        \fill[blue!50] (28+.1,2) rectangle (30,3);
        \fill[red!50] (30,2) rectangle (32-.1,3);
        \draw[ultra thick] (0+.1,2) rectangle (4-.1,3);
        \draw[ultra thick] (4+.1,2) rectangle (8-.1,3);
        \draw[ultra thick] (8+.1,2) rectangle (12-.1,3);
        \draw[ultra thick] (12+.1,2) rectangle (16-.1,3);
        \draw[ultra thick] (16+.1,2) rectangle (20-.1,3);
        \draw[ultra thick] (20+.1,2) rectangle (24-.1,3);
        \draw[ultra thick] (24+.1,2) rectangle (28-.1,3);
        \draw[ultra thick] (28+.1,2) rectangle (32-.1,3);
        \fill[blue!50] (0+.1,0) rectangle (1,1);
        \fill[red!50] (3,0) rectangle (4-.1,1);
        \fill[red!50] (4+.1,0) rectangle (5,1);
        \fill[blue!50] (7,0) rectangle (8-.1,1);
        \fill[blue!50] (8+.1,0) rectangle (9,1);
        \fill[red!50] (9,0) rectangle (10-.1,1);
        \fill[red!50] (10+.1,0) rectangle (11,1);
        \fill[blue!50] (11,0) rectangle (12-.1,1);
        \fill[blue!50] (12+.1,0) rectangle (13,1);
        \fill[blue!50] (13,0) rectangle (14-.1,1);
        \fill[blue!50] (14+.1,0) rectangle (15,1);
        \fill[red!50] (17,0) rectangle (18-.1,1);
        \fill[red!50] (18+.1,0) rectangle (19,1);
        \fill[red!50] (19,0) rectangle (20-.1,1);
        \fill[red!50] (20+.1,0) rectangle (21,1);
        \fill[blue!50] (21,0) rectangle (22-.1,1);
        \fill[blue!50] (22+.1,0) rectangle (23,1);
        \fill[red!50] (23,0) rectangle (24-.1,1);
        \fill[red!50] (24+.1,0) rectangle (25,1);
        \fill[blue!50] (27,0) rectangle (28-.1,1);
        \fill[blue!50] (28+.1,0) rectangle (29,1);
        \fill[red!50] (31,0) rectangle (32-.1,1);
        \draw[ultra thick] (0+.1,0) rectangle (2-.1,1);
        \draw[ultra thick] (2+.1,0) rectangle (4-.1,1);
        \draw[ultra thick] (4+.1,0) rectangle (6-.1,1);
        \draw[ultra thick] (6+.1,0) rectangle (8-.1,1);
        \draw[ultra thick] (8+.1,0) rectangle (10-.1,1);
        \draw[ultra thick] (10+.1,0) rectangle (12-.1,1);
        \draw[ultra thick] (12+.1,0) rectangle (14-.1,1);
        \draw[ultra thick] (14+.1,0) rectangle (16-.1,1);
        \draw[ultra thick] (16+.1,0) rectangle (18-.1,1);
        \draw[ultra thick] (18+.1,0) rectangle (20-.1,1);
        \draw[ultra thick] (20+.1,0) rectangle (22-.1,1);
        \draw[ultra thick] (22+.1,0) rectangle (24-.1,1);
        \draw[ultra thick] (24+.1,0) rectangle (26-.1,1);
        \draw[ultra thick] (26+.1,0) rectangle (28-.1,1);
        \draw[ultra thick] (28+.1,0) rectangle (30-.1,1);
        \draw[ultra thick] (30+.1,0) rectangle (32-.1,1);
    \end{tikzpicture}
    }
    \caption{Colored representation of $\tau^k(\square\ 12)$ for $k = 0 \twodots 4$ (for $p = 3$)}
    \label{fig:sample-tau-2}
\end{figure}

Figure~\ref{fig:sample-tau-3} depicts $\tau^k(\square\ 12)$ for $k = 0 \twodots 4$ graphically with colors (for $p = 5$).

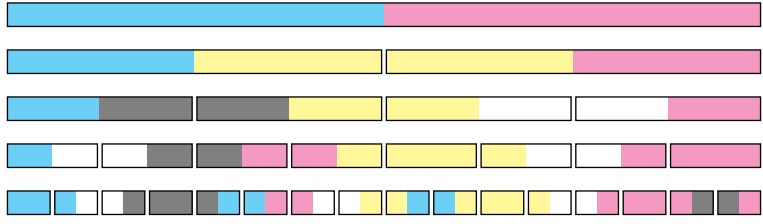
\begin{figure}[H]
    \centering
    \resizebox{10cm}{!}
    {
    \begin{tikzpicture}
                \fill[cyan!50] (0+.1,8) rectangle (16,9);
        \fill[magenta!50] (16,8) rectangle (32-.1,9);
        \draw[ultra thick] (0+.1,8) rectangle (32-.1,9);
        \fill[cyan!50] (0+.1,6) rectangle (8,7);
        \fill[yellow!50] (8,6) rectangle (16-.1,7);
        \fill[yellow!50] (16+.1,6) rectangle (24,7);
        \fill[magenta!50] (24,6) rectangle (32-.1,7);
        \draw[ultra thick] (0+.1,6) rectangle (16-.1,7);
        \draw[ultra thick] (16+.1,6) rectangle (32-.1,7);
        \fill[cyan!50] (0+.1,4) rectangle (4,5);
        \fill[black!50] (4,4) rectangle (8-.1,5);
        \fill[black!50] (8+.1,4) rectangle (12,5);
        \fill[yellow!50] (12,4) rectangle (16-.1,5);
        \fill[yellow!50] (16+.1,4) rectangle (20,5);
        \fill[magenta!50] (28,4) rectangle (32-.1,5);
        \draw[ultra thick] (0+.1,4) rectangle (8-.1,5);
        \draw[ultra thick] (8+.1,4) rectangle (16-.1,5);
        \draw[ultra thick] (16+.1,4) rectangle (24-.1,5);
        \draw[ultra thick] (24+.1,4) rectangle (32-.1,5);
        \fill[cyan!50] (0+.1,2) rectangle (2,3);
        \fill[black!50] (6,2) rectangle (8-.1,3);
        \fill[black!50] (8+.1,2) rectangle (10,3);
        \fill[magenta!50] (10,2) rectangle (12-.1,3);
        \fill[magenta!50] (12+.1,2) rectangle (14,3);
        \fill[yellow!50] (14,2) rectangle (16-.1,3);
        \fill[yellow!50] (16+.1,2) rectangle (18,3);
        \fill[yellow!50] (18,2) rectangle (20-.1,3);
        \fill[yellow!50] (20+.1,2) rectangle (22,3);
        \fill[magenta!50] (26,2) rectangle (28-.1,3);
        \fill[magenta!50] (28+.1,2) rectangle (30,3);
        \fill[magenta!50] (30,2) rectangle (32-.1,3);
        \draw[ultra thick] (0+.1,2) rectangle (4-.1,3);
        \draw[ultra thick] (4+.1,2) rectangle (8-.1,3);
        \draw[ultra thick] (8+.1,2) rectangle (12-.1,3);
        \draw[ultra thick] (12+.1,2) rectangle (16-.1,3);
        \draw[ultra thick] (16+.1,2) rectangle (20-.1,3);
        \draw[ultra thick] (20+.1,2) rectangle (24-.1,3);
        \draw[ultra thick] (24+.1,2) rectangle (28-.1,3);
        \draw[ultra thick] (28+.1,2) rectangle (32-.1,3);
        \fill[cyan!50] (0+.1,0) rectangle (1,1);
        \fill[cyan!50] (1,0) rectangle (2-.1,1);
        \fill[cyan!50] (2+.1,0) rectangle (3,1);
        \fill[black!50] (5,0) rectangle (6-.1,1);
        \fill[black!50] (6+.1,0) rectangle (7,1);
        \fill[black!50] (7,0) rectangle (8-.1,1);
        \fill[black!50] (8+.1,0) rectangle (9,1);
        \fill[cyan!50] (9,0) rectangle (10-.1,1);
        \fill[cyan!50] (10+.1,0) rectangle (11,1);
        \fill[magenta!50] (11,0) rectangle (12-.1,1);
        \fill[magenta!50] (12+.1,0) rectangle (13,1);
        \fill[yellow!50] (15,0) rectangle (16-.1,1);
        \fill[yellow!50] (16+.1,0) rectangle (17,1);
        \fill[cyan!50] (17,0) rectangle (18-.1,1);
        \fill[cyan!50] (18+.1,0) rectangle (19,1);
        \fill[yellow!50] (19,0) rectangle (20-.1,1);
        \fill[yellow!50] (20+.1,0) rectangle (21,1);
        \fill[yellow!50] (21,0) rectangle (22-.1,1);
        \fill[yellow!50] (22+.1,0) rectangle (23,1);
        \fill[magenta!50] (25,0) rectangle (26-.1,1);
        \fill[magenta!50] (26+.1,0) rectangle (27,1);
        \fill[magenta!50] (27,0) rectangle (28-.1,1);
        \fill[magenta!50] (28+.1,0) rectangle (29,1);
        \fill[black!50] (29,0) rectangle (30-.1,1);
        \fill[black!50] (30+.1,0) rectangle (31,1);
        \fill[magenta!50] (31,0) rectangle (32-.1,1);
        \draw[ultra thick] (0+.1,0) rectangle (2-.1,1);
        \draw[ultra thick] (2+.1,0) rectangle (4-.1,1);
        \draw[ultra thick] (4+.1,0) rectangle (6-.1,1);
        \draw[ultra thick] (6+.1,0) rectangle (8-.1,1);
        \draw[ultra thick] (8+.1,0) rectangle (10-.1,1);
        \draw[ultra thick] (10+.1,0) rectangle (12-.1,1);
        \draw[ultra thick] (12+.1,0) rectangle (14-.1,1);
        \draw[ultra thick] (14+.1,0) rectangle (16-.1,1);
        \draw[ultra thick] (16+.1,0) rectangle (18-.1,1);
        \draw[ultra thick] (18+.1,0) rectangle (20-.1,1);
        \draw[ultra thick] (20+.1,0) rectangle (22-.1,1);
        \draw[ultra thick] (22+.1,0) rectangle (24-.1,1);
        \draw[ultra thick] (24+.1,0) rectangle (26-.1,1);
        \draw[ultra thick] (26+.1,0) rectangle (28-.1,1);
        \draw[ultra thick] (28+.1,0) rectangle (30-.1,1);
        \draw[ultra thick] (30+.1,0) rectangle (32-.1,1);
    \end{tikzpicture}
    }
    \caption{Colored representation of $\tau^k(\square\ 12)$ for $k = 0 \twodots 4$ (for $p = 5$)}
    \label{fig:sample-tau-3}
\end{figure}

\subsubsection{Additivity}
\label{sec:tau-prop-additivity}

\begin{lemma}
$\forall m, x, y, x', y' \in \mathbb{F}_p$,
$\forall k \in \mathbb{N}$,
we have the following identities:

\begin{equation}
    \tau^k(\square\ xy) + \tau^k(\square\ x'y') = \tau^k(\square(x+x',y+y'))
\end{equation}
\begin{equation}
    m \cdot \tau^k(\square\ xy) = \tau^k(\square(mx,my))                                              \\
\end{equation}

where additions and scalar multiplications
on supertiles are performed componentwise on each endpoint
(see figure~\ref{fig:additivity-tau}).
\end{lemma}

\begin{figure}[H]
    \centering
    \resizebox{8cm}{!}
    {
    \begin{tikzpicture}
    	    \node[scale=5] at (16,8.5) {+};
    	    \node[scale=5] at (16,2.5) {=};
                \draw (0,12) -- (32,12);
        \node[scale=1] at (0,12) [circle,fill=black] {};
        \node[scale=4,below] at (0,12) {1};
        \node[scale=1] at (4,12) [circle,fill=black] {};
        \node[scale=4,below] at (4,12) {0};
        \node[scale=1] at (8,12) [circle,fill=black] {};
        \node[scale=4,below] at (8,12) {4};
        \node[scale=1] at (12,12) [circle,fill=black] {};
        \node[scale=4,below] at (12,12) {2};
        \node[scale=1] at (16,12) [circle,fill=black] {};
        \node[scale=4,below] at (16,12) {3};
        \node[scale=1] at (20,12) [circle,fill=black] {};
        \node[scale=4,below] at (20,12) {3};
        \node[scale=1] at (24,12) [circle,fill=black] {};
        \node[scale=4,below] at (24,12) {0};
        \node[scale=1] at (28,12) [circle,fill=black] {};
        \node[scale=4,below] at (28,12) {2};
        \node[scale=1] at (32,12) [circle,fill=black] {};
        \node[scale=4,below] at (32,12) {2};
        \draw (0,6) -- (32,6);
        \node[scale=1] at (0,6) [circle,fill=black] {};
        \node[scale=4,below] at (0,6) {0};
        \node[scale=1] at (4,6) [circle,fill=black] {};
        \node[scale=4,below] at (4,6) {3};
        \node[scale=1] at (8,6) [circle,fill=black] {};
        \node[scale=4,below] at (8,6) {3};
        \node[scale=1] at (12,6) [circle,fill=black] {};
        \node[scale=4,below] at (12,6) {1};
        \node[scale=1] at (16,6) [circle,fill=black] {};
        \node[scale=4,below] at (16,6) {3};
        \node[scale=1] at (20,6) [circle,fill=black] {};
        \node[scale=4,below] at (20,6) {4};
        \node[scale=1] at (24,6) [circle,fill=black] {};
        \node[scale=4,below] at (24,6) {1};
        \node[scale=1] at (28,6) [circle,fill=black] {};
        \node[scale=4,below] at (28,6) {4};
        \node[scale=1] at (32,6) [circle,fill=black] {};
        \node[scale=4,below] at (32,6) {3};
        \draw (0,0) -- (32,0);
        \node[scale=1] at (0,0) [circle,fill=black] {};
        \node[scale=4,below] at (0,0) {1};
        \node[scale=1] at (4,0) [circle,fill=black] {};
        \node[scale=4,below] at (4,0) {3};
        \node[scale=1] at (8,0) [circle,fill=black] {};
        \node[scale=4,below] at (8,0) {2};
        \node[scale=1] at (12,0) [circle,fill=black] {};
        \node[scale=4,below] at (12,0) {3};
        \node[scale=1] at (16,0) [circle,fill=black] {};
        \node[scale=4,below] at (16,0) {1};
        \node[scale=1] at (20,0) [circle,fill=black] {};
        \node[scale=4,below] at (20,0) {2};
        \node[scale=1] at (24,0) [circle,fill=black] {};
        \node[scale=4,below] at (24,0) {1};
        \node[scale=1] at (28,0) [circle,fill=black] {};
        \node[scale=4,below] at (28,0) {1};
        \node[scale=1] at (32,0) [circle,fill=black] {};
        \node[scale=4,below] at (32,0) {0};
    \end{tikzpicture}
    }
    \caption{$\tau^3(\square\ 12) + \tau^3(\square\ 03) = \tau^3(\square\ 10)$ (for $p = 5$)}
    \label{fig:additivity-tau}
\end{figure}
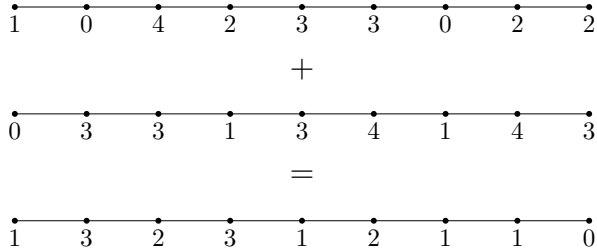

\subsubsection{Symmetries}
\label{sec:tau-prop-symmetries}

$\forall k \in \mathbb{N}$,
$\forall x, y \in \mathbb{F}_p$:
\begin{itemize}
\item if $x = y$
      then $\tau^k(\square\ xy)$
      has reflection symmetry:
      the values at two symmetrical endpoints
      with respect to the center of the supertile are equal,
\item if $x = -y$
      then $\tau^k(\square\ xy)$
      has reflection \quotes{odd} symmetry:
      the values at two symmetrical endpoints
      with respect to the center of the supertile are opposite and sum to $0$
      (see figures~\ref{fig:sample-tau} and \ref{fig:sample-tau-2}).
\end{itemize}

\subsubsection{Positions within supertiles}
\label{sec:tau-prop-positions}

\begin{theo}
\label{lemma:tau-tile-position}
\label{theo:tau-prop-rebuilding}
$\forall k \in \mathbb{N}$,
$\forall n \in 0 \twodots 2^k-1$,
let $w \in \{0,1\}^k$ be the binary expansion of $n$.
The $n$-th tile $t$ within a $k$-supertile $s$ satisfies:
\begin{equation}
    [t] = [s] \cdot [w]
\end{equation}
where:
\begin{equation}
    [0] = L, [1] = R
\end{equation}
\end{theo}

\begin{lemma}
\label{lemma:tau-unique}
A supertile $s$ is uniquely determined by one of its tile $t$
and the position $w$ of $t$ within $s$.
\end{lemma}

\subsubsection{Patterns of zeros}
\label{sec:tau-prop-patterns}

\begin{lemma}
A supertile $\tau^k(t)$ contains $\square\ 00$
if and only if $t = \square\ 00$.
\end{lemma}

\begin{lemma}
$\forall m \in \mathbb{F}_p^*$,
$\forall k \in \mathbb{N}$,
$\forall t \in P$,
the $0$'s at the endpoints of $m \cdot \tau^k(t)$ and $\tau^k(t)$
are exactly located at the same places.
\end{lemma}

\begin{remark}
Tiles with one $0$ remain in the endpoint holding that $0$
upon repeated applications of $\tau$:
\begin{equation}
    \tau(\square\ 0y) = \square\ 0y \ \square\ yy
\end{equation}
\begin{equation}
    \tau(\square\ x0) = \square\ xx \ \square\ x0
\end{equation}
\end{remark}

\begin{lemma}
$\forall k \in \mathbb{N}$,
$\forall t \in \{ \square\ 10, \square\ 01 \}$,
the values at two symmetrical endpoints in $\tau^k(t)$
with respect to the center of $\tau^k(t)$
cannot both equal $0$.
\end{lemma}

\subsubsection{Irreducibility}
\label{sec:tau-prop-irreducibility}

\begin{theo}
$(P^*, \tau)$ is irreducible.
\end{theo}

\begin{lemma}
\label{lemma:inx-tau}
$\forall x \in \mathbb{F}_p^*$,
$\forall y,y' \in \mathbb{F}_p$,
$\square\ xy \sim \square\ xy'$.
\end{lemma}

\begin{proof}
$\square\ xy \sim \square(x,x+y)$.\\
$\forall k \in \mathbb{N}$, applying this relation $k$ times,
we obtain $\square\ xy \sim \square(x,kx+y)$.\\
As $x$ is invertible, $k \rightarrow kx+y$ runs through $\mathbb{F}_p$.
\end{proof}

\begin{lemma}
\label{lemma:iny-tau}
$\forall x,x' \in \mathbb{F}_p^*$,
$\forall y \in \mathbb{F}_p^*$,
$\square\ xy \sim \square\ x'y$.
\end{lemma}

\begin{proof}
Similar to lemma~\ref{lemma:inx-tau}.
\end{proof}

\begin{lemma}
$\forall t \in P^*$,
$t \sim \square\ 11$ and $\square\ 11 \sim t$.
\end{lemma}

\begin{proof}
Let $[t] = [x\ y]$ (necessarily, $(x,y) \ne (0,0)$).\\
If $x \ne 0$, then:
\begin{itemize}
\item $\square\ xy \sim \square\ x1 \sim \square\ 11$,
\item $\square\ 11 \sim \square\ x1 \sim \square\ xy$.
\end{itemize}
If $y \ne 0$, then:
\begin{itemize}
\item $\square\ xy \sim \square\ 1y \sim \square\ 11$,
\item $\square\ 11 \sim \square\ 1y \sim \square\ xy$.
\end{itemize}

Hence, $\forall t, t' \in P^*$, $t \sim \square\ 11 \sim t'$.
\end{proof}

\subsubsection{Primitivity}
\label{sec:tau-prop-primitivity}

\begin{theo}
$(P^*, \tau)$ is primitive.
\end{theo}

\subsection{\texorpdfstring{Tiling $V_y$}{Tiling Vy}}
\label{sec:tau-tiling-v}

\subsubsection{Construction}
\label{sec:tau-tiling-v-cons}

\begin{lemma}
$\forall y \in \mathbb{F}_p$,
the sequence $\{ \tau^k(\square\ 0y) \}_{k \in \mathbb{N}}$ has a limit, say $V_y$,
that constitutes a tiling of the half-line.
\end{lemma}

\begin{proof}
The substitution $\tau$
can be seen as a $2$-uniform morphism
operating on the alphabet $P$.

As $\tau(\square\ 0y) = \square\ 0y \square\ yy$ starts with $\square\ 0y$,
the given limit exists (and is a pure morphic word).
\end{proof}

Figure~\ref{fig:stationary-tau} presents the first steps
of the construction of $V_3$ (for $p = 5$).

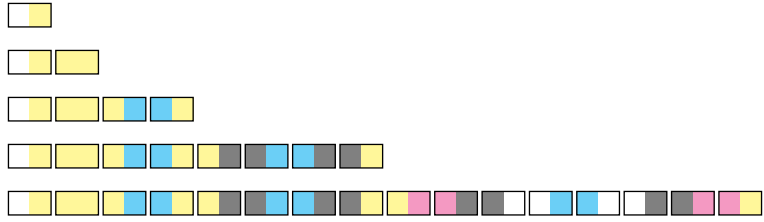
\begin{figure}[H]
    \centering
    \resizebox{10cm}{!}
    {
    \begin{tikzpicture}
                \fill[yellow!50] (1,8) rectangle (2-.1,9);
        \draw[ultra thick] (0+.1,8) rectangle (2-.1,9);
        \fill[yellow!50] (1,6) rectangle (2-.1,7);
        \fill[yellow!50] (2+.1,6) rectangle (3,7);
        \fill[yellow!50] (3,6) rectangle (4-.1,7);
        \draw[ultra thick] (0+.1,6) rectangle (2-.1,7);
        \draw[ultra thick] (2+.1,6) rectangle (4-.1,7);
        \fill[yellow!50] (1,4) rectangle (2-.1,5);
        \fill[yellow!50] (2+.1,4) rectangle (3,5);
        \fill[yellow!50] (3,4) rectangle (4-.1,5);
        \fill[yellow!50] (4+.1,4) rectangle (5,5);
        \fill[cyan!50] (5,4) rectangle (6-.1,5);
        \fill[cyan!50] (6+.1,4) rectangle (7,5);
        \fill[yellow!50] (7,4) rectangle (8-.1,5);
        \draw[ultra thick] (0+.1,4) rectangle (2-.1,5);
        \draw[ultra thick] (2+.1,4) rectangle (4-.1,5);
        \draw[ultra thick] (4+.1,4) rectangle (6-.1,5);
        \draw[ultra thick] (6+.1,4) rectangle (8-.1,5);
        \fill[yellow!50] (1,2) rectangle (2-.1,3);
        \fill[yellow!50] (2+.1,2) rectangle (3,3);
        \fill[yellow!50] (3,2) rectangle (4-.1,3);
        \fill[yellow!50] (4+.1,2) rectangle (5,3);
        \fill[cyan!50] (5,2) rectangle (6-.1,3);
        \fill[cyan!50] (6+.1,2) rectangle (7,3);
        \fill[yellow!50] (7,2) rectangle (8-.1,3);
        \fill[yellow!50] (8+.1,2) rectangle (9,3);
        \fill[black!50] (9,2) rectangle (10-.1,3);
        \fill[black!50] (10+.1,2) rectangle (11,3);
        \fill[cyan!50] (11,2) rectangle (12-.1,3);
        \fill[cyan!50] (12+.1,2) rectangle (13,3);
        \fill[black!50] (13,2) rectangle (14-.1,3);
        \fill[black!50] (14+.1,2) rectangle (15,3);
        \fill[yellow!50] (15,2) rectangle (16-.1,3);
        \draw[ultra thick] (0+.1,2) rectangle (2-.1,3);
        \draw[ultra thick] (2+.1,2) rectangle (4-.1,3);
        \draw[ultra thick] (4+.1,2) rectangle (6-.1,3);
        \draw[ultra thick] (6+.1,2) rectangle (8-.1,3);
        \draw[ultra thick] (8+.1,2) rectangle (10-.1,3);
        \draw[ultra thick] (10+.1,2) rectangle (12-.1,3);
        \draw[ultra thick] (12+.1,2) rectangle (14-.1,3);
        \draw[ultra thick] (14+.1,2) rectangle (16-.1,3);
        \fill[yellow!50] (1,0) rectangle (2-.1,1);
        \fill[yellow!50] (2+.1,0) rectangle (3,1);
        \fill[yellow!50] (3,0) rectangle (4-.1,1);
        \fill[yellow!50] (4+.1,0) rectangle (5,1);
        \fill[cyan!50] (5,0) rectangle (6-.1,1);
        \fill[cyan!50] (6+.1,0) rectangle (7,1);
        \fill[yellow!50] (7,0) rectangle (8-.1,1);
        \fill[yellow!50] (8+.1,0) rectangle (9,1);
        \fill[black!50] (9,0) rectangle (10-.1,1);
        \fill[black!50] (10+.1,0) rectangle (11,1);
        \fill[cyan!50] (11,0) rectangle (12-.1,1);
        \fill[cyan!50] (12+.1,0) rectangle (13,1);
        \fill[black!50] (13,0) rectangle (14-.1,1);
        \fill[black!50] (14+.1,0) rectangle (15,1);
        \fill[yellow!50] (15,0) rectangle (16-.1,1);
        \fill[yellow!50] (16+.1,0) rectangle (17,1);
        \fill[magenta!50] (17,0) rectangle (18-.1,1);
        \fill[magenta!50] (18+.1,0) rectangle (19,1);
        \fill[black!50] (19,0) rectangle (20-.1,1);
        \fill[black!50] (20+.1,0) rectangle (21,1);
        \fill[cyan!50] (23,0) rectangle (24-.1,1);
        \fill[cyan!50] (24+.1,0) rectangle (25,1);
        \fill[black!50] (27,0) rectangle (28-.1,1);
        \fill[black!50] (28+.1,0) rectangle (29,1);
        \fill[magenta!50] (29,0) rectangle (30-.1,1);
        \fill[magenta!50] (30+.1,0) rectangle (31,1);
        \fill[yellow!50] (31,0) rectangle (32-.1,1);
        \draw[ultra thick] (0+.1,0) rectangle (2-.1,1);
        \draw[ultra thick] (2+.1,0) rectangle (4-.1,1);
        \draw[ultra thick] (4+.1,0) rectangle (6-.1,1);
        \draw[ultra thick] (6+.1,0) rectangle (8-.1,1);
        \draw[ultra thick] (8+.1,0) rectangle (10-.1,1);
        \draw[ultra thick] (10+.1,0) rectangle (12-.1,1);
        \draw[ultra thick] (12+.1,0) rectangle (14-.1,1);
        \draw[ultra thick] (14+.1,0) rectangle (16-.1,1);
        \draw[ultra thick] (16+.1,0) rectangle (18-.1,1);
        \draw[ultra thick] (18+.1,0) rectangle (20-.1,1);
        \draw[ultra thick] (20+.1,0) rectangle (22-.1,1);
        \draw[ultra thick] (22+.1,0) rectangle (24-.1,1);
        \draw[ultra thick] (24+.1,0) rectangle (26-.1,1);
        \draw[ultra thick] (26+.1,0) rectangle (28-.1,1);
        \draw[ultra thick] (28+.1,0) rectangle (30-.1,1);
        \draw[ultra thick] (30+.1,0) rectangle (32-.1,1);
    \end{tikzpicture}
    }
    \caption{Colored representation of $\tau^k(\square\ 03)$ for $k = 0 \twodots 4$ (for $p = 5$)}
    \label{fig:stationary-tau}
\end{figure}

\begin{lemma}
$V_y$ is directly related to Dijkstra's \quotes{fusc} function:
the ($0$-based) $n$-th tile in $V_y$ equals:
\begin{equation}
    V_y(n) = y \cdot \square(fusc(n)\ mod\ p, fusc(n+1)\ mod\ p)
\end{equation}
\end{lemma}

\begin{proof}
The construction of $V_y$ mimics that of Stern's diatomic series
(see figure~\ref{fig:stern}).
\end{proof}

\begin{remark}
$\forall m, y, y' \in \mathbb{F}_p$:
\begin{equation}
    V_y + V_{y'} = V_{y+y'}
\end{equation}
\begin{equation}
    m \cdot V_y = V_{my}
\end{equation}
\end{remark}

\subsubsection{Automaticity}
\label{sec:tau-tiling-v-automaticity}

\begin{theo}
$\forall y \in \mathbb{F}_p$,
the tiling $V_y$ is $2$-automatic.
\end{theo}

\begin{proof}
As a fixed point of a $2$-uniform morphism, $V_y$ is $2$-automatic.
\end{proof}

Let $O$ be the deterministic finite automaton with the following characteristics:
\begin{itemize}
\item the set of states is $\mathbb{F}_p^{2 \times 2}$,
\item the input alphabet is $\{0,1\}$,
\item the initial state is $I_2$,
\item the transitions from state $m$ are performed as follows:\\
\begin{tabular}{|c|l|}
    \hline
    \rule[-1ex]{0pt}{2.5ex} Input symbol & Output state \\
    \hline
    \rule[-1ex]{0pt}{2.5ex} $0$ & $L^{-1} \cdot m \cdot L$ \\
    \rule[-1ex]{0pt}{2.5ex} $1$ & $L^{-1} \cdot m \cdot R$ \\
    \hline
\end{tabular}
\item $\forall w \in \{0,1\}^*$,
      the final state of the automaton after processing the input $w$, say $m$,
      encodes a tile $t'$ as follows:
\begin{equation}
    [t'] = [0\ y] \cdot m
\end{equation}
\end{itemize}

$V_y$ can be computed with $O$.

\subsubsection{Self-similarity}
\label{sec:tau-tiling-v-self-similarity}

\begin{theo}
$\forall y \in \mathbb{F}_p$,
the tiling $V_y$ is self-similar.
\end{theo}

\begin{proof}
\begin{equation}
    \tau(V_y) = V_y
\end{equation}
\end{proof}

\subsubsection{Nonperiodicity}
\label{sec:tau-tiling-v-nonperiodicity}

\begin{theo}
$\forall y \in \mathbb{F}_p^*$,
the tiling $V_y$ is nonperiodic.
\end{theo}

\begin{proof}
\begin{remark}
As $y$ is invertible,
the periodicity of $V_y$ is equivalent
to the periodicity of the \quotes{fusc} function reduced modulo $p$.
\end{remark}

Suppose that $V_y$ is periodic.

\begin{lemma}
The least period of $V_y$ would be odd.
\end{lemma}

If $V_y$ is $2m$-periodic with $m > 0$, then:
\begin{equation}
    fusc(n) = fusc(2n) = fusc(2n + 2m) = fusc(2(n+m)) = fusc(n + m)
\end{equation}
So we can assume that the period is odd.

\vspace{\baselineskip}

As $V_y$ is primitive and self-similar,
$\forall k \in \mathbb{N}$,
$V_y$ contains $\tau^k(\square(\frac{1}{2},\frac{1}{2}))$.

$\forall k \ge 0$:
\begin{itemize}
\item the two central tiles of $\tau^{k+1}(\square(\frac{1}{2},\frac{1}{2}))$
      are $\square(k+\frac{1}{2},1)$ and $\square(1,k+\frac{1}{2})$,
\item in particular, when $k = p - \frac{1}{2}$,
      we have the central tiles $\square\ 01$ and $\square\ 10$ (in that order).
\end{itemize}

$\forall k' > 0$:
\begin{itemize}
\item $\tau^{k'-1}(\square\ 01\ \square\ 10)$ contains $2^{k'}$ tiles,
\item the leftmost being $t = \square\ 01$
      and the rightmost being $t' = \square\ 10$,
\item the tiles $t$ and $t'$ are at distance $2^{k'}-1$,
      and $t \ne t'$.
\end{itemize}

Suppose that $V_y$ is $m$-periodic (with $m$ odd):
\begin{itemize}
\item $m$ divides $2^{ord_2(m)}-1$,
\item so the leftmost and rightmost tiles
      in $\tau^{ord_2(m)-1}(\square\ 01\ \square\ 10)$ must be the same, a contradiction,
\item such supertiles appear infinitely many times in $V_y$,
\item hence $V_y$ is neither periodic nor eventually periodic.
\end{itemize}
\end{proof}

\begin{corrollary}
The function \quotes{$k \rightarrow fusc(k) \mod p$} is not eventually periodic.
\end{corrollary}

\subsection{\texorpdfstring{Tiling $W_t$}{Tiling Wt}}
\label{sec:tau-tiling-w}

\subsubsection{Construction}
\label{sec:tau-tiling-w-cons}

\begin{lemma}
\begin{equation}
    L^p = I_2
\end{equation}
\end{lemma}

\begin{lemma}
$\forall t \in P$,
the sequence $\{ \tau^{kp}(t) \}_{k \in \mathbb{N}}$ has a limit, say $W_t$,
that constitutes a tiling of the half-line.
\end{lemma}

\begin{proof}
The substitution $\tau^p$
can be seen as a $2^p$-uniform morphism
operating on the alphabet $P$.

As $\tau^p(t)$ starts with $t$,
the given limit exists (and is a pure morphic word).
\end{proof}

Figure~\ref{fig:stationary-tau-w} presents the first steps
of the construction of $W_{\square\ 10}$ (for $p = 5$).

\begin{figure}[H]
    \centering
    \resizebox{10cm}{!}
    {
    \begin{tikzpicture}
                \fill[cyan!50] (0+.1,2) rectangle (1,3);
        \draw[ultra thick] (0+.1,2) rectangle (2-.1,3);
        \fill[cyan!50] (0+.1,0) rectangle (1,1);
        \fill[black!50] (3,0) rectangle (4-.1,1);
        \fill[black!50] (4+.1,0) rectangle (5,1);
        \fill[magenta!50] (5,0) rectangle (6-.1,1);
        \fill[magenta!50] (6+.1,0) rectangle (7,1);
        \fill[yellow!50] (7,0) rectangle (8-.1,1);
        \fill[yellow!50] (8+.1,0) rectangle (9,1);
        \fill[yellow!50] (9,0) rectangle (10-.1,1);
        \fill[yellow!50] (10+.1,0) rectangle (11,1);
        \fill[magenta!50] (13,0) rectangle (14-.1,1);
        \fill[magenta!50] (14+.1,0) rectangle (15,1);
        \fill[magenta!50] (15,0) rectangle (16-.1,1);
        \fill[magenta!50] (16+.1,0) rectangle (17,1);
        \fill[magenta!50] (17,0) rectangle (18-.1,1);
        \fill[magenta!50] (18+.1,0) rectangle (19,1);
        \fill[yellow!50] (21,0) rectangle (22-.1,1);
        \fill[yellow!50] (22+.1,0) rectangle (23,1);
        \fill[yellow!50] (23,0) rectangle (24-.1,1);
        \fill[yellow!50] (24+.1,0) rectangle (25,1);
        \fill[magenta!50] (25,0) rectangle (26-.1,1);
        \fill[magenta!50] (26+.1,0) rectangle (27,1);
        \fill[black!50] (27,0) rectangle (28-.1,1);
        \fill[black!50] (28+.1,0) rectangle (29,1);
        \fill[cyan!50] (31,0) rectangle (32-.1,1);
        \fill[cyan!50] (32+.1,0) rectangle (33,1);
        \fill[black!50] (33,0) rectangle (34-.1,1);
        \fill[black!50] (34+.1,0) rectangle (35,1);
        \fill[yellow!50] (35,0) rectangle (36-.1,1);
        \fill[yellow!50] (36+.1,0) rectangle (37,1);
        \fill[magenta!50] (39,0) rectangle (40-.1,1);
        \fill[magenta!50] (40+.1,0) rectangle (41,1);
        \fill[yellow!50] (43,0) rectangle (44-.1,1);
        \fill[yellow!50] (44+.1,0) rectangle (45,1);
        \fill[black!50] (45,0) rectangle (46-.1,1);
        \fill[black!50] (46+.1,0) rectangle (47,1);
        \fill[cyan!50] (47,0) rectangle (48-.1,1);
        \fill[cyan!50] (48+.1,0) rectangle (49,1);
        \fill[yellow!50] (49,0) rectangle (50-.1,1);
        \fill[yellow!50] (50+.1,0) rectangle (51,1);
        \fill[magenta!50] (51,0) rectangle (52-.1,1);
        \fill[magenta!50] (52+.1,0) rectangle (53,1);
        \fill[yellow!50] (53,0) rectangle (54-.1,1);
        \fill[yellow!50] (54+.1,0) rectangle (55,1);
        \fill[cyan!50] (55,0) rectangle (56-.1,1);
        \fill[cyan!50] (56+.1,0) rectangle (57,1);
        \fill[magenta!50] (57,0) rectangle (58-.1,1);
        \fill[magenta!50] (58+.1,0) rectangle (59,1);
        \fill[cyan!50] (59,0) rectangle (60-.1,1);
        \fill[cyan!50] (60+.1,0) rectangle (61,1);
        \fill[cyan!50] (61,0) rectangle (62-.1,1);
        \fill[cyan!50] (62+.1,0) rectangle (63,1);
        \draw[ultra thick] (0+.1,0) rectangle (2-.1,1);
        \draw[ultra thick] (2+.1,0) rectangle (4-.1,1);
        \draw[ultra thick] (4+.1,0) rectangle (6-.1,1);
        \draw[ultra thick] (6+.1,0) rectangle (8-.1,1);
        \draw[ultra thick] (8+.1,0) rectangle (10-.1,1);
        \draw[ultra thick] (10+.1,0) rectangle (12-.1,1);
        \draw[ultra thick] (12+.1,0) rectangle (14-.1,1);
        \draw[ultra thick] (14+.1,0) rectangle (16-.1,1);
        \draw[ultra thick] (16+.1,0) rectangle (18-.1,1);
        \draw[ultra thick] (18+.1,0) rectangle (20-.1,1);
        \draw[ultra thick] (20+.1,0) rectangle (22-.1,1);
        \draw[ultra thick] (22+.1,0) rectangle (24-.1,1);
        \draw[ultra thick] (24+.1,0) rectangle (26-.1,1);
        \draw[ultra thick] (26+.1,0) rectangle (28-.1,1);
        \draw[ultra thick] (28+.1,0) rectangle (30-.1,1);
        \draw[ultra thick] (30+.1,0) rectangle (32-.1,1);
        \draw[ultra thick] (32+.1,0) rectangle (34-.1,1);
        \draw[ultra thick] (34+.1,0) rectangle (36-.1,1);
        \draw[ultra thick] (36+.1,0) rectangle (38-.1,1);
        \draw[ultra thick] (38+.1,0) rectangle (40-.1,1);
        \draw[ultra thick] (40+.1,0) rectangle (42-.1,1);
        \draw[ultra thick] (42+.1,0) rectangle (44-.1,1);
        \draw[ultra thick] (44+.1,0) rectangle (46-.1,1);
        \draw[ultra thick] (46+.1,0) rectangle (48-.1,1);
        \draw[ultra thick] (48+.1,0) rectangle (50-.1,1);
        \draw[ultra thick] (50+.1,0) rectangle (52-.1,1);
        \draw[ultra thick] (52+.1,0) rectangle (54-.1,1);
        \draw[ultra thick] (54+.1,0) rectangle (56-.1,1);
        \draw[ultra thick] (56+.1,0) rectangle (58-.1,1);
        \draw[ultra thick] (58+.1,0) rectangle (60-.1,1);
        \draw[ultra thick] (60+.1,0) rectangle (62-.1,1);
        \draw[ultra thick] (62+.1,0) rectangle (64-.1,1);
    \end{tikzpicture}
    }
    \caption{Colored representation of $\tau^{kp}(\square\ 10)$ for $k = 0, 1$ (for $p = 5$)}
    \label{fig:stationary-tau-w}
\end{figure}

\begin{lemma}
$\forall y \in \mathbb{F}_p$:
\begin{equation}
    V_y = W_{\square\ 0y}
\end{equation}
\end{lemma}

\begin{remark}
$\forall m \in \mathbb{F}_p$,
$\forall t, t' \in P$:
\begin{equation}
    W_t + W_{t'} = W_{t+t'}
\end{equation}
\begin{equation}
    m \cdot W_t = W_{mt}
\end{equation}
\end{remark}

\subsubsection{Automaticity}
\label{sec:tau-tiling-w-automaticity}

\begin{theo}
$\forall t \in P$,
$W_t$ is $2$-automatic.
\end{theo}

We can use the same automaton as in section~\ref{sec:tau-tiling-v-automaticity}
to compute $W_t$.

\subsubsection{Self-similarity}
\label{sec:tau-tiling-w-self-similarity}

\begin{theo}
$\forall t \in P$,
$W_t$ is self-similar.
\end{theo}

\begin{proof}
\begin{equation}
    \tau^p(W_t) = W_t
\end{equation}
\end{proof}

\subsubsection{Nonperiodicity}
\label{sec:tau-tiling-w-nonperiodicity}

\begin{theo}
$\forall t \in P^*$,
$W_t$ is nonperiodic.
\end{theo}

\begin{proof}
Suppose that $W_t$ is $m$-periodic for some $m > 0$.

We know that $V_1$ is nonperiodic, so for some $n \in \mathbb{N}$,
$V_1(n) \ne V_1(n + m)$.

Let $k \in \mathbb{N}$ be such that $n + m < 2^k$.

As $W_t$ is primitive and self-similar,
it contains infinitely many copies of $\tau^k(\square\ 01)$.
This $k$-supertile has two different tiles at distance $m$,
so $W_t$ cannot be periodic or eventually periodic.
\end{proof}

\section{Acknowledgments}
I would like to express my very great appreciation
to Mihai Prunescu, Neil Sloane and Dirk Frettl\"oh
for their valuable and constructive suggestions.

\clearpage

\appendix

\section{Code sample (PARI/GP)}

The following PARI/GP program
creates an image of $\sigma^8(\bigtriangleup 120)$ (for $p = 3$).
This image is used on the front page.

\begin{Verbatim}
\\ draw point at triangular coordinates z
\\ with color v (0 = white)
dot(z,v) = {
    if (v,
        plotcolor(w, [ 2, /*blue*/
                       4, /*red*/
                       5, /*cornsilk*/
                       1, /*black*/
                       3, /*sienna*/
                       7  /*gainsborough*/ ][lift(v)]);
        plotpoints(w, [2*real(z)+imag(z)], [imag(z)]);
    );
}

\\ draw tile or supertile
\\       z(c)
\\      / \
\\     /   \
\\ x(a)-----y(b)
stern(a,b,c, x,y,z, k) = {
    if (k==0,
        \\ tile
        dot(a,x);
        dot(b,y);
        dot(c,z),
        \\ k-supertile
        stern(a,(a+b)/2,(a+c)/2, x,x+y,x+z, k-1);
        stern((a+b)/2,b,(b+c)/2, x+y,y,y+z, k-1);
        stern((a+c)/2,(b+c)/2,c, x+z,y+z,z, k-1);
        stern((a+c)/2,(b+c)/2,(a+b)/2, x+z,y+z,x+y, k-1);
    );
}

{
    p = 3; \\ 3, 5 or 7
    my (k = 8, x = 1, y = 2, z = 0);

    \\ draw sigma^k(/\xyz) for current odd prime number p
    plotinit(w=0, width=2048, height=2048);
    plotscale(w, 0,width, 0,height);
    stern(0,2^k,I*2^k, Mod(x,p),Mod(y,p),Mod(z,p),k);
    plotdraw(w);
    plotkill(w);
}
\end{Verbatim}

\section{Other substitutions}
\label{sec:other}

This section describes substitutions similar to $\sigma$
that operate on triangles or on squares
with corners decorated by elements of $F_p$
and exhibit interesting graphical features.

For all of these substitutions, corners that meet at the same point
hold the same value.

When operating on equilateral triangles, as with $\sigma$,
the substitution rule for $\bigtriangledown xyz$
can be obtained by turning a half turn that for $\bigtriangleup xyz$.

\subsection{\texorpdfstring{Substitution $\sigma_1$}{Substitution sigma1}}
\label{sec:variant-1}

The substitution $\sigma_1$ operates on equilateral triangles
and transforms a tile into four tiles
as depicted in figure \ref{fig:variant-1-r}.
This substitution uses multiplication instead of addition.

See figure \ref{fig:variant-1} for an example of supertile.

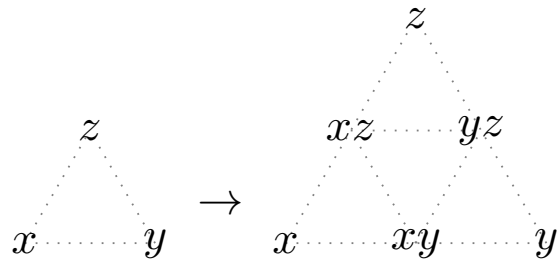
\begin{figure}[H]
    \centering
    \resizebox{8cm}{!}
    {
    \begin{tikzpicture}
    	\begin{scope}[yscale=.87,xslant=.5]
        \node at (1+1/3,1/3) {$\rightarrow$};
        \draw[dotted,gray] (0,0) -- (1,0);
        \draw[dotted,gray] (0,0) -- (0,1);
        \draw[dotted,gray] (1,0) -- (0,1);
        \draw[dotted,gray] (2,0) -- (4,0);
        \draw[dotted,gray] (2,0) -- (2,2);
        \draw[dotted,gray] (4,0) -- (2,2);
        \draw[dotted,gray] (2,1) -- (3,1);
        \draw[dotted,gray] (3,0) -- (3,1);
        \draw[dotted,gray] (3,0) -- (2,1);
        \node at (0,0) {$x$};
        \node at (0,1) {$z$};
        \node at (1,0) {$y$};
        \node at (2,0) {$x$};
        \node at (2,1) {$xz$};
        \node at (2,2) {$z$};
        \node at (3,0) {$xy$};
        \node at (3,1) {$yz$};
        \node at (4,0) {$y$};
    	\end{scope}
    \end{tikzpicture}
    }
    \caption{The substitution rule for $\sigma_1$}
    \label{fig:variant-1-r}
\end{figure}

\begin{figure}[b]
    \begin{center}
    \vstretch{1.732}{
        \includegraphics[height=7cm, keepaspectratio]{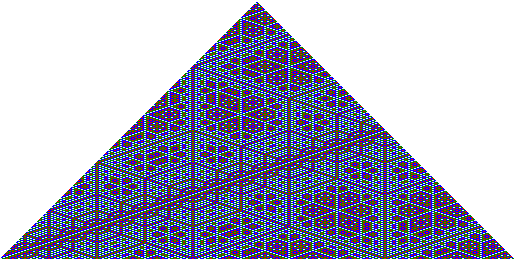}
    }
    \caption{Colored representation of $\sigma_1^8(\bigtriangleup 123)$ (for $p = 5$)}
    \label{fig:variant-1}
    \end{center}
\end{figure}

\subsection{\texorpdfstring{Substitution $\sigma_2$}{Substitution sigma2}}
\label{sec:variant-2}

The substitution $\sigma_2$ operates on equilateral triangles
and transforms a tile into nine tiles
as depicted in figure \ref{fig:variant-2-r}.

This substitution is related to the OEIS sequence \seqnum{A356097}.

See figure \ref{fig:variant-2} for an example of supertile.

\begin{figure}[H]
    \centering
    \resizebox{8cm}{!}
    {
    \begin{tikzpicture}
    	\begin{scope}[yscale=.87,xslant=.5]
        \node at (1+1/3,1/3) {$\rightarrow$};
        \draw[dotted,gray] (0,0) -- (1,0);
        \draw[dotted,gray] (0,0) -- (0,1);
        \draw[dotted,gray] (1,0) -- (0,1);
        \node at (0,0) {$x$};
        \node at (1,0) {$y$};
        \node at (0,1) {$z$};
        \draw[dotted,gray] (2,0) -- (5,0);
        \draw[dotted,gray] (2,0) -- (2,3);
        \draw[dotted,gray] (5,0) -- (2,3);
        \draw[dotted,gray] (2,1) -- (4,1);
        \draw[dotted,gray] (3,0) -- (3,2);
        \draw[dotted,gray] (4,0) -- (2,2);
        \draw[dotted,gray] (2,2) -- (3,2);
        \draw[dotted,gray] (4,0) -- (4,1);
        \draw[dotted,gray] (3,0) -- (2,1);
        \node at (2,0) {$x$};
        \node at (3,0) {$x$};
        \node at (2,1) {$x$};
        \node[scale=.8] at (3,1) {$x+y+z$};
        \node at (2,2) {$z$};
        \node at (3,2) {$z$};
        \node at (2,3) {$z$};
        \node at (4,0) {$y$};
        \node at (4,1) {$y$};
        \node at (5,0) {$y$};
    	\end{scope}
    \end{tikzpicture}
    }
    \caption{The substitution rule for $\sigma_2$}
    \label{fig:variant-2-r}
\end{figure}
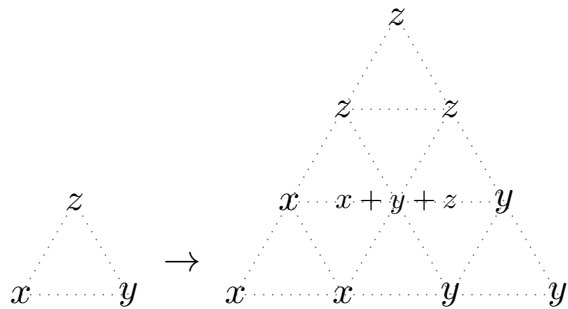

\begin{figure}[b]
    \begin{center}
    \vstretch{1.732}{
        \includegraphics[height=7cm, keepaspectratio]{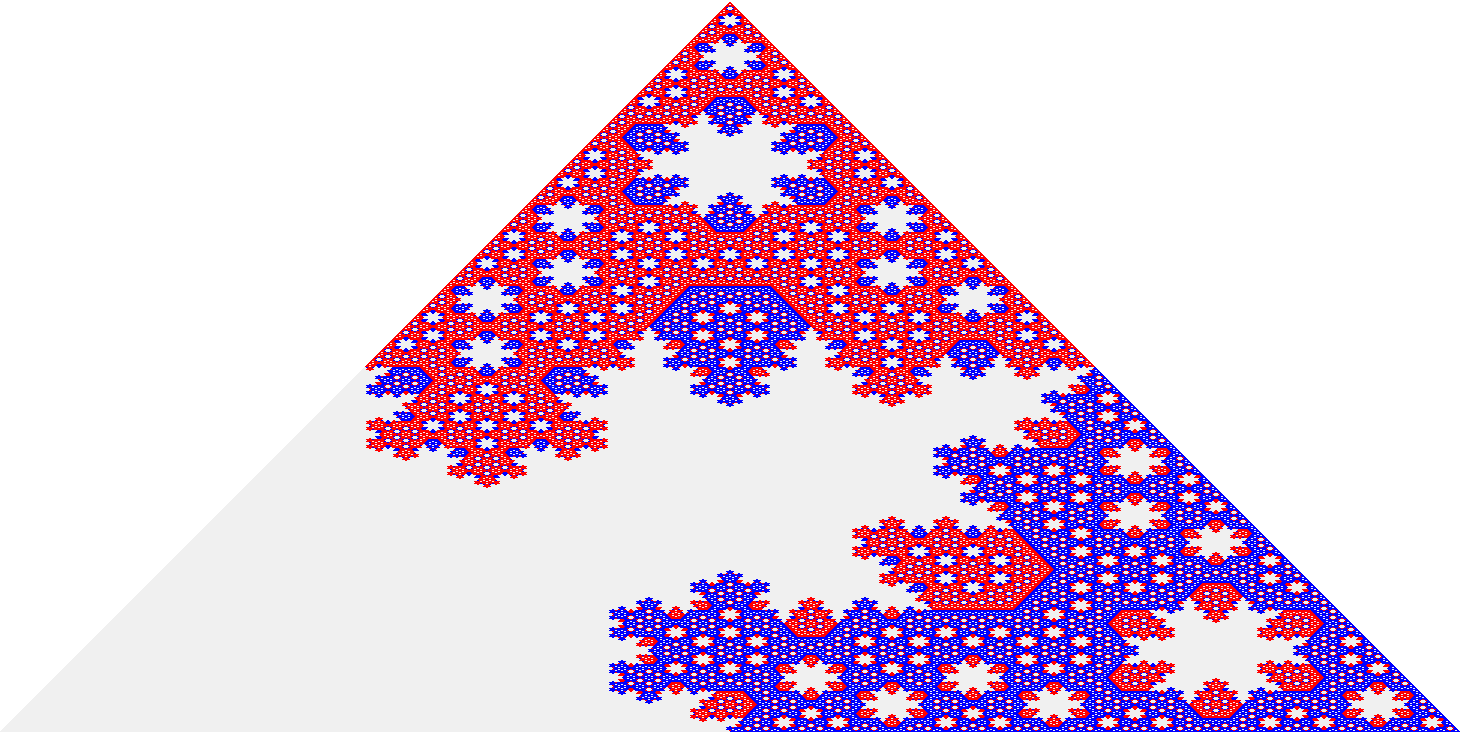}
    }
    \caption{Colored representation of $\sigma_2^6(\bigtriangleup 012)$ (for $p = 3$)}
    \label{fig:variant-2}
    \end{center}
\end{figure}

\subsection{\texorpdfstring{Substitution $\sigma_3$}{Substitution sigma3}}
\label{sec:variant-3}

The substitution $\sigma_3$ operates on equilateral triangles
and transforms a tile into nine tiles
as depicted in figure \ref{fig:variant-3-r}.

This substitution is related to the OEIS sequence \seqnum{A356096}.

See figure \ref{fig:variant-3} for an example of supertile.

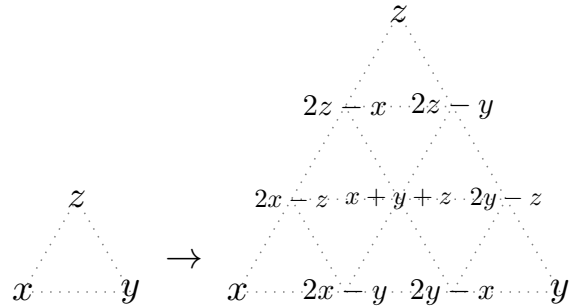
\begin{figure}[H]
    \centering
    \resizebox{8cm}{!}
    {
    \begin{tikzpicture}
    	\begin{scope}[yscale=.87,xslant=.5]
        \node at (1+1/3,1/3) {$\rightarrow$};
        \draw[dotted,gray] (0,0) -- (1,0);
        \draw[dotted,gray] (0,0) -- (0,1);
        \draw[dotted,gray] (1,0) -- (0,1);
        \node at (0,0) {$x$};
        \node at (1,0) {$y$};
        \node at (0,1) {$z$};
        \draw[dotted,gray] (2,0) -- (5,0);
        \draw[dotted,gray] (2,0) -- (2,3);
        \draw[dotted,gray] (5,0) -- (2,3);
        \draw[dotted,gray] (2,1) -- (4,1);
        \draw[dotted,gray] (3,0) -- (3,2);
        \draw[dotted,gray] (4,0) -- (2,2);
        \draw[dotted,gray] (2,2) -- (3,2);
        \draw[dotted,gray] (4,0) -- (4,1);
        \draw[dotted,gray] (3,0) -- (2,1);
        \node at (2,0) {$x$};
        \node[scale=.8] at (3,0) {$2x-y$};
        \node[scale=.7] at (2,1) {$2x-z$};
        \node[scale=.7] at (3,1) {$x+y+z$};
        \node[scale=.8] at (2,2) {$2z-x$};
        \node[scale=.8] at (3,2) {$2z-y$};
        \node at (2,3) {$z$};
        \node[scale=.8] at (4,0) {$2y-x$};
        \node[scale=.7] at (4,1) {$2y-z$};
        \node at (5,0) {$y$};
    	\end{scope}
    \end{tikzpicture}
    }
    \caption{The substitution rule for $\sigma_3$}
    \label{fig:variant-3-r}
\end{figure}

\begin{figure}[b]
    \begin{center}
    \vstretch{1.732}{
        \includegraphics[height=7cm, keepaspectratio]{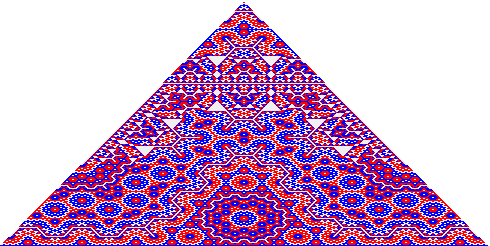}
    }
    \caption{Colored representation of $\sigma_3^5(\bigtriangleup 112)$ (for $p = 3$)}
    \label{fig:variant-3}
    \end{center}
\end{figure}

\subsection{\texorpdfstring{Substitution $\sigma_4$}{Substitution sigma4}}
\label{sec:variant-4}

The substitution $\sigma_4$ operates on equilateral triangles
and transforms a tile into nine tiles
as depicted in figure \ref{fig:variant-4-r}.

This substitution is related to the OEIS sequence \seqnum{A356002}.

See figure \ref{fig:variant-4} for an example of supertile.

\begin{figure}[H]
    \centering
    \resizebox{8cm}{!}
    {
    \begin{tikzpicture}
    	\begin{scope}[yscale=.87,xslant=.5]
        \node at (1+1/3,1/3) {$\rightarrow$};
        \draw[dotted,gray] (0,0) -- (1,0);
        \draw[dotted,gray] (0,0) -- (0,1);
        \draw[dotted,gray] (1,0) -- (0,1);
        \node at (0,0) {$x$};
        \node at (1,0) {$y$};
        \node at (0,1) {$z$};
        \draw[dotted,gray] (2,0) -- (5,0);
        \draw[dotted,gray] (2,0) -- (2,3);
        \draw[dotted,gray] (5,0) -- (2,3);
        \draw[dotted,gray] (2,1) -- (4,1);
        \draw[dotted,gray] (3,0) -- (3,2);
        \draw[dotted,gray] (4,0) -- (2,2);
        \draw[dotted,gray] (2,2) -- (3,2);
        \draw[dotted,gray] (4,0) -- (4,1);
        \draw[dotted,gray] (3,0) -- (2,1);
        \node at (2,0) {$x$};
        \node[scale=.8] at (3,0) {$2x+y$};
        \node[scale=.7] at (2,1) {$2x+z$};
        \node[scale=.7] at (3,1) {$x+y+z$};
        \node[scale=.8] at (2,2) {$x+2z$};
        \node[scale=.8] at (3,2) {$y+2z$};
        \node at (2,3) {$z$};
        \node[scale=.8] at (4,0) {$x+2y$};
        \node[scale=.7] at (4,1) {$2y+z$};
        \node at (5,0) {$y$};
    	\end{scope}
    \end{tikzpicture}
    }
    \caption{The substitution rule for $\sigma_4$}
    \label{fig:variant-4-r}
\end{figure}
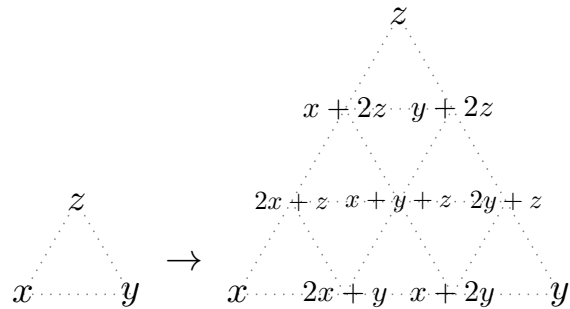

\begin{figure}[b]
    \begin{center}
    \vstretch{1.732}{
        \includegraphics[height=7cm, keepaspectratio]{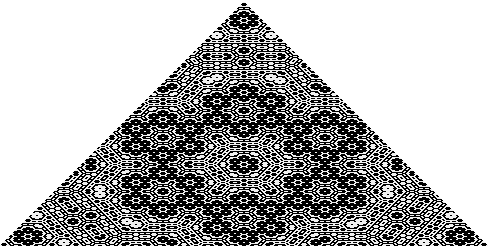}
    }
    \caption{Colored representation of $\sigma_4^5(\bigtriangleup 111)$ (for $p = 4$)}
    \label{fig:variant-4}
    \end{center}
\end{figure}

\subsection{\texorpdfstring{Substitution $\sigma_5$}{Substitution sigma5}}
\label{sec:variant-5}

The substitution $\sigma_5$ operates on equilateral triangles
and transforms a tile into nine tiles
as depicted in figure \ref{fig:variant-5-r}.

This substitution is related to the OEIS sequence \seqnum{A356098}.

See figure \ref{fig:variant-5} for an example of supertile.

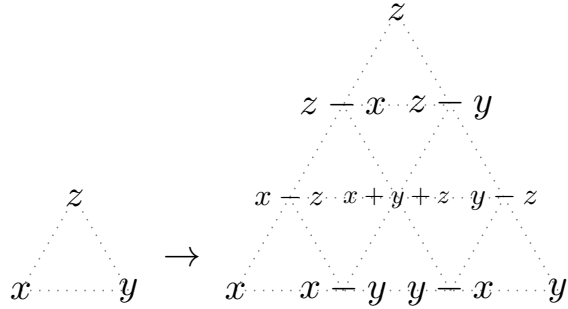
\begin{figure}[H]
    \centering
    \resizebox{8cm}{!}
    {
    \begin{tikzpicture}
    	\begin{scope}[yscale=.87,xslant=.5]
        \node at (1+1/3,1/3) {$\rightarrow$};
        \draw[dotted,gray] (0,0) -- (1,0);
        \draw[dotted,gray] (0,0) -- (0,1);
        \draw[dotted,gray] (1,0) -- (0,1);
        \node at (0,0) {$x$};
        \node at (1,0) {$y$};
        \node at (0,1) {$z$};
        \draw[dotted,gray] (2,0) -- (5,0);
        \draw[dotted,gray] (2,0) -- (2,3);
        \draw[dotted,gray] (5,0) -- (2,3);
        \draw[dotted,gray] (2,1) -- (4,1);
        \draw[dotted,gray] (3,0) -- (3,2);
        \draw[dotted,gray] (4,0) -- (2,2);
        \draw[dotted,gray] (2,2) -- (3,2);
        \draw[dotted,gray] (4,0) -- (4,1);
        \draw[dotted,gray] (3,0) -- (2,1);
        \node at (2,0) {$x$};
        \node at (3,0) {$x-y$};
        \node[scale=.8] at (2,1) {$x-z$};
        \node[scale=.7] at (3,1) {$x+y+z$};
        \node at (2,2) {$z-x$};
        \node at (3,2) {$z-y$};
        \node at (2,3) {$z$};
        \node at (4,0) {$y-x$};
        \node[scale=.8] at (4,1) {$y-z$};
        \node at (5,0) {$y$};
    	\end{scope}
    \end{tikzpicture}
    }
    \caption{The substitution rule for $\sigma_5$}
    \label{fig:variant-5-r}
\end{figure}

\begin{figure}[b]
    \begin{center}
    \vstretch{1.732}{
        \includegraphics[height=7cm, keepaspectratio]{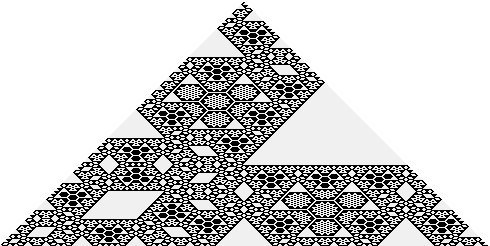}
    }
    \caption{Colored representation of $\sigma_5^5(\bigtriangleup 011)$ (for $p = 2$)}
    \label{fig:variant-5}
    \end{center}
\end{figure}

\subsection{\texorpdfstring{Substitution $\sigma_6$}{Substitution sigma6}}
\label{sec:variant-6}

The substitution $\sigma_6$ operates on squares
and transforms a tile into nine tiles
as depicted in figure \ref{fig:variant-6-r}.

This substitution is related to the OEIS sequence \seqnum{A356245}.

See figure \ref{fig:variant-6} for an example of supertile.

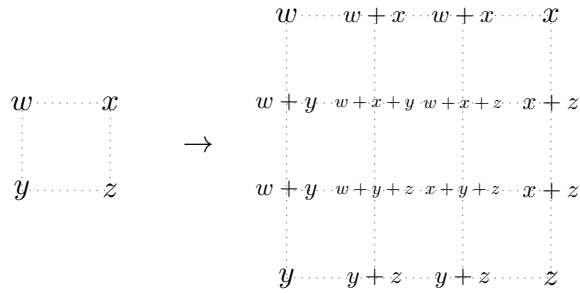
\begin{figure}[H]
    \centering
    \resizebox{8cm}{!}
    {
    \begin{tikzpicture}
    	\begin{scope}
        \node at (2,-.5) {$\rightarrow$};

        \draw[gray,dotted](0,0) -- (1,0);
        \draw[gray,dotted](0,0) -- (0,-1);
        \draw[gray,dotted](0,-1) -- (1,-1);
        \draw[gray,dotted](1,0) -- (1,-1);
        \node at (0,0) {$w$};
        \node at (1,0) {$x$};
        \node at (0,-1) {$y$};
        \node at (1,-1) {$z$};
        \draw[gray,dotted](3,1) -- (6,1);
        \draw[gray,dotted](3,1) -- (3,-2);
        \draw[gray,dotted](3,0) -- (6,0);
        \draw[gray,dotted](4,1) -- (4,-2);
        \draw[gray,dotted](3,-1) -- (6,-1);
        \draw[gray,dotted](5,1) -- (5,-2);
        \draw[gray,dotted](3,-2) -- (6,-2);
        \draw[gray,dotted](6,1) -- (6,-2);
        \node at (3,1) {$w$};
        \node[scale=.8] at (4,1) {$w+x$};
        \node[scale=.8] at (5,1) {$w+x$};
        \node at (6,1) {$x$};
        \node[scale=.8] at (3,0) {$w+y$};
        \node[scale=.6] at (4,0) {$w+x+y$};
        \node[scale=.6] at (5,0) {$w+x+z$};
        \node[scale=.8] at (6,0) {$x+z$};
        \node[scale=.8] at (3,-1) {$w+y$};
        \node[scale=.6] at (4,-1) {$w+y+z$};
        \node[scale=.6] at (5,-1) {$x+y+z$};
        \node[scale=.8] at (6,-1) {$x+z$};
        \node at (3,-2) {$y$};
        \node[scale=.8] at (4,-2) {$y+z$};
        \node[scale=.8] at (5,-2) {$y+z$};
        \node at (6,-2) {$z$};
    	\end{scope}
    \end{tikzpicture}
    }
    \caption{The substitution rule for $\sigma_6$}
    \label{fig:variant-6-r}
\end{figure}

\begin{figure}[b]
    \begin{center}
        \includegraphics[height=12cm, keepaspectratio]{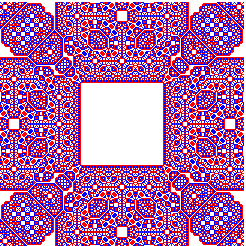}
    \caption{Colored representation of $\sigma_6^5(\square 1111)$ (for $p = 3$)}
    \label{fig:variant-6}
    \end{center}
\end{figure}

\subsection{\texorpdfstring{Substitution $\sigma_7$}{Substitution sigma7}}
\label{sec:variant-7}

The substitution $\sigma_7$ operates on isosceles right triangles
and transforms a tile into four tiles
as depicted in figure \ref{fig:variant-7-r}.

This substitution is related to the OEIS sequence \seqnum{A358871}.

See figure \ref{fig:variant-7} for an example of supertile.

\begin{figure}[H]
    \centering
    \resizebox{8cm}{!}
    {
    \begin{tikzpicture}
    	\begin{scope}
        \node at (2.5,0) {$\rightarrow$};
        
        \draw[gray,dotted] (0,0) -- (2,0) -- (1,1) -- cycle;
        \node at (0,0) {$x$};
        \node at (2,0) {$y$};
        \node at (1,1) {$z$};
        
        \draw[gray,dotted] (3,0) -- (5,0) -- (4,1) -- cycle;
        \draw[gray,dotted] (4,0) -- (4,1);
        \draw[gray,dotted] (3.5,.5) -- (4,0) -- (4.5,.5);
        \node at (3,0) {$x$};
        \node at (5,0) {$y$};
        \node at (4,1) {$z$};
        \node[scale=.7] at (4,0) {$x+y$};
        \node[scale=.7] at (3.5,.5) {$x+z$};
        \node[scale=.7] at (4.5,.5) {$y+z$};
    	\end{scope}
    \end{tikzpicture}
    }
    \caption{The substitution rule for $\sigma_7$}
    \label{fig:variant-7-r}
\end{figure}
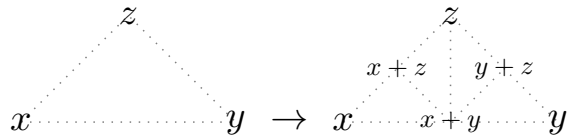

\begin{figure}[b]
    \begin{center}
        \includegraphics[height=7cm, keepaspectratio]{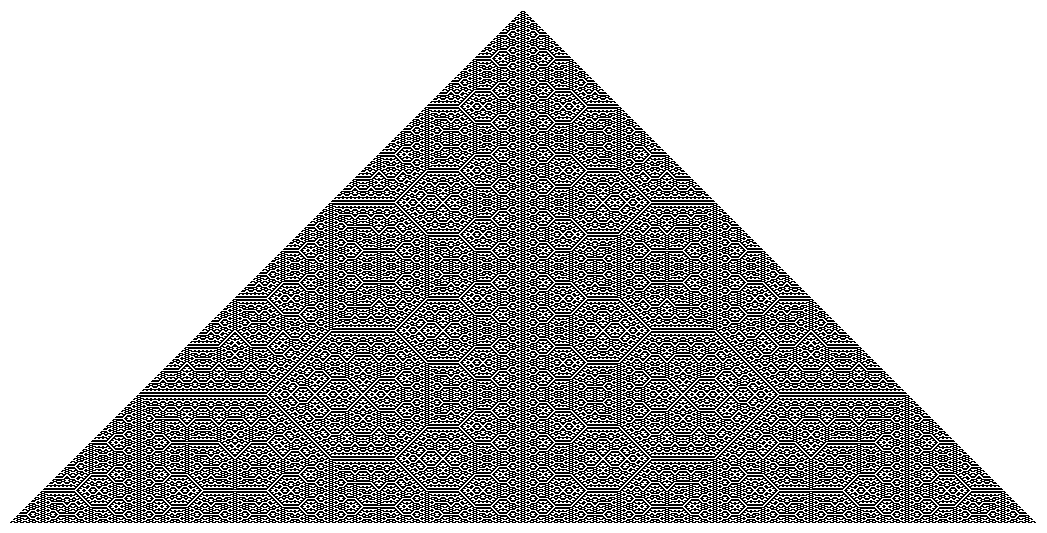}
    \caption{Colored representation of $\sigma_7^9(\bigtriangleup 110)$ (for $p = 2$)}
    \label{fig:variant-7}
    \end{center}
\end{figure}


\clearpage

\newpage

\begin{thebibliography}{1}

    \bibitem{automatic}
        Jean-Paul Allouche, Jeffrey Shallit,
        \quotes{Automatic Sequences: Theory, Applications, Generalizations},
        Cambridge University Press, 2003

	\bibitem{tilings}
        Dirk Frettl\"oh, Edmund Harriss, Franz G\"ahler,
        \quotes{Tilings encyclopedia},
        published electronically at \url{https://tilings.math.uni-bielefeld.de}

	\bibitem{lehmer}
        D. H. Lehmer,
        \quotes{On Stern’s diatomic series},
        Amer. Math. Monthly \textbf{36} (1929), 59–67

	\bibitem{oeis}
		The OEIS Foundation Inc. (2022),
		\quotes{The On-Line Encyclopedia of Integer Sequences},
		published electronically at \url{https://oeis.org}

	\bibitem{priebe}
		Natalie Priebe Frank,
		\quotes{A primer on substitution tilings of the Euclidean plane},
		\url{https://arxiv.org/abs/0705.1142} [math.DS]

	\bibitem{prunescu}
		Mihai Prunescu,
		\quotes{Self-similar carpets over finite fields},
		\url{https://arxiv.org/abs/0708.0899} [math.NT]

    \bibitem{multiscale}
        Yotam Smilansky, Yaar Solomon,
        \quotes{Multiscale Substitution Tilings},
        \url{https://arxiv.org/abs/2003.11735} [math.DS]

	\bibitem{wolfram}
		Stephen Wolfram,
		\quotes{A New Kind of Science},
		Wolfram Media, Inc., May 2002,
		published electronically at \url{https://www.wolframscience.com/nks/}

\end{thebibliography}
\end{document}